\documentclass{amsart}
\usepackage{amsmath}
\usepackage{paralist}
\usepackage{amsfonts}
\usepackage{amssymb}
\usepackage{amsthm}
\usepackage{amscd}
\usepackage{float}
\usepackage{tikz}
\usepackage{graphicx}
\usepackage[colorlinks=true]{hyperref}
\hypersetup{urlcolor=blue, citecolor=red}
\usepackage{hyperref}

\newtheorem{theorem}{Theorem}[section]
\newtheorem{corollary}[theorem]{Corollary}

\newtheorem{lemma}[theorem]{Lemma}
\newtheorem{proposition}[theorem]{Proposition}

\theoremstyle{definition}
\newtheorem{definition}[theorem]{Definition}

\newcommand{\R}{\mathbb{R}}
\newcommand{\Z}{\mathbb{Z}}
\newcommand{\N}{\mathbb{N}}
\newcommand{\C}{\mathbb{C}}

\newcommand{\cD}{\mathcal{D}}
\newcommand{\cE}{\mathcal{E}}
\newcommand{\cH}{\mathcal{H}}
\newcommand{\cB}{\mathcal{B}}
\newcommand{\cJ}{\mathcal{J}}

\makeatletter
\@namedef{subjclassname@2020}{%
  \textup{2020} Mathematics Subject Classification}
\makeatother

\begin{document}

\title[Quasilinear Maxwell equations in 2d]{On quasilinear Maxwell equations in two dimensions}
\author[Robert Schippa]{Robert Schippa*}
\email{robert.schippa@kit.edu}
\author{Roland Schnaubelt}
\email{schnaubelt@kit.edu}
\address{Department of Mathematics, Karlsruhe Institute of Technology, Englerstrasse 2, 76131 Karlsruhe, Germany}
\thanks{*Corresponding author}
\begin{abstract}
New sharp Strichartz estimates for the Maxwell system in two dimensions with rough permittivity and non-trivial charges are proved. 
We use the FBI transform to carry out the analysis in phase space. For this purpose, the Maxwell equations
are conjugated to a system of half-wave equations with rough coefficients. For this system, Strichartz estimates 
are proved similarly as in previous work by Tataru on scalar wave equations with rough coefficients. 
We use the estimates to improve the local well-posedness theory for quasilinear Maxwell equations in two dimensions.
\end{abstract}

\maketitle
\section{Introduction and main results}
The Maxwell equations in three spatial dimensions govern the propagation of electromagnetic fields. We refer to the physics literature with its many 
excellent accounts (e.g. \cite{FeynmanLeightonSands1964,LandauLifschitz1990}) for explaining the role of 
\emph{electric and magnetic fields} $(\cE,\cB): \R \times \R^3 \rightarrow \R^3 \times \R^3$ and 
\emph{displacement and magnetizing fields} $(\cD,\cH): \R \times \R^3 \rightarrow \R^3 \times \R^3$.
The \emph{electric charges} $\rho_e:\R \times \R^3 \rightarrow \R$ act as sources of the displacement field.
In the following space-time coordinates are denoted by 
$x=(x_0,x_1,\ldots,x_n) = (t,x^\prime) \in \R \times \R^n$ and the dual variables in Fourier space 
by $\xi = (\xi_0,\xi_1,\ldots,\xi_n) = (\tau,\xi^\prime) \in \R \times \R^n$.

In the absence of currents, the Maxwell system in media is given by
\begin{equation}
\label{eq:3dMaxwellEquations}
\left\{ \begin{aligned}
\partial_t \cD &= \nabla \times \cH, \quad \nabla \cdot \cD = \rho_e, \\
\partial_t \cB &= - \nabla \times \cE, \quad \nabla \cdot \cB = 0,\\
\cE(0,\cdot) &= \cE_0, \quad \cB(0,\cdot) = \cB_0.
\end{aligned} \right.
\end{equation}
These equations have to be supplemented with material laws linking, e.g., $\cE$ with $\cD$, and $\cH$ with $\cB$. 
We consider the \emph{constitutive relations} 
\begin{equation}
\label{eq:ConstitutiveRelation}
\begin{split}
\cD(x) &= \varepsilon(x) \cE(x), \quad \varepsilon: \R \times \R^3 \rightarrow \R^{3 \times 3}, \\
\cB(x) &= \mu(x) \cH(x), \quad \mu: \R \times \R^3 \rightarrow \R^{3 \times 3},
\end{split}
\end{equation}
which are linear, pointwise, and instantaneous. The coefficient $\varepsilon$ is referred to as 
\emph{permittivity} and $\mu$  as \emph{permeability}. We aim to describe dispersive properties
of electromagnetic fields in possibly anisotropic and inhomogeneous media, so that we allow for 
$x$-dependent and matrix-valued coefficients. In the following we consider $\mu \equiv 1$
for simplicity.   The relations \eqref{eq:ConstitutiveRelation} with 
$\mu(x) \equiv \mu_0 $ are frequently used to model phenomena in optics (cf.\ \cite{MoloneyNewell2004}).
We remark that our arguments extend to a variable permeability provided it satisfies 
the same ellipticity and regularity assumptions as the permittivity. 

We first focus on rough permittivity $\varepsilon$ with coefficients in $C^s$ for $0 < s \leq 2$
as an intermediate step to quasilinear Maxwell equations, where $\varepsilon = \varepsilon(\cE)$. 
A prominent example is  the \emph{Kerr nonlinearity}  given by 
\begin{equation}
\label{eq:KerrNonlinearity}
\varepsilon=\varepsilon(\cE)=(1+|\cE|^2).
\end{equation}

In this paper the Maxwell system in two spatial dimensions is considered, which can be derived taking $\cE$ 
to be perpendicular to media interfaces (cf.\ \cite{BDDGL,McDonald2019}). Actually, if $\varepsilon_{j3}=\varepsilon_{3j}=0$ 
for $j\in\{1,2\}$, if $\cE_0$, $\cB_0=\cH_0$, and $\rho_e$ in \eqref{eq:3dMaxwellEquations} only depend on $(x,y)\in\R^2$, and 
if the components $\cE_{03}$, $\cH_{01}$ and $\cH_{02}$ vanish, then the solutions $(\cE,\cH)$ to \eqref{eq:3dMaxwellEquations}
have the same properties. 
Hence, the resulting Maxwell system in two spatial dimensions is given by
\begin{equation}
\label{eq:2dMaxwellEquations}
\left\{ \begin{aligned}
\partial_t \cD &= \nabla_{\perp} \cH, \quad \nabla \cdot \cD = \rho_e, \\
\partial_t \cH &= - \nabla \times \cE, \\
 \cD(0,\cdot) &= \cD_0, \quad \cH(0,\cdot) = \cH_0.
\end{aligned} \right.
\end{equation}
In the above display we have $\cD,\cE: \R \times \R^2 \rightarrow \R^2$ and $\cH, \rho_e: \R \times \R^2 \rightarrow \R$,
and we set $\nabla_{\perp} = (\partial_2,-\partial_1)^t$. We suppose that $\varepsilon$ is a matrix-valued 
function $\varepsilon: \R \times \R^2 \rightarrow \R^{2\times 2}$ such that for some constants $\Lambda_1,\Lambda_2 > 0$ 
and all  $\xi^\prime \in \R^2$ and $x \in \R \times \R^2$ we have
\begin{equation}
\label{eq:EllipticityPermittivity}
\Lambda_1 |\xi^\prime|^2 \leq \sum_{i,j=1}^2 \varepsilon^{ij}(x) \xi_i \xi_j \leq \Lambda_2 |\xi^\prime|^2, \quad \varepsilon^{ij}(x) = \varepsilon^{ji}(x).
\end{equation}
Throughout the paper we shall use sum convention and sum over indices appearing twice, e.g.,
\begin{equation*}
\varepsilon^{ij}(x) \xi_i \xi_j = \sum_{i,j=1}^2 \varepsilon^{ij}(x) \xi_i \xi_j.
\end{equation*}
Using the matrix
\begin{equation}
\label{eq:RoughSymbolMaxwell2d}
P(x,D) = 
\begin{pmatrix}
\partial_t & 0 & - \partial_2 \\
0 & \partial_t & \partial_1 \\
-\partial_2(\varepsilon_{11} \cdot) + \partial_1 (\varepsilon_{21} \cdot) & \partial_1 (\varepsilon_{22} \cdot) - \partial_2(\varepsilon_{12} \cdot) & \partial_t
\end{pmatrix}
\end{equation}
with rough symbols, the PDEs in \eqref{eq:2dMaxwellEquations} can be rewritten as
\begin{equation}
\label{eq:MaxwellConcise}
P(x,D) (\cD_1,\cD_2,\cH) = 0, \quad \nabla \cdot \cD = \rho_e,
\end{equation}
where we set
\begin{equation*}
\varepsilon^{-1}(x) =  \big( \varepsilon_{ij}(x) \big)_{i,j=1,2} .
\end{equation*}
Let $P(x,D)(\cD_1,\cD_2,\cH)=(g_1,g_2,h)$, where $\cJ=-(g_1,g_2)$ is the electric current and $h$ has no physical meaning.
Then the electric charges in \eqref{eq:MaxwellConcise} are given by
\begin{equation} \label{eq:rho}
\rho_e(t)=\nabla\cdot \cD_0 + \int_0^t \nabla\cdot (g_1,g_2)ds.
\end{equation}

There is a large body of literature for Maxwell equations on smooth space-times, investigating more fundamental decay properties in higher dimensions (cf.\ \cite{MarzuolaMetcalfeTataruTohaneanu2010,MetcalfeTataruTohaneanu2017}). In these works, local energy decay
is proved, which implies Strichartz estimates.

Note that well-posedness of \eqref{eq:MaxwellConcise} with $\varepsilon = \varepsilon(\cE)$ in Sobolev spaces 
$H^s(\R^2)$ with $s > 2$ can be established by the energy method for hyperbolic systems (cf.\ \cite{BahouriCheminDanchin2011,BenzoniGavageSerre2007,Majda1984}). We revisit the argument in the last section and show how Strichartz estimates yield improvements. 
For a detailed account on local well-posedness results for the Maxwell system on spatial domains in $\R^3$ with nonlinear material laws, we refer to the PhD thesis \cite{SpitzPhdThesis2017} by M. Spitz. In \cite{Spitz2019} Spitz showed local well-posedness of the Maxwell equations with perfectly conducting boundary conditions in $H^3(G)$ for domains $G \subseteq \R^3$. This approach neglects dispersive effects. 

We are not aware of works on Maxwell equations with rough coefficients in the anisotropic case taking 
advantage of dispersion. In the isotropic case, i.e., $\varepsilon(x) = e(x) 1_{3 \times 3}$ with 
$e:\R \times \R^3 \rightarrow \R$ and $e \in C^2$, the second author proved global-in-time Strichartz estimates 
from local energy decay, jointly with P. D'Ancona \cite{DanconaSchnaubelt2021}. Local-in-time estimates 
for smooth scalar coefficients were treated in  \cite{DumasSueur2012}.

 We remark that in the constant-coefficient case, Liess \cite{Liess1991} (see also \cite{MandelSchippa2021}) showed decay estimates by Fourier analytic methods. Liess pointed out that in three spatial dimensions, the time-decay of $(\cD,\cH)$ is weaker if $\varepsilon$ has three different eigenvalues. In this case, the time-decay corresponds to the one of solutions to the two-dimensional wave equations. Hence, in three dimensions additional hypotheses are necessary to recover Strichartz estimates for the three-dimensional wave equation. This will be subject of future work.

 In the present paper, we  analyze the dispersive properties of the Maxwell system on $\R^2$ for rough 
 pointwise material laws in the anisotropic case. We prove Strichartz estimates by linking Maxwell equations
 to half-wave equations. The connection is established by analysis in phase space.

To relate our problem to the scalar wave equation, let $\cH$ satisfy  \eqref{eq:2dMaxwellEquations} and 
$\varepsilon$ be time-independent. Differentiating \eqref{eq:2dMaxwellEquations} in time, we infer
\begin{equation}
\label{eq:DerivedWaveEquation}
\partial^2_{t} \cH = \partial_2(\varepsilon_{11} \partial_2 \cH) - \partial_1 (\varepsilon_{12} \partial_2 \cH) - \partial_2 (\varepsilon_{21} \partial_1 \cH) + \partial_1 (\varepsilon_{22} \partial_1 \cH) =: \Delta_{\varepsilon^{-1}} \cH.
\end{equation}
Hence, $\cH$ solves a wave equation with rough coefficients, for which Strichartz estimates are known (cf.\ \cite{Tataru2000,Tataru2001,Tataru2002}). However,  we aim to show Strichartz estimates directly for the first-order system \eqref{eq:2dMaxwellEquations} because taking additional derivatives typically gives rise to loss in regularity.

 To put our results into perspective, we review Strichartz estimates for wave equations.
  For $u:\R \times \R^n \rightarrow \C$ let
 \begin{equation*}
 \| u \|_{L_{x_0}^p L_{x^\prime}^q} = \Big( \int_{\R} \Big( \int_{\R^n} |u(x_0,x^\prime)|^q dx^\prime \Big)^{p/q} dx_0 \Big)^{1/p}.
 \end{equation*}
 We will frequently omit to indicate space and time integration in $L^p L^q$-norms and set $L^p=L^pL^p$.
Keel and Tao \cite{KeelTao1998} established the sharp range for solutions to the wave equation 
with constant coefficients in Euclidean space. Let
\begin{equation*}
u:\R \times \R^n \rightarrow \R, \quad \Box u = 0, \quad u(0) = u_0, \quad u_t(0) = u_1.
\end{equation*}
Then the estimate
\begin{equation}
\label{eq:StrichartzEstimatesEuclideanSpace}
\Vert u \Vert_{L^p L^q} \lesssim \Vert u_0 \Vert_{H^\rho} + \Vert u_1 \Vert_{H^{\rho-1}}
\end{equation}
holds provided that
\begin{align*}
2 \leq p, q \leq \infty, \quad \rho = n \left( \frac{1}{2} - \frac{1}{q} \right) - \frac{1}{p}, \\
\frac{2}{p} + \frac{n-1}{q} \leq \frac{n-1}{2}, \quad (p,q,n) \neq (2,\infty,3).
\end{align*}
Tupels $(\rho,p,q,n)$ satisfying these relations will be called Strichartz pairs. If 
\begin{equation*}
\frac{2}{p} + \frac{n-1}{q} = \frac{n-1}{2},
\end{equation*}
 then $(\rho,p,q,n)$ will be referred to as sharp Strichartz pair.
Also note that the sharp Strichartz pairs imply the other ones by Sobolev's embedding.

An important question is how Strichartz estimates extend to variable metrics, i.e., $g=g(x)$. On the one hand, it is known (cf.\  \cite{Kapitanskii1989, Kapitanskii1989b,MockenhauptSeegerSogge1993, SeegerSoggeStein1991}) that solutions to the variable-coefficient wave equation
\begin{equation*}
\Box_{g(x)} u = \partial_i g^{ij} \partial_j u =0,
\end{equation*}
with  uniformly hyperbolic $g \in C^\infty$ have the same dispersive properties, at least locally in time, from which Strichartz estimates \eqref{eq:StrichartzEstimatesEuclideanSpace} follow. 

In view of non-smooth coefficients, Smith and Sogge \cite{SmithSogge1994} pointed out that for merely H\"older-continuous metrics $g \in C^{s}$ with  $0<s<2$ Strichartz estimates \eqref{eq:StrichartzEstimatesEuclideanSpace} fail. 
On the other hand, Smith \cite{Smith1998} showed that Strichartz estimates remain valid for $C^2$-coefficients in low dimensions.
In a series of papers \cite{Tataru2000,Tataru2001,Tataru2002} Tataru then recovered Strichartz estimates 
\eqref{eq:StrichartzEstimatesEuclideanSpace} for wave equations with $C^2$-coefficients for all dimensions; see also the 
preceding paper \cite{BahouriChemin1999} by Bahouri and Chemin  and the related work by  Klainerman \cite{Klainerman2001}. 
Tataru showed corresponding sharp estimates with additional derivative loss for $C^s$-coefficients as a 
minor variation of the $C^2$-case in \cite{Tataru2001}. Smith and Tataru proved sharpness in \cite{SmithTataru2002}.

We show the following theorems for the Maxwell system on $\R^2$. To state the results, let
\begin{align*}
(|D|^\alpha f) \widehat{\,} (\xi) = |\xi|^\alpha \hat{f}(\xi), \quad
(|D^\prime|^\alpha f ) \widehat{\,}(\xi) = |\xi^\prime|^\alpha \hat{f}(\xi).
\end{align*}
For the sake of simplicity, we suppose the fields to be smooth and understand the Strichartz estimates as a priori estimates. This makes no difference for the application to quasilinear equations.

\begin{theorem}
\label{thm:StrichartzEstimatesMaxwell2d}
Let $\varepsilon: \R \times \R^2 \rightarrow \R^{2 \times 2}$ be a matrix-valued function with coefficients in $C^2$ satisfying \eqref{eq:EllipticityPermittivity}. Let $u=(\cD_1,\cD_2,H): \R \times \R^2 \rightarrow \R^3$ with $\partial_1 \cD_1 + \partial_2 \cD_2 = \rho_e$, and $P$ as in \eqref{eq:RoughSymbolMaxwell2d}. Then, we find the following estimate to hold:
\begin{equation}
\label{eq:StrichartzEstimatesMaxwell2d}
\Vert |D|^{-\rho} u \Vert_{L^p L^q} \lesssim\kappa \Vert u \Vert_{L^2} +\kappa^{-1} \Vert P u \Vert_{L^2} + \| |D|^{-\frac{1}{2}} \rho_e \|_{L^2}
\end{equation} 
provided that the right hand-side is finite, $(\rho,p,q,2)$ is a Strichartz pair, and
\begin{equation*}
\| \partial^2_x \varepsilon \|_{L^\infty} \leq\kappa^4.
\end{equation*}
\end{theorem}
The space $C^s=C^s(\R^m)$ for $s\ge0$  is equipped with its standard norm.
For vectors the norms are given by $\Vert u \Vert_X = \Vert u_1 \Vert_X + \Vert u_2 \Vert_X + \Vert u_3 \Vert_X$. 
The additional parameter $\kappa$ is crucial to control the size of coefficients when dealing with quasilinear problems.

Note that \eqref{eq:StrichartzEstimatesMaxwell2d} implies the estimate with $\| |D'|^{-\frac{1}{2}} \rho_e \|_{L^2}$ on the right-hand side. Moreover, if $\| |D|^{-\frac{1}{2}} \rho_e \|_{L^2} \sim \| |D|^{\frac{1}{2}} \cD \|_{L^2}$, the estimate for 
the displacement field $\cD=(\cD_1,\cD_2)$ in
\eqref{eq:StrichartzEstimatesMaxwell2d} already follows from Sobolev's embedding. Hence, the Strichartz estimates are most 
relevant for  charges with the additional regularity  $\| |D|^{-\frac{1}{2}} \rho_e \|_{L^2} \lesssim \| \cD\|_{L^2}$, since in 
this case our results point out that the Maxwell system with $C^2$-coefficients exhibits the same dispersive properties 
as scalar wave equations. This includes the important charge-free case, of course. Furthermore, we note that already in the constant-coefficient case, the Strichartz estimates for the wave equation
\begin{equation*}
\| (\cD,\cH) \|_{L^p(0,T;L^q)} \lesssim \| (\cD_0, \cH_0) \|_{H^s} \quad \big( s= 2 \big( \frac{1}{2} - \frac{1}{q} \big) - \frac{1}{p} \big)
\end{equation*}
fail in general: Consider a Strichartz pair $(s,p,q,2)$, $\varepsilon = e 1_{2 \times 2}$, $e > 0$, $\cD_0 = \nabla \varphi_0$ for $\varphi_0 \in H^{s+1}$ and $\cH_0 = 0$: Then, $(\cD,\cH) = (\cD_0,0)$ is a stationary solution, but there are $\varphi_0 \in H^{s+1}$ such that $\cD_0 \notin L^q(\R^2)$ for $p \neq \infty$.

\medskip

Truncating the frequencies of $\varepsilon_{ij}$ appropriately and using
the above theorem, we can show Strichartz estimates for $C^s$-coefficients
if we  allow for a loss of derivatives compared to \eqref{eq:StrichartzEstimatesMaxwell2d} as in \cite{Tataru2001}.
\begin{theorem}
\label{thm:StrichartzEstimatesMaxwell2dCs}
Let $\varepsilon: \R \times \R^2 \rightarrow \R^{2 \times 2}$ be a matrix-valued function with coefficients in $C^s$, $0 \leq s < 2$, satisfying \eqref{eq:EllipticityPermittivity}. Let $u=(\cD_1,\cD_2,\cH): \R \times \R^2 \rightarrow \R^3$ with $\partial_1 \cD_1 + \partial_2 \cD_2 = \rho_e$. Then, we obtain the estimate
\begin{equation}
\label{eq:StrichartzEstimatesMaxwell2dCs}
\Vert |D|^{-\rho-\frac{\sigma}{2}} u \Vert_{L^p L^q} \lesssim\kappa \Vert u \Vert_{L^2} +\kappa^{-1} \Vert P u \Vert_{\dot{H}^{-\sigma}} + \| |D|^{-\frac{1}{2}- \frac{\sigma}{2}} \rho_e \|_{L^2}
\end{equation} 
provided that the right hand-side is finite, $(\rho,p,q,2)$ is a Strichartz pair,
\begin{equation*}
\sigma = \frac{2-s}{2+s}, \qquad\text{and}\qquad 
\| \varepsilon^{ij} \|_{\dot{C}^s} \leq\kappa^4.
\end{equation*}
\end{theorem}

The conclusion of Theorem \ref{thm:StrichartzEstimatesMaxwell2d} remains true with small modifications if the second 
derivatives of the coefficients belong to $L^1L^\infty$. The motivation for this setup is the quasilinear 
case $\varepsilon=\varepsilon(\cE)$, as discussed below.
For \eqref{eq:MaxwellConcise} we prove the following variant of Theorem \ref{thm:StrichartzEstimatesMaxwell2d}.
\begin{theorem}
\label{thm:StrichartzEstimatesL1LinfCoefficients}
Let $\varepsilon: \R \times \R^2 \rightarrow \R^{2 \times 2}$ be a matrix-valued function with Lipschitz coefficients, satisfying \eqref{eq:EllipticityPermittivity} and $\partial^2_x \varepsilon \in L^1 L^\infty$. Let $u=(\cD_1,\cD_2,\cH): \R \times \R^2 \rightarrow \R^3$ with $\partial_1 \cD_1 + \partial_2 \cD_2 = \rho_e$, and $(\rho,p,q,2)$ be a Strichartz pair. Then,
\begin{equation}
\label{eq:StrichartzEstimatesL1LinfCoefficients}
\begin{split}
\| |D^\prime|^{-\rho} u \|_{L^p(0,T;L^q)} &\lesssim\kappa^{\frac{1}{p}} \| u \|_{L^\infty L^2} +\kappa^{-\frac{1}{p^\prime}} \| P(x,D) u \|_{L^1 L^2} \\
 &\quad + T^{\frac{1}{2}}\| |D'|^{-\frac{1}{2}} \rho_e(0) \|_{L^2(\R^2)} + T^{\frac{1}{2}}\| |D'|^{-\frac{1}{2}} \partial_t \rho_e \|_{L^1 L^2},
 \end{split}
\end{equation}
whenever the right-hand side is finite, provided that $\kappa \geq 1$ and
\begin{equation*}
T \| \partial^2_x \varepsilon \|_{L^1 L^\infty} \leq\kappa^2.
\end{equation*}
\end{theorem}
Observe that $\rho_e(0)= \nabla\cdot \cD_0$ and  $\partial_t \rho_e=\partial_1 (Pu)_1+ \partial_2 (Pu)_2$
due to \eqref{eq:rho}.  Compared to \eqref{eq:StrichartzEstimatesMaxwell2d}, above one thus takes the $L^2_{x'}$ norm
into sum and integral apearing in  \eqref{eq:rho}. This happens when passing to $L^\infty L^2$- and $L^1 L^2$-norms
 on the right-hand side of \eqref{eq:StrichartzEstimatesL1LinfCoefficients}, 
see Paragraph \ref{subsubsection:ReductionCubeLargeFrequencies}. These norms are better suited for the application to  
quasilinear problems than the  $L^2$-norms appearing on the right-hand side of \eqref{eq:StrichartzEstimatesMaxwell2d}.

But for these applications one still needs a version for coefficients with less regularity. 
To state it, let $( S_\lambda)_{\lambda \in 2^{\mathbb{Z}}}$ denote a homogeneous Littlewood--Paley decomposition in space-time, and $(S_\lambda')_{\lambda \in 2^{\mathbb{Z}}}$ one in the spatial variables only. To avoid problems when summing norms on Littlewood--Paley blocks, the regularity of solutions is measured in homogeneous Besov-type spaces $\dot{B}^{pqr}_s$ with norms
\begin{equation*}
\| u \|^r_{\dot{B}^{pqr}_s} = \sum_\lambda \lambda^{rs} \| S_\lambda u \|^r_{L^p L^q},
\end{equation*}
and the usual modification for $r = \infty$. For the coefficients,
following \cite{Taylor1991,Tataru2002} we use the microlocalizable scale of spaces $\mathcal{X}^s$ given by
\begin{equation*}
\| v \|_{\mathcal{X}^s} = \sup\nolimits_{\lambda} \lambda^s \| S_\lambda v \|_{L^1 L^\infty}.
\end{equation*}
The $\mathcal{X}^s$-regularity is an adequate substitute for the $C^s$-regularity of the coefficients in our setting.
 For these regularities we prove the following Strichartz estimate.
\begin{theorem}
\label{thm:MaxwellStrichartzEstimatesBesovRegularity}
Let $\varepsilon \in \mathcal{X}^s$, $0 \leq s < 2$, and $u$, $(\rho,p,q,2)$ and $\sigma$ be as in the assumptions of Theorem \ref{thm:StrichartzEstimatesMaxwell2dCs}. Then, we find the following estimate to hold:
\begin{equation}
\label{eq:GeneralBesovRegularityEstimates}
\begin{split}
\| |D|^{-\rho-\frac{\sigma}{p}} u \|_{\dot{B}^{pq \infty}_0} &\lesssim\kappa^{\frac{1}{p}} \| u \|_{L^\infty L^2} +\kappa^{-\frac{1}{p^\prime}} \| |D|^{-\sigma} P u \|_{L^1 L^2} \\
&\quad + T^{\frac{1}{2}}\| |D|^{-\frac{1}{2}-\frac{\sigma}{p}} \rho_e \|_{L^\infty L^2} + T^{\frac{1}{2}}\| |D|^{-\frac{1}{2}-\frac{\sigma}{p}} \partial_t \rho_e \|_{L^1 L^2}
\end{split}
\end{equation}
for all $u$ compactly supported in $[0,T]$ and $\kappa,T$ satisfying
\begin{equation*}
T^s \| \varepsilon \|^2_{\mathcal{X}^s} \lesssim\kappa^{2+s}.
\end{equation*}
\end{theorem}
Estimates  involving the norm of $P(x,D)u$ in $L^{p'}L^{q'}$ follow as in 
\cite{Tataru2001,Tataru2002}. We sketch the proof for the next result with $C^2$ coefficients.
\begin{theorem}
\label{thm:MaxwellInhomogeneousStrichartz}
Let $\varepsilon \in C^2$ and $(\rho,p,q,2)$ be a Strichartz pair. Then, we have
\begin{equation}
\label{eq:InhomogeneousStrichartzEstimatesC2}
\| |D|^{-\rho} u \|_{L^p L^q} \lesssim\kappa \| u \|_{L^2} +\kappa^{-1} \| f_1 \|_{L^2} + \| |D|^\rho f_2 \|_{L^{p^\prime} L^{q^\prime}} + \| |D|^{-\frac{1}{2}} \rho_e \|_{L^2}
\end{equation}
whenever
\begin{equation*}
P(x,D) u = f_1 + f_2\qquad\text{and}\qquad 
\| \partial_x^2 \varepsilon \|_{\infty} \leq\kappa^4.
\end{equation*}
\end{theorem}
One can establish versions for $C^s$-coefficients with $0\leq s < 2$, cf.\ \cite[Theorem~4]{Tataru2001}, and for 
$\mathcal{X}^s$-coefficients, cf.\ \cite[Corollary~1.6]{Tataru2002}, using arguments from these papers.

On a finite time interval,  the homogeneous problem \eqref{eq:2dMaxwellEquations} can easily be treated 
by the above results if $\partial_t \varepsilon\in L^1L^\infty$, since a standard energy estimate yields
\begin{equation}\label{eq:energy0}
\|u\|_{L^\infty L^2}\lesssim e^{c\|\partial_t \varepsilon\|_{L^1L^\infty}}\|u(0)\|_{L^2(\mathbb{R}^2)}.
\end{equation}

As in \cite[Corollary~5]{Tataru2001}, we can also prove estimates with two different Strichartz pairs. We provide such a result on a finite time interval $(0,T)$, fixing $\delta$ and $T$, and we further suppose that the solutions are charge-free and $\varepsilon$ is isotropic. These limitations stem from the use of duality in the proof. Perhaps the latter assumptions can be weakened by modifying the proof of Theorem \ref{thm:MaxwellInhomogeneousStrichartz} to treat inhomogeneous terms $\langle D' \rangle^{\tilde{\rho}} f \in L^{\tilde{p}'} L^{\tilde{q}'}$. This is not pursued presently.

\begin{corollary}
\label{cor:InhomogeneousStrichartzC1Coefficients}
Let $\varepsilon = eI_{2\times2} \in C^s$, $1 \leq s \leq 2$, $P(x,D) u = f$, $\partial_1 u_1 + \partial_2 u_2 = 0$, and $(\rho,p,q,2)$, $(\tilde{\rho},\tilde{p},\tilde{q},2)$ be Strichartz pairs. Then, we find the following estimate to hold:
\begin{equation}
\label{eq:InhomogeneousStrichartzC1Coefficients}
\| \langle D' \rangle^{-\rho-\frac{\sigma}{2}} u \|_{L^p(0,T;L^q)} \lesssim_{T,\kappa} \| u(0) \|_{L^2} + \| \langle D' \rangle^{\tilde{\rho} + \frac{\sigma}{2}} f \|_{L^{\tilde{p}'}(0,T;L^{\tilde{q}'})}.
\end{equation}
\end{corollary}
The proof relies on a now standard application of the Christ--Kiselev lemma, \cite{ChristKiselev2001}. However, as the time-dependent generators of the Maxwell system are not self-adjoint in $L^2$, additional considerations are necessary.

To study quasilinear equations, we use a similar result in the context of
Theorem~\ref{thm:MaxwellStrichartzEstimatesBesovRegularity}, cf.\ Corollary~1.7 in \cite{Tataru2002}. 
The quantity $\|\partial_x\varepsilon\|_{L^2 L^\infty}$ can be controlled for coefficients 
$\varepsilon=\varepsilon(\cE)$ arising in a bootstrap argument. 
\begin{corollary}
\label{cor:StrichartzEstimatesL2LinfCoefficients}
Assume that  $\|\partial_x \varepsilon\|_{ L^2 L^\infty} \lesssim 1$ and for some $\tilde{s} \in [1,2)$, suppose that $\| \varepsilon \|_{\mathcal{X}^{\tilde{s}}} \lesssim 1$. Let $(\rho,p,q,2)$ be a Strichartz pair. Then the solution $u$ to
\begin{equation*}
\left\{
\begin{aligned}
P(x,D) u &= f, \qquad \partial_1 u_1 + \partial_2 u_2 = \rho_e, \\
u(0) &= u_0
\end{aligned} \right.
\end{equation*}
satisfies
\begin{equation*}
\begin{split}
\| \langle D' \rangle^{-\alpha} u \|_{L^p(0,T;L^q)} &\lesssim_T \| u_0 \|_{L^2(\mathbb{R}^2)} + \| f \|_{L^1(0,T; L^2)} \\
&\qquad + \| \langle D' \rangle^{-\frac{1}{2}-\frac{\sigma}{p}} \rho_e(0) \|_{L^2} + \|\langle D' \rangle^{-\frac{1}{2}-\frac{\sigma}{p}} \partial_t \rho_e \|_{L^1(0,T; L^2)}
\end{split}
\end{equation*}
for $\alpha > \rho + \frac{\sigma}{p}$ and $\sigma = \sigma(\tilde{s})= \frac{2-\tilde{s}}{2+\tilde{s}}$.
\end{corollary}

For the proof of Theorem \ref{thm:StrichartzEstimatesMaxwell2d} we shall make use of the FBI transform, conjugating the problem to phase space. Prior to the phase space analysis, we make the following reductions already pointed out by Tataru \cite{Tataru2001,Tataru2002} and also in earlier work \cite{BahouriChemin1999} in the context of wave equations with rough coefficients. The first reduction consists in a direct estimate of low frequencies $\lesssim 1$, so that we only have to estimate high frequencies. One can then localize in space-time and reduce to solutions supported in the unit cube. Moreover, by Littlewood-Paley theory and commutator estimates, it suffices to establish an estimate for dyadic (high) frequencies $\lambda \in 2^{\mathbb{N}}$ and with coefficients of the permittivity truncated at frequencies $< \lambda^{1/2}$. Finally,  we reduce to the case of frequencies $|\xi_0|\lesssim |(\xi_1,\xi_2)|$ by estimating the contributions from the region $|\xi_0|\gg |(\xi_1,\xi_2)|$ off the light cone using properties of the FBI transform. 

After these reductions, we can diagonalize \eqref{eq:2dMaxwellEquations} to two non-degenerate half-wave equations and one degenerate half-wave equation. One can bound the appropriate norm of the degenerate component by $\|u\|_2$ and the divergence of $(\cD_1,\cD_2)$, again making use of the FBI transform. The nondegenerate components are discussed below. During the conjugation procedure, we encounter pseudo-differential operators with rough symbols. For these we give expansions of composites, which resemble the smooth case. However, we have to be careful with $L^2$-estimates, since we cannot spare several derivatives in the spatial variables. In our opinion the diagonalization procedure is the main novelty of the paper. It quantifies the hyperbolic degeneracy of the Maxwell system through the electric charges and allows to recover wave Strichartz estimates in the charge-free case. The method of proof possibly extends to other first-order systems like Dirac equations with variable coefficients (cf.\ \cite{CacciafestadeSuzzoni2019}).

The second key ingredient in the proof of Theorems \ref{thm:StrichartzEstimatesMaxwell2d} and \ref{thm:StrichartzEstimatesL1LinfCoefficients} is the following result for the half-wave equation, which we prove by varying Tataru's arguments and using his results for the wave equation \cite{Tataru2001,Tataru2002}. We write 
\begin{equation*}
\tilde{\varepsilon}(x) = (\tilde{\varepsilon}^{ij}(x))=
\begin{pmatrix}
\varepsilon_{22}(x) & -\varepsilon_{12}(x) \\
-\varepsilon_{21}(x) & \varepsilon_{11}(x)
\end{pmatrix}
\end{equation*}
which is $\varepsilon$ up to determinant, and let $\tilde\varepsilon_{\lambda^\frac12}$ denote these coefficients with Fourier support truncated to $\{|\xi| \leq \lambda^{\frac{1}{2}} \}$. 
\begin{proposition}
\label{prop:HalfWaveEstimates}
Let $\lambda \in 2^{\mathbb{N}_0},$ $\lambda \gg 1,$ and  $n \geq 2$. Assume $\varepsilon = \varepsilon^{ij}(x)$ satisfies $\varepsilon^{ij} \in C^2$, $\| \partial^2_x \varepsilon \|_{L^\infty} \leq 1$, and \eqref{eq:EllipticityPermittivity}. Let $Q(x,D)$ denote the pseudo-differential operator with symbol 
\begin{equation*}
q(x,\xi) = - \xi_0 + \big( \tilde\varepsilon^{ij}_{\lambda^\frac12} (x) \xi_i \xi_j \big)^{1/2}.
\end{equation*}
Moreever, let $u$ decay rapidly outside the unit cube and $(\rho,p,q,n)$ be a Strichartz pair. 
Then, we find the estimates 
\begin{equation}
\label{eq:StrichartzEstimatesHalfWaveEquationC2Coefficients}
\lambda^{-\rho} \Vert S_\lambda u \Vert_{L^p L^q} \lesssim \Vert S_\lambda u \Vert_{L^2} + \Vert Q(x,D) S_\lambda u \Vert_{L^2}
\end{equation}
to hold with an implicit constant uniform in $\lambda$. For Lipschitz coefficients $\varepsilon^{ij}$ 
with $\| \partial^2_x \varepsilon \|_{L^1 L^\infty} \leq 1$, we obtain
\begin{equation}
\label{eq:StrichartzEstimatesHalfWaveEquationL1LinfCoefficients}
\lambda^{-\rho} \Vert S_\lambda u \Vert_{L^p L^q} \lesssim \Vert S_\lambda u \Vert_{L^\infty L^2} + \Vert Q(x,D) S_\lambda u \Vert_{ L^2}.
\end{equation}
\end{proposition}

At last, we apply the Strichartz estimates to the local well-posedness theory of the system
\begin{equation}
\label{eq:QuasilinearMaxwellEquations}
\left\{ \begin{aligned}
\partial_t u_1 &= \partial_2 u_3, \qquad u(0) = u_0 \in H^s(\R^2;\R)^3, \\
\partial_t u_2 &= - \partial_1 u_3, \qquad \partial_1 u_1 + \partial_2 u_2 = 0, \\
\partial_t u_3 &= \partial_2 (\varepsilon^{-1}(u) u_1) - \partial_1 (\varepsilon^{-1}(u) u_2),
\end{aligned} \right.
\end{equation}
where $\varepsilon^{-1}(u) = \psi(|u_1|^2+|u_2|^2)$, $\psi:\R_{\geq 0} \to \R_{\geq 0}$ is smooth, monotone increasing, and  $\psi(0)=1$. Observe that the Kerr nonlinearity as given in \eqref{eq:KerrNonlinearity} is covered. 
One can apply our methods also to matrix-valued $\varepsilon(\cE)$ under symmetry constraints providing energy bounds. 
We remark that  one can transform \eqref{eq:QuasilinearMaxwellEquations} into a system of wave equations taking second derivatives in time. Although it might be possible in principle to apply the previously known Strichartz estimates for wave equations, this approach surely finds its limitations when anistropic material laws are considered.

By local well-posedness, we mean existence, uniqueness, and continuous dependence of the solutions in $H^s$ 
locally in time. We refer to the recent lecture notes by Ifrim and Tataru \cite{IfrimTataru2020} for explaining 
the notion of local well-posedness for quasilinear equations in detail. Energy methods, neglecting 
dispersive properties of \eqref{eq:QuasilinearMaxwellEquations}, give local well-posedness for $s>2$ as noted 
above. For the scalar quasilinear  wave equation on $\R^2$, Tataru \cite{Tataru2002} proved local well-posedness
in $H^s$ for $s>11/6$. We establish the analogous result for the  Maxwell system on $\R^2$.  
\begin{theorem}
\label{thm:LocalWellposednessQuasilinearMaxwell}
\eqref{eq:QuasilinearMaxwellEquations} is locally well-posed for $s>11/6$.
\end{theorem}

Finally, we show that the derivative loss for Strichartz estimates is sharp for permittivity coefficients 
in $C^s$ for $1 \leq s \leq 2$. For this purpose, we elaborate on the connection with wave equations with 
rough coefficients described in \eqref{eq:DerivedWaveEquation} and use the time-independent counterexamples of 
Smith and Tataru \cite{SmithTataru2002}. Hence, although the improvement in 
Theorem \ref{thm:LocalWellposednessQuasilinearMaxwell} seems little over the energy method, 
it appears to be the limit of proving well-posedness in $H^s$ with Strichartz estimates for general 
coefficients $\partial_x \varepsilon \in L^1 L^\infty$. It could still be possible to make further improvements
by the arguments of Smith and Tataru \cite{SmithTataru2002}, see also Klainerman--Rodnianski \cite{KlainermanRodnianski2005}, in the context of quasilinear wave equations. In these works was exploited that the metrical tensor solves a quasilinear wave equation itself.
We note that on $\R^3$ and for isotropic material laws as above, our methods should give an improvement
of the regularity level in the local wellposedness theory by $\frac13$ from $s>\frac52$ to $s>\frac{13}6$,
in accordance with \cite{Tataru2002}.

\medskip

\emph{Outline of the paper.} In Section \ref{section:PDOFBITransform} we recall properties of pseudo-differential operators with rough symbols and of the FBI transform. In Section \ref{section:ReductionHalfWaveEstimates} we first localize the functions in space and frequency and then carry out the conjugation procedure, reducing Theorems \ref{thm:StrichartzEstimatesMaxwell2d} and \ref{thm:StrichartzEstimatesL1LinfCoefficients} to dyadic estimates for the half-wave equation. In Section \ref{section:ProofHalfWaveEstimates} we prove these crucial dyadic estimates stated in Proposition \ref{prop:HalfWaveEstimates}, following the arguments in \cite{Tataru2001,Tataru2002}. In Section \ref{section:Consequences} we treat weaker Strichartz estimates, assuming less regularity of the coefficients, 
as formulated in Theorem \ref{thm:StrichartzEstimatesMaxwell2dCs}, Theorem \ref{thm:MaxwellStrichartzEstimatesBesovRegularity}, and Corollary \ref{cor:StrichartzEstimatesL2LinfCoefficients}. 
Here we also sketch the proof of  Theorem~\ref{thm:MaxwellInhomogeneousStrichartz}. In Section \ref{section:QuasilinearMaxwell} we improve the local well-posedness for quasilinear Maxwell equations as stated in Theorem \ref{thm:LocalWellposednessQuasilinearMaxwell}. In Section \ref{section:Sharpness} we elaborate on the link to wave equations and show sharpness of the derivative loss for permittivity coefficients in $C^s$ for $1 \leq s \leq 2$.



\section{Pseudo-differential operators with rough symbols and properties of the FBI transform}
\label{section:PDOFBITransform}
This section is devoted to preliminaries on the encountered pseudo-\-dif\-ferential operators and the FBI transform. As in Tataru's works \cite{Tataru2000,Tataru2001,Tataru2002}, we make use of the latter to find suitable conjugates of pseudo-differential operators in phase space. In these references, the key application was to use the conjugate of the rough wave operator as weight in phase space to prove Strichartz estimates. In the present paper, rough symbols additionally come up when conjugating Maxwell equations to a diagonal system of scalar half-wave equations. Thus, we have to analyze the $L^2$-boundedness and compositions of rough symbols. Taylor's monograph \cite{Taylor1991} contains many results for symbols which are not smooth in the spatial variables. Here we state the results in the form needed in the present context, and we shall revisit some of the arguments as these will be used in later sections.

\subsection{The FBI transform}
 We first recall basic facts about the FBI transform (cf.\ \cite{Delort1992,Tataru2000}). For $\lambda \in 2^{\mathbb{Z}}$, the FBI transform of an integrable function $f: \R^m \rightarrow \C$ is defined by
\begin{equation*}
\begin{split}
T_\lambda f(z) &= C_m \lambda^{\frac{3m}{4}} \int_{\R^m} e^{-\frac{\lambda}{2}(z-y)^2} f(y) dy, \quad z = x-i\xi \in T^* \R^m \equiv \R^{2m}, \\ 
 C_m &= 2^{-\frac{m}{2}} \pi^{- \frac{3m}{4}}.
\end{split}
\end{equation*}
We have the isometric mapping property $T_\lambda: L^2(\R^m) \rightarrow L^2_{\Phi}(T^* \R^m)$, where $\Phi(z) = e^{- \lambda \xi^2}$. It is natural to write $z=x-i\xi$ since $T_\lambda f$ is in fact holomorphic.
The connection with the Fourier transform is emphasized by writing
\begin{equation*}
T_\lambda f(z) = C_m \lambda^{\frac{3m}{4}} e^{\frac{\lambda}{2} \xi^2} \int_{\R^m} e^{- \frac{\lambda}{2}(x-y)^2 } e^{i \lambda \xi.(x-y)} f(y) dy.
\end{equation*}
An inversion formula for the FBI transform is given by the adjoint in $L^2_\Phi$:
\begin{equation*}
T_\lambda^* F(y) = C_m \lambda^{\frac{3m}{4}} \int_{\R^{2m}} e^{- \frac{\lambda}{2}(\overline{z}-y)^2} \Phi(z) F(z) dx d\xi.
\end{equation*}
We recall the following identities for conjugating symbols with the FBI transform.
In the following we consider symbols $a(x,\xi) \in C^s_x C^\infty_c$ compactly supported in $\xi$. More specifically, we shall assume
\begin{equation*}
a(x,\xi) = 0 \text{ for } \xi \notin B(0,2).
\end{equation*}
Let $a_\lambda(x,\xi) = a(x,\xi / \lambda)$ denote the scaled symbol supported at frequencies $\lesssim \lambda$, and 
$A_\lambda=A_\lambda(x,D)$ be the corresponding pseudo-differential operator. 
 
As in \cite{Tataru2000,Tataru2001}, the idea is to find an `approximate conjugate' $\tilde{A}_\lambda$ of $A_\lambda$ such that
\begin{equation*}
T_\lambda A_\lambda(y,D) \approx \tilde{A}_\lambda T_\lambda.
\end{equation*}
We record the basic identities
\begin{align*}
T_\lambda (yf)(z) &= (x + \frac{1}{- i\lambda}(\partial_\xi - \lambda \xi)) T_\lambda f,\\
T_\lambda (\frac{D}{\lambda} f)(z) &= (\xi + \frac{1}{\lambda}(\frac{1}{i} \partial_x - \lambda \xi)) T_\lambda f,
\end{align*}
yielding the formal asymptotics
\begin{equation*}
T_\lambda A_\lambda(x,D) \approx \sum_{\alpha, \beta} (\partial_\xi - \lambda \xi)^\alpha \frac{\partial_x^\alpha \partial_\xi^\beta a(x,\xi)}{|\alpha| ! |\beta| ! (-i\lambda)^{|\alpha|} \lambda^{|\beta|}} ( \frac{1}{i} \partial_x - \lambda \xi)^\beta T_\lambda.
\end{equation*}
We recall error bounds for the truncated approximations
\begin{equation*}
\tilde{a}^s_\lambda = \sum_{|\alpha| + |\beta| < s} (\partial_\xi - \lambda \xi)^\alpha \frac{\partial_x^\alpha \partial_\xi^\beta a(x,\xi)}{|\alpha| ! |\beta| ! (-i\lambda)^{|\alpha|} \lambda^{|\beta|}} ( \frac{1}{i} \partial_x - \lambda \xi)^\beta.
\end{equation*}
For $s \leq 1$, we have
\begin{equation*}
\tilde{a}^s_\lambda = a,
\end{equation*}
and for $1 < s \leq 2$,
\begin{equation}\label{eq:a^2}
\tilde{a}^s_\lambda = a + \frac{1}{-i \lambda} a_x(\partial_\xi - \lambda \xi) + \frac{1}{\lambda} a_\xi (\frac{1}{i} \partial_x - \lambda \xi)
= a+  \frac{2}{\lambda}(\overline{\partial} a)(\partial -i\lambda \xi),
\end{equation}
where $\partial=\frac12(\partial_x +i\partial_\xi)$ and $\overline{\partial}=\frac12(\partial_x -i\partial_\xi)$.
We will not need higher approximations because for coefficients in $C^2$ the Strichartz estimates for the Euclidean (half-)wave equation hold true, which are known to be optimal (cf.\ \cite{KeelTao1998}).
To prove Theorem \ref{thm:StrichartzEstimatesMaxwell2d}, it will be enough to use  the first-order approximation from the previous display.

Consider the remainder
\begin{equation}
\label{eq:Remainder}
R^s_{\lambda,a} = T_\lambda A_\lambda - \tilde{a}^s_\lambda T_\lambda.
\end{equation}
In \cite{Tataru2000,Tataru2001} the following approximation result was proved.
\begin{theorem}[{\cite[Theorem~5,~p.~393]{Tataru2001}}]
\label{thm:ApproximationTheorem}
Suppose that $a \in C^s_x C^\infty_c$. Then,
\begin{align*}
\Vert R^s_{\lambda,a} \Vert_{L^2 \rightarrow L^2_{\Phi}} &\lesssim \lambda^{-s/2}, \\
\Vert (\partial_\xi - \lambda) R^s_{\lambda,a} \Vert_{L^2 \rightarrow L^2_{\Phi}} &\lesssim \lambda^{1/2-s/2}. 
\end{align*}
\end{theorem}

To prove our main results, we use the following multiplier theorem for $T_\lambda$.
\begin{proposition}
\label{prop:MultiplierProposition}
Let $1 \leq p,q \leq \infty$, $a \in C^s_x C^\infty_c(\R^m \times \R^m)$ with $a(x,\xi) = 0$ for $\xi \notin B(0,2)$, and
\begin{equation*}
\sup_{x \in \R^m}  \sum_{0 \leq |\alpha| \leq m + 1} \Vert D^\alpha_\xi a(x,\cdot) \Vert_{L^1_{\xi}}  \leq C.
\end{equation*}
Then, we find the following estimate to hold:
\begin{equation*}
\Vert T_\lambda^* a(x,\xi) T_\lambda f \Vert_{L_{x_0}^p L_{x^\prime}^q} \lesssim C \Vert f \Vert_{L_{x_0}^p L_{x^\prime}^q}. 
\end{equation*}
\end{proposition}
\begin{proof}
We start with the special case $b \in L^\infty_x(\R^m)$, $c \in C^\infty_c(\R^m)$, $a(x,\xi) = b(x) c(\xi)$. It can be treated by a straight-forward kernel estimate. We first compute
\begin{equation}
\label{eq:KernelTensorStructure}
\begin{split}
&T_\lambda^* (b(x) c(\xi)) T_\lambda f(y) \\
&= C_m^2 \lambda^{\frac{3m}{2}} \int_{\R^{2m}} e^{-\frac{\lambda}{2}(\overline{z}-y)^2} e^{-\lambda \xi^2} b(x) c(\xi) \int_{\R^m} e^{- \frac{\lambda}{2}(z-y^\prime)^2} f(y^\prime) dy^\prime dx d\xi \\
&= C_m^2 \lambda^{\frac{3m}{2}} \int_{\R^m} \int_{\R^{2m}} e^{-\frac{\lambda}{2}(x+i\xi-y)^2} e^{- \frac{\lambda}{2}(x-i\xi-y^\prime)^2} e^{- \lambda \xi^2} b(x) c(\xi) dx d\xi f(y^\prime) dy^\prime \\
&= C_m^2 \lambda^{\frac{3m}{2}} \int_{\R^m} \int_{\R^m} e^{- \frac{\lambda}{2}((x-y)^2 + (x-y^\prime)^2)} b(x) dx \int_{\R^m} e^{i\lambda \xi.(y-y^\prime)} c(\xi) d\xi f(y^\prime) dy^\prime \\
&=: C_m^2 \lambda^{\frac{3m}{2}} \int_{\R^m} B_\lambda(y,y^\prime) \hat{c}(\lambda(y-y^\prime)) f(y^\prime) dy^\prime.
\end{split}
\end{equation}
Note that
\begin{equation*}
|B_\lambda(y,y^\prime)| \leq \Vert b \Vert_{L^\infty} \int_{\R^m} e^{-\lambda x^2} dx \lesssim \lambda^{- \frac{m}{2}} \Vert b \Vert_{L^\infty},
\end{equation*}
and further,
\begin{equation*}
|\hat{c}(\lambda(y-y^\prime))| \leq C_N (1+\lambda |y-y^\prime|)^{-N} \quad \text{ for any } N \in \mathbb{N}.
\end{equation*}
Write $y=(y_1,y_r) \in \R \times \R^{m-1}$ and likewise for $y^\prime$. Two successive applications of Young's inequality imply
\begin{equation*}
\begin{split}
&\quad \lambda^m \Vert \int (1+\lambda|y-y^\prime|)^{-N} f(y^\prime) dy^\prime \Vert_{L_{x_0}^p L_{x^\prime}^q} \\
&\lesssim \lambda^m \Vert \int (1+\lambda|y_1-y_1^\prime|)^{-\frac{N}{2}} (1+\lambda|y_r - y_r^\prime|)^{-\frac{N}{2}} f(y_1^\prime,y_r^\prime) dy^\prime \Vert_{L_{x_0}^p L_{x^\prime}^q} \\
&\lesssim \lambda \int (1+\lambda|y_1 - y_1^\prime|)^{- \frac{N}{2}} \Vert f(y_1^\prime,\cdot)  \Vert_{L^q_{y_r^\prime}} dy_1^{\prime} \lesssim \Vert f \Vert_{L_{x_0}^p L_{x^\prime}^q}
\end{split}
\end{equation*}
for large $N$. The two estimates for $B_\lambda$ and $\hat{c}$ yield
\begin{equation*}
\Vert T_\lambda^* (b(x)c(\xi)) T_\lambda f \Vert_{L_{x_0}^p L_{x^\prime}^q}
    \lesssim \Vert b \Vert_{L^\infty} \Vert c \Vert_{C^N} \Vert f \Vert_{L_{x_0}^p L_{x^\prime}^q}.
\end{equation*}

We turn to the case of general $a(x,\xi)$ according to the assumptions. Let $\beta \in C^\infty_c(\R^m)$ with $\beta \equiv 1$ for $\{|\xi| \leq 2\}$ and $\text{supp}(\beta) \subseteq B(0,3)$. The main reduction is an expansion into the rapidly converging Fourier series
\begin{equation*}
a(x,\xi) = \beta(\xi) \sum_{k \in \Z^m} e^{ik \xi} \hat{a}_k(x), \quad \hat{a}_k(x) = \int_{[-\pi,\pi]^n} e^{-ik\xi} a(x,\xi) d\xi,
  \quad \beta_k(\xi)=e^{ik \xi} \beta(\xi)
\end{equation*}
 (cf.\ \cite{Taylor1991}). Hence,
\begin{equation*}
\Vert T_\lambda^* a(x,\xi) T_\lambda f \Vert_{L^p_{x_0} L^q_{x'}} \lesssim \sum_{k \in \Z^m} \Vert T_\lambda^* (\hat{a}_k(x) \beta_k(\xi)) T_\lambda f \Vert_{L^p_{x_0} L^q_{x'}}.
\end{equation*}
Integration by parts yields
\begin{equation*}
|\hat{a}_k(x)| \lesssim_{\ell} (1+|k|)^{-\ell} \sum_{0 \leq |\alpha| \leq \ell} \Vert D_\xi^\alpha a(x,\cdot) \Vert_{L^1_\xi}.
\end{equation*}
Take $\ell = m + 1$ so that $\sum_{k \in \mathbb{Z}^m} (1+|k|)^{-\ell} \lesssim_m 1$. In this case, in \eqref{eq:KernelTensorStructure} we estimate  the kernel $B_\lambda(y,y^\prime)$ by
\begin{equation*}
|B_\lambda(y,y^\prime)| \lesssim_m (1+|k|)^{-\ell} \lambda^{-\frac{m}{2}} \sup_x  \sum_{0 \leq |\alpha| \leq \ell} \Vert D_\xi^\alpha a(x,\cdot) \Vert_{L_\xi^1} \lesssim_m (1+|k|)^{-\ell} C \lambda^{-\frac{m}{2}},
\end{equation*}
and  find $\hat{c}(\lambda(y-y^\prime)) = \hat{\beta}(k+\lambda(y-y^\prime))$. Taking absolute values, we infer
\begin{equation*}
T_\lambda^* ( \hat{a}_k(x) \beta_k(\xi) T_\lambda f)(y) \lesssim C_a (1+|k|)^{-\ell} \lambda^m \int | \hat{\beta}(k+\lambda(y-y^\prime)) f(y^\prime)| dy^\prime.
\end{equation*}
Since 
\begin{equation*}
\Vert \int |\hat{\beta}(k+\lambda(\cdot-y^\prime))| |f(y^\prime)| dy^\prime \Vert_{L_{x_0}^p L_{x^\prime}^q} = \Vert \int | \hat{\beta}(\lambda(\cdot-y^\prime))| |f(y^\prime)| dy^\prime \Vert_{L_{x_0}^p L_{x^\prime}^q},
\end{equation*}
 we can finish the proof  by
\begin{equation*}
\sum_{k \in \mathbb{Z}^m} \Vert T_\lambda^* (\hat{a}_k(x) e^{ik.\xi} \beta(\xi)) T_\lambda f \Vert_{L^p L^q} \lesssim_m \sum_{k \in \Z^m} (1+|k|)^{-\ell} C \Vert f \Vert_{L^p L^q} \lesssim_m C \Vert f \Vert_{L^p L^q}. \qedhere
\end{equation*}
\end{proof}

\subsection{Properties of rough symbols}

In this subsection compositions of pseudo-differential operators are recalled and their $L^2$-boundedness is quantified. The theory for smooth symbols is vast (cf.\ \cite{Hoermander2007,Sogge2017,Taylor1991}). For instance, the $L^p$-boundedness of symbols $a \in S_{1,\delta}^0$, $0 \leq \delta < 1$, is well-known, see \cite[Section~0.11]{Taylor1991}. Here we give a proof which quantifies in particular the $L^2$-boundedness of symbols $a \in C^s_x C^\infty_c$, see \cite[Chapter~2]{Taylor1991}. The argument is detailed as it becomes important in later sections.
\begin{lemma}
\label{lem:L2BoundednessRoughSymbols}
Let $1 \leq p,q \leq \infty$, and $a \in C^s_x C^\infty_c(\R^m \times \R^m)$ with $a(x,\xi)=0$ for $\xi \notin B(0,2)$. Suppose that
\begin{equation*}
\sup_{x \in \R^m} \sum_{0 \leq |\alpha| \leq m+1} \| D_\xi^{\alpha} a(x,\cdot) \|_{L_\xi^1}  \leq C.
\end{equation*}
Then, we find the following estimate to hold:
\begin{equation*}
\| a(x,D) f \|_{L^p L^q} \lesssim C \| f \|_{L^p L^q}.
\end{equation*}
\end{lemma}
\begin{proof}
We first consider the special case of separated variables $a(x,\xi) = b(x) c(\xi)$. H\"older's and Young's inequality give
\begin{equation*}
\begin{split}
\| a(x,D) f \|_{L^p L^q} &= \| b(x) c(D) f \|_{L^p L^q} \leq \| b \|_{L^\infty(\R^m)}  \| c(D) f \|_{L^p L^q} \\
 &\leq \| b \|_{L^\infty(\R^m)} \| \check{c} \|_{L^1(\R^m)} \| f \|_{L^p L^q(\R^m)}.
\end{split}
\end{equation*}
Integrating by parts to estimate $\| \check{c} \|_{L^1}$ yields $\| b \|_{L^\infty(\R^m)} \| \check{c} \|_{L^1} \lesssim C$.

In the general case, we write
\begin{equation*}
\begin{split}
a(x,D) f &= (2 \pi)^{-m} \int_{\R^m} e^{i x \xi} a(x,\xi) \hat{f}(\xi) d\xi 
= (2 \pi)^{-m} \int_{\R^m} e^{ix \xi} a(x,\xi) \beta(\xi) \hat{f}(\xi) d\xi
\end{split}
\end{equation*}
for $\beta \in C^\infty_c$ with support in $[-\pi,\pi]^m$ and $\beta(\xi) \equiv 1$ on $B(0,2)$. We expand $a(x,\xi) \beta(\xi)$ into the  Fourier series in $\xi$
\begin{equation*}
a(x,\xi) \beta(\xi) = \beta(\xi) \sum_{k \in \Z^m} a_k(x) e^{i k \xi},
\end{equation*}
with $a_k(x) = \int_{[-\pi,\pi]^m} e^{-ik\xi} a(x,\xi)  d\xi$.

We shall estimate every single term $a_k(x) e^{i k \xi} \beta(\xi)$ via the above argument and then sum over $k$ using decay from the regularity in $\xi$. We have
\begin{equation*}
\int a_k(x) e^{ik\xi} \beta(\xi) e^{i x \xi} \hat{f}(\xi) d\xi = a_k(x) (\beta(D) f)(x+k).
\end{equation*}
We find $|a_k(x)| \leq \|a(x, \cdot)\|_{L_\xi^1}$, and $\| \beta(D) f(\cdot + k) \|_{L^pL^q} \lesssim \| f \|_{L^pL^q}$ 
by Young's inequality and translation invariance. This estimate does not decay in $k$ sufficiently.
For decay in $k$, we integrate by parts in $\xi$ obtaining
\begin{equation*}
|a_k(x)| \lesssim_\ell (1+|k|)^{-\ell} \sum_{0 \leq |\alpha| \leq \ell} \| D^\alpha_\xi a(x,\cdot) \|_{L_\xi^1}.
\end{equation*}
For $\ell =m+1$ this is summable in $k \in \Z^m$, and we estimate
\begin{align*}
\| a(x,D) f \|_{L^p L^q} &\leq \sum_{k \in \Z^m} \| a_k \|_{L^\infty(\R^m)} \| f \|_{L^p L^q} \\
&\lesssim \sum_{k \in \Z^m} (1+|k|)^{-(m+1)} \sup_{x \in \R^m} \big( \sum_{0 \leq |\alpha| \leq m+1} \| D_\xi^\alpha a(x,\cdot) \|_{L_\xi^1} \big) \| f \|_{L^p L^q} \\
&\lesssim C \| f \|_{L^p L^q}. \qedhere
\end{align*}
\end{proof}

We turn to compound symbols. As we shall see in Section \ref{section:ReductionHalfWaveEstimates}, it suffices to prove dyadic estimates\footnote{Here we suppose that we are in the charge-free case for simplicity of exposition.}
\begin{equation*}
\lambda^{-\rho} \Vert S_\lambda u \Vert_{L^p L^q} \lesssim \Vert S_\lambda u \Vert_{L^2} + \Vert P^\lambda S_\lambda u \Vert_{L^2},
\end{equation*}
where $S_\lambda$ localizes to frequencies of size $\lambda \in 2^{\mathbb{N}_0}$ and $P^\lambda$ denotes the operator
\begin{equation*}
P^\lambda(x,D) = 
\begin{pmatrix}
\partial_t & 0 & - \partial_2 \\
0 & \partial_t & \partial_1 \\
-\partial_2(\varepsilon_{11}^{\lambda^{\frac{1}{2}}} \cdot) + \partial_1 (\varepsilon_{21}^{\lambda^{\frac{1}{2}}} \cdot) & \partial_1 (\varepsilon_{22}^{\lambda^{\frac{1}{2}}} \cdot) - \partial_2(\varepsilon_{12}^{\lambda^{\frac{1}{2}}} \cdot) & \partial_t
\end{pmatrix}
.
\end{equation*}
For the components of the permittivity, $\varepsilon_{ij}^{\lambda^{\frac{1}{2}}}$ means that the frequencies are truncated to size at most $\lambda^{\frac{1}{2}}$. Recall the symbol classes for $m \in \R$, $0 \leq \delta < \rho \leq 1$:
\begin{equation*}
S^m_{\rho,\delta} = \{ a \in C^\infty(\R^m \times \R^m) \, : \, |\partial_x^\alpha \partial_\xi^\beta a(x,\xi)| \lesssim \langle \xi \rangle^{m - \rho |\beta| + \delta |\alpha|} \}.
\end{equation*}

Hence, the pseudo-differential operators we encounter are smooth in $x$ and the considered symbols are in 
$S^m_{1,\frac{1}{2}}$. Boundedness $P(x,D): H^k(\R^d) \to H^{k-m}(\R^d)$ is proved in \cite[Prop.~0.5E]{Taylor1991}. 
Using  Lemma \ref{lem:L2BoundednessRoughSymbols}, we will show that the estimates are  independent of the dyadic frequency 
$\lambda$ when considering Littlewood--Paley pieces.

We recall compositions of pseudo-differential operators. Below we denote
\begin{equation*}
\partial_x^\alpha = \partial_{x_1}^{\alpha_1} \ldots \partial_{x_m}^{\alpha_m} \quad  \text{ and } 
 \quad D_\xi^\alpha = \partial_\xi^\alpha / (i^{|\alpha|}) \qquad \text{for \ } \alpha \in \mathbb{N}_0^m.
\end{equation*}

\begin{proposition}[{\cite[Proposition~0.3C]{Taylor1991}}]
\label{prop:KohnNirenberg}
Given $P(x,\xi) \in OPS^{m_1}_{\rho_1,\delta_1}$, $Q(x,\xi) \in OPS^{m_2}_{\rho_2,\delta_2}$, suppose that
\begin{equation*}
0 \leq \delta_2 < \rho \leq 1 \text{ with } \rho = \min(\rho_1,\rho_2).
\end{equation*}
Then, $(P \circ Q)(x,D) \in OPS^{m_1+m_2}_{\rho,\delta}$ with $\delta = \max(\delta_1,\delta_2)$, and $P(x,D) \circ Q(x,D)$ satisfies the asymptotic expansion
\begin{equation*}
(P \circ Q)(x,D) = \sum_{\alpha} \frac{1}{\alpha !} (D_\xi^\alpha P \; \partial_x^{\alpha} Q) (x,D) + R,
\end{equation*}
where $R: \mathcal{S}^\prime \to C^\infty$ is a smoothing operator.
\end{proposition}
When applying this formal expansion in Section \ref{section:ReductionHalfWaveEstimates}, we can verify with 
Lemma \ref{lem:L2BoundednessRoughSymbols} that the operators coming up in the expansion satisfy acceptable $L^2$-bounds. 
We revisit the proof of Theorem \ref{prop:KohnNirenberg} to find an explicit form of the remainder $R$, after truncating 
the series expansion. We shall derive
\begin{equation}
\label{eq:TruncatedSeries}
(P \circ Q)(x,D) = \sum_{|\alpha| \leq N} \frac{1}{\alpha !} (D_\xi^\alpha P \;\partial_x^{\alpha} Q)(x,D) + R_N(x,D)
\end{equation}
with a remainder $R_N$ for which we can infer $L^2$-bounds decaying in $\lambda$. We remark that the estimates almost follow 
from \cite[Proposition~0.5E]{Taylor1991} and Proposition \ref{prop:KohnNirenberg}. However, the encountered symbols have to be 
appropriately localized in frequency (see the end of the section), and we thus elect to give more details on the expansion and 
the error bounds.

We turn to the details. The distribution kernel of $(P \circ Q)(x,D)$ is given by
\begin{equation*}
I(x,y) = \int e^{i \langle x-y, \xi \rangle} (P \circ Q)(x,\xi) d\xi.
\end{equation*}
Here and below, the oscillatory integrals are understood in the sense of distributions (cf.\ \cite[Section~VII.8]{Hoermander2003}). For this purpose let $\rho \in C^\infty_c$ with $\rho \equiv 1$ in a neighbourhood of $0$, set $\rho_\delta(\xi) = \rho(\delta \xi)$, and read
\begin{equation*}
I(x,y) = \lim_{\delta \to 0} \int \rho_\delta(\xi) e^{i \langle x-y, \xi \rangle} (P \circ Q)(x,\xi) d\xi.
\end{equation*}
Furthermore,
\begin{align*}
(P \circ Q)(x,\xi) &= (2 \pi )^{-m} \iint e^{i [ \langle x-z, \eta \rangle + \langle z-x, \xi \rangle]} P(x,\eta) Q(z,\xi) d\eta dz \\
&= \big( \frac{\lambda}{2 \pi} \big)^m \iint e^{i \lambda \langle x-z, \tilde{\eta} - \tilde{\xi} \rangle } P(x,\lambda \tilde{\eta}) Q(z,\lambda \tilde{\xi}) d\tilde{\eta} dz,
\end{align*}
where $\eta = \lambda \tilde{\eta}$ and $\xi = \lambda \tilde{\xi}$. This integral is regarded as
\begin{align*}
&(P \circ Q)(x,\xi)= \\
& \big( \frac{\lambda}{2 \pi} \big)^m \lim_{\delta \to 0} \iint \rho_\delta(z^\prime, \eta^\prime) e^{i \frac{\lambda}{2} \langle(z^\prime, \eta^\prime), A(z^\prime, \eta^\prime) \rangle } P(x,\lambda(\eta^\prime + \tilde{\xi})) Q(x-z^\prime, \lambda \tilde{\xi}) dz^\prime d\eta^\prime,
\end{align*}
where
\begin{equation*}
A=
\begin{pmatrix}
0 & I_m \\
I_m & 0
\end{pmatrix}
.
\end{equation*}
The phase is stationary at $(z^\prime,\eta^\prime) = 0$. We take a smooth cutoff $\tilde\rho=\rho_{\delta'}$ around the origin. The possibility to choose the cutoff size $\delta'>\delta$ is used later.

The contribution away from the origin is smoothing:
\begin{align*}
\tilde{R}(x,\xi) = \lim_{\delta \to 0} &\iint e^{i \frac{\lambda}{2} \langle (z^\prime, \eta^\prime), A(z^\prime, \eta^\prime) \rangle} (1 - \tilde\rho(\eta^\prime,z^\prime)) \rho_\delta(\eta^\prime,z^\prime) \\
& \; \times P(x,\lambda(\eta^\prime + \tilde{\xi})) Q(x-z^\prime,\lambda \tilde{\xi}) dz^\prime d\eta^\prime.
\end{align*}
Indeed, in Section \ref{section:ProofHalfWaveEstimates} we shall see that for the expressions coming up in our analysis, we can show bounds $O_{L^2}(\lambda^{-N})$ for any $N$, independent of $\delta$. 

For the main contribution, we use Taylor's formula for
\begin{equation*}
f(\eta^\prime,z^\prime) = \tilde\rho(\eta^\prime,z^\prime) P(x,\lambda(\eta^\prime + \tilde{\xi})) Q(x-z^\prime,\lambda \tilde{\xi})
\end{equation*}
at the origin, namely
\begin{equation*}
f(\eta^\prime,z^\prime) = \sum_{|\alpha| \leq k} \frac{(D^\alpha f)(0)}{\alpha !} (\eta^\prime,z^\prime)^\alpha + \sum_{|\beta| = k+1} R_\beta(\eta^\prime,z^\prime) (\eta^\prime,z^\prime)^\beta.
\end{equation*}
(Here we can omit $\rho_\delta$ if $\delta$ is small enough compared to $\delta'$.) 

We turn to the first expression. If derivatives act on $\tilde\rho(z^\prime,\eta^\prime)$ and we evaluate at the origin, then the contribution will vanish. For $\alpha \in \mathbb{N}_0^{2m}$ we write in the following 
\begin{equation*}
\alpha = \gamma_1 \cup \gamma_2 \text{ with } \alpha = (\gamma_{11},\ldots,\gamma_{1m},\gamma_{21},\ldots,\gamma_{2m}).
\end{equation*}
We are left with
\begin{equation*}
\iint (\partial_{\eta^\prime}^{\gamma_1} P)(x,\lambda(\eta^\prime + \tilde{\xi})) \partial_{z^\prime}^{\gamma_2} Q(x-z^\prime, \lambda \tilde{\xi})) \vert_{(\eta^\prime,z^\prime) = 0} \frac{(\eta^\prime,z^\prime)^\alpha}{\alpha !} e^{i \frac{\lambda}{2} \langle(z^\prime, \eta^\prime), A(z^\prime,\eta^\prime) \rangle} dz^\prime d\eta^\prime.
\end{equation*}
Since (cf.\ \cite[Section~3.1,~p.~100]{Sogge2017})
\begin{equation*}
(2 \pi)^{-m} \iint (\eta^\prime,z^\prime)^\alpha e^{i \frac{\lambda}{2} \langle (z^\prime,\eta^\prime), A(z^\prime,\eta^\prime) \rangle } dz^\prime d\eta^\prime = \lambda^{-m}
\begin{cases}
\lambda^{-|\alpha|} \frac{\gamma_1 !}{i^{|\gamma_1|}}, \quad \gamma_1 = \gamma_2, \\
0, \quad \text{else},
\end{cases}
\end{equation*}
the Taylor polynomial yields the asserted asymptotic expansion.

We turn to the Taylor remainder estimate. We use the integral representation
\begin{equation*}
R_\beta(\eta^\prime,z^\prime) = \frac{|\beta|}{\beta !} \int_0^1 (1-t)^{|\beta|-1} \partial^\beta f(t(\eta^\prime,z^\prime)) dt.
\end{equation*}
Write $\partial^\beta = \partial^{\beta_1}_{\eta^\prime} \partial^{\beta_2}_{z^\prime}$.  By choosing $\delta' \leq \lambda^{-1}$, derivatives acting on $\rho_{\delta'}=\tilde{\rho}$ yield additional negative powers in $\lambda$, which makes the resulting expressions better behaved.
We thus suppose in the following that the derivatives do not act on $\tilde{\rho}$. We analyze the expression
\begin{align*}
&\quad \iint \tilde\rho(t\eta^\prime,tz^\prime) \lambda^{|\beta_1|} (\partial_{\eta^\prime}^{\beta_1} P)(x,\lambda (t \eta^\prime + \tilde{\xi})) (\partial_{z^\prime}^{\beta_2} Q)(x-tz^\prime,\lambda \tilde{\xi}) \\
&\qquad \times e^{i \frac{\lambda}{2} \langle (z^\prime,\eta^\prime), A(z^\prime,\eta^\prime) \rangle} (\eta^\prime)^{\beta_1} (z^\prime)^{\beta_2} dz^\prime d\eta^\prime\\
&= C \!\iint\! \tilde\rho(t\eta^\prime,tz^\prime) (\partial_{\eta^\prime}^{\beta_1} P)(x,\lambda(t \eta^\prime + \tilde{\xi})) (\partial_{z^\prime}^{\beta_2} Q)(x-tz^\prime, \lambda \tilde{\xi}) (\partial_{z^\prime}^{\beta_1} e^{i \lambda \langle z^\prime, \eta^\prime \rangle} ) (z^\prime)^{\beta_2} d\eta dz^\prime.
\end{align*}
Next, we integrate by parts in $z^\prime$. The derivatives can act on $Q$ or $(z^\prime)^{\beta_2}$ or on $\tilde\rho$. The latter  yields lower-order terms as argued above. We have to use the product rule. The derivatives acting on $z^\prime$ will be denoted with the multiindex $\beta_{sub} \leq \beta_1$, which is supposed to be understood componentwise. We can further suppose that $\beta_{sub} \leq \beta_2$ because the contribution vanishes otherwise. Hence, we can continue the above display by
\begin{equation}
\label{eq:IntegrationbyParts}
\begin{split}
&= 
\sum_{\beta_{sub} \leq \min(\beta_1,\beta_2)}\!C_{\beta} \iint \tilde\rho(t\eta^\prime,tz^\prime) (\partial_{\eta^\prime}^{\beta_1} P)(x,\lambda)(t \eta^\prime + \tilde{\xi})) (\partial_{z^\prime}^{\beta_1+\beta_2-\beta_{sub}} Q)(x-tz^\prime, \lambda \tilde{\xi}) \\
&\qquad\qquad \times e^{i \lambda \langle z^\prime, \eta^\prime \rangle} (z^\prime)^{\beta_2 - \beta_{sub}} d\eta^\prime dz^\prime +l.o.t. \\
&= \sum_{\beta_{sub} \leq \min(\beta_1,\beta_2)} \tilde C_{\beta}\iint\tilde \rho(t\eta^\prime,tz^\prime) (\partial_{\eta^\prime}^{\beta_1} P) (x,\lambda (t\eta^\prime + \tilde{\xi})) (\partial_{z^\prime}^{\beta_1 + \beta_2 - \beta_{sub}} Q)(x-tz^\prime, \lambda \tilde{\xi}) \\
&\qquad \qquad\times \partial_{\eta^\prime}^{\beta_2 - \beta_{sub}} \big( \frac{e^{i \lambda \langle z^\prime, \eta^\prime \rangle}}{\lambda^{|\beta_2 - \beta_{sub}|}} \big) d\eta^\prime dz^\prime +l.o.t. \\
&= \sum_{\beta_{sub} \leq \min(\beta_1,\beta_2)} C_{\beta}^\prime \iint \tilde\rho(t\eta^\prime, tz^\prime) (\partial^{\beta_1 + \beta_2 - \beta_{sub}}_{\eta^\prime} P)(x,\lambda(t \eta^\prime + \tilde{\xi}))) \\
&\qquad \qquad\times (\partial^{\beta_1+\beta_2-\beta_{sub}}_{z^\prime} Q)(x-tz^\prime, \lambda\tilde{\xi})e^{i \lambda \langle z^\prime, \eta^\prime \rangle} d\eta^\prime dz^\prime + l.o.t.
\end{split}
\end{equation}

In Section \ref{section:ReductionHalfWaveEstimates}, $L^2$-bounds for instances of this expression are a consequence of Lemma \ref{lem:L2BoundednessRoughSymbols}, possibly after choosing the cutoff $\tilde\rho$ differently.

We give a first application, which will be useful in Section \ref{section:ReductionHalfWaveEstimates}.
Let $(S_\lambda^\prime)_{\lambda \in 2^{\mathbb{N}_0}}$ denote an inhomogeneous Littlewood--Paley decomposition in $\R^n$,\footnote{We refer to Subsection \ref{subsection:ReductionC2} for details.} $\tilde{S}'_\lambda$ denote projections with mildly enlarged support, and suppose that $(\varepsilon^{ij}) \in C^1$ satisfies \eqref{eq:EllipticityPermittivity}. For $\lambda \gg 1$ let $D_\varepsilon$ and $\frac{1}{D_\varepsilon}$ denote the pseudo-differential operators given by
\begin{align*}
(D_\varepsilon f)(x^\prime) &= \frac{1}{(2 \pi)^n} \int_{\R^n} e^{i \langle x^\prime, \xi^\prime \rangle} \|\xi\|_{\varepsilon(x)} \hat{f}(\xi^\prime) d\xi^\prime, \\
\frac{1}{D_\varepsilon} f(x^\prime) &= \frac{1}{(2 \pi)^n} \int_{\R^n} e^{i \langle x^\prime, \xi^\prime \rangle} \frac{1}{\|\xi\|_{\varepsilon(x)}} \hat{f}(\xi^\prime) d\xi^\prime,
\end{align*}
where $\|\xi\|_{\varepsilon(x)} =\|\xi'\|_{\varepsilon(x)}= \big( \varepsilon_{\leq \lambda^{\frac{1}{2}}}^{ij} \xi_i \xi_j \big)^{1/2}$ (summing over $i,j\in\{1,\ldots,n\}$) and $\varepsilon^{ij}_{\leq \lambda^{\frac{1}{2}}}$ are the coefficients with smoothly truncated frequencies at $\lambda^{\frac{1}{2}}$. In Section \ref{section:ReductionHalfWaveEstimates} we shall see that choosing $\lambda=\lambda(\varepsilon) \gg 1$ preserves ellipticity of $(\varepsilon_{\leq \lambda^{\frac{1}{2}}}^{ij})_{i,j}$.

The following lemma shows that these operators essentially respect frequency localization as a consequence of the asymptotic expansion by Proposition \ref{prop:KohnNirenberg}.
\begin{lemma}
\label{lem:FrequencyPropagation}
Let $\lambda, \mu \in 2^{\mathbb{N}}$, $N\in\N$, and $1 \ll \min( \lambda, \mu) \ll \max(\lambda, \mu)$. Then, we find the following estimate to hold:
\begin{align}
\label{eq:ComparableFrequenciesDepsilon}
\| S^\prime_\mu D_\varepsilon S^\prime_\lambda f \|_{L^2} &\lesssim_N (\lambda \vee \mu)^{-N} \| \tilde{S}^\prime_\lambda f \|_{L^2}, \\
\label{eq:ComparableFrequeciesDepsilonInv}
\| S^\prime_\mu \frac{1}{D_\varepsilon} S^\prime_\lambda f \|_{L^2} &\lesssim_N (\lambda \vee \mu)^{-N} \| \tilde{S}^\prime_\lambda f \|_{L^2}.
\end{align}
\end{lemma}
\begin{proof}
We shall focus on the first estimate, as the proof of the second is similar. Firstly, suppose that $1 \ll \lambda \ll \mu$. We argue that \eqref{eq:ComparableFrequenciesDepsilon} follows from the expansion in the Kohn-Nirenberg theorem. Note that $D_\varepsilon S^\prime_\lambda$ has the symbol $\|\xi\|_{\varepsilon(x)} a_\lambda(\xi^\prime)$ and  $S^\prime_\mu$ has the symbol $a_\mu(\xi^\prime)$.
This means that all terms in the asymptotic expansion vanish because the supports in $\xi$ of the two symbols are disjoint. Furthermore, the estimate for the Taylor remainder follows from the representation \eqref{eq:IntegrationbyParts}. In fact, any derivative acting on $P$ gives a factor $\mu^{-1}$, whereas derivatives acting on $Q$ only lose factors $\lambda^{1/2}$. We turn to the estimate of the remainder
\begin{equation*}
C \lambda^n \iint e^{i \frac{\lambda}{2} \langle (z^\prime,\eta^\prime), A(z^\prime,\eta^\prime) \rangle} \rho_{\delta}(\eta^\prime,z^\prime) (1- \rho_{\delta^\prime}(\eta^\prime,z^\prime)) a(\frac{\lambda}{\mu}(\eta^\prime + \tilde{\xi})) \lambda \|\tilde{\xi}\|_{\varepsilon(x - z^\prime)} a(\tilde{\xi}) dz^\prime d\eta^\prime.
\end{equation*}
We suppose that $0 < \delta < \delta^\prime\le \lambda^{-1}$ with $\delta^\prime$ fixed. The phase $\frac{\lambda}{2} \langle (z^\prime,\eta^\prime),A(z^\prime,\eta^\prime) \rangle$ is non-stationary away from the origin. Consequently, we can integrate by parts in $(z^\prime,\eta^\prime)$. This gives factors $(\lambda|(z^\prime,\eta^\prime)|)^{-1}$ per integration by parts. When a derivative acts on $\rho_{\delta}$ or $\rho_{\delta^\prime}$, this gives factors of $\delta$ or $\delta^\prime$, respectively, and when it acts on $a\big( \frac{\lambda}{\mu} (\eta^\prime + \tilde{\xi}) \big)$, this gives factors $\frac{\lambda}{\mu}$, which are all favourable. More care is required when derivatives $\partial_{z^\prime}$ act on $\|\tilde{\xi}\|_{\varepsilon(x-z^\prime)}$. Since $\varepsilon^{ij}$ is regularized and we can only estimate $\| |D| \varepsilon^{ij} \|_{L^\infty} \lesssim 1$, additional derivatives in $z^\prime$ give powers of $\lambda^{\frac{1}{2}}$. We thus obtain sufficient decay to apply Lemma \ref{lem:L2BoundednessRoughSymbols} and conclude \eqref{eq:ComparableFrequenciesDepsilon}. 

We turn to the proof for $1 \ll \mu \ll \lambda$. Following along the above lines, we estimate the remainder
\begin{align*}
&C \lambda^n \iint e^{i \frac{\lambda}{2} \langle (z^\prime,\eta^\prime), A(z^\prime,\eta^\prime) \rangle } (1- \rho_{\delta^\prime}(z^\prime,\eta^\prime)) \rho_{\delta}(z^\prime,\eta^\prime) \\
&\qquad \times \tilde{a} \big( \frac{\lambda}{\mu} (\eta^\prime + \tilde{\xi}) \big) \lambda \|\tilde{\xi}\|_{\varepsilon(x-z^\prime)} a(\tilde{\xi}) dz^\prime d\eta^\prime,
\end{align*}
where $\text{supp}(\tilde{a}) \subseteq B(0,2)$ and $\tilde{a} \equiv 1$ on $B(0,1)$.
Integration by parts in $(z^\prime,\eta^\prime)$ yields
\begin{itemize}
\item powers of $(|(z^\prime,\eta^\prime)| \lambda)^{-1}$ from the non-stationary phase,
\item powers of $\frac{\lambda}{\mu}$ from derivatives acting on $\tilde{a} \big( \frac{\lambda}{\mu}(\eta^\prime + \tilde{\xi}) \big)$,
\item powers of $\lambda^{\frac{1}{2}}$ from derivatives acting on $\|\tilde{\xi}\|_{\varepsilon(x-z^\prime)}$.
\end{itemize}
We observe that due to the support of $(1-\rho_{\delta^\prime}(z^\prime,\eta^\prime))$ that $|(z^\prime,\eta^\prime)| \gtrsim (\delta^\prime)^{-1}$. Since $\delta^\prime \leq \lambda^{-1}$, every integration by parts gives a factor of $(\lambda|(z^\prime,\eta^\prime)|)^{-\frac{1}{2}}$.

The estimate for the Taylor remainder becomes more involved, too. Still, taking derivatives of $P$ as in \eqref{eq:IntegrationbyParts} yields
\begin{equation*}
(\partial^{\beta_1 + \beta_2 - \beta_{sub}}_{\eta^\prime} P)(x,\lambda(t \eta^\prime + \tilde{\xi})) = \big( \frac{1}{\mu} \big)^{|\beta_1 + \beta_2 - \beta_{sub}|} (\partial^{\beta_1+\beta_2-\beta_{sub}}_{\xi} \tilde{a})\big( \frac{\lambda (t \eta^\prime + \tilde{\xi})}{\mu} \big).
\end{equation*}
In the derivatives $\partial^{\beta_1+\beta_2-\beta_{sub}}_{z^\prime} Q$ we only lose $\lambda^{\frac{|\beta_1+\beta_2-\beta_{sub}|}{2}}$. Choosing $\mu = \lambda /C$ with $C$ a large, but fixed constant, the proof is complete. Alternatively, one can argue by taking adjoints (cf. \cite[Prop.~0.3B]{Taylor1991}).
\end{proof}

We end the section with discussing variants, which will be useful later on. Consider the region
\begin{equation*}
\{ |\xi_0| \lesssim |(\xi_1,\xi_2)| \sim \lambda \} = A_\lambda \subseteq \R^3.
\end{equation*}
Let $S^\prime_{\lambda,\tau}$ denote the smooth frequency projection to $A_\lambda$ and $\tilde{A}_\lambda$ a mildly enlarged region and $\tilde{S}^\prime_{\lambda,\tau}$ the corresponding frequency projection.
Let $S^\tau_{\gg \lambda}$ be the smooth frequency projection to frequencies $\{|\xi_0| \gg \lambda\}$. By the same argument as above, we see
\begin{equation*}
\| S^\tau_{\gg \lambda} \partial_i^k D_\varepsilon \tilde{S}^\prime_{\lambda,\tau} \|_{L^2 \to L^2} \lesssim_{k,N} \lambda^{-N}.
\end{equation*}
 Together with the above estimates, we see that $D_\varepsilon S^\prime_{\lambda, \tau} f$ and $\frac{1}{D_\varepsilon} S^\prime_{\lambda,\tau} f$ are still essentially frequency localized in Fourier space in $A_\lambda$. More precisely, we find the following estimate to hold:
\begin{equation*}
\| (1- \tilde{S}^\prime_{\lambda,\tau}) \partial_i^k D_\varepsilon S^\prime_{\lambda,\tau} f \Vert_{L^2} \lesssim_{N,k} \lambda^{-N} \| \tilde{S}^\prime_{\lambda,\tau} f \|_{L^2},
\end{equation*}
and likewise for $\frac{1}{D_\varepsilon}$.

\section{Reduction to dyadic estimates for the half-wave equation}
\label{section:ReductionHalfWaveEstimates}

In this section, we show that Theorems \ref{thm:StrichartzEstimatesMaxwell2d} and \ref{thm:StrichartzEstimatesL1LinfCoefficients} follow from dyadic estimates for the half-wave equation, given in Proposition \ref{prop:HalfWaveEstimates}. The key point is to diagonalize the principal symbol, which is carried out first. Furthermore, by commutator and microlocal estimates, we localize in phase space to a region close to the characteristic surface. We require that $\varepsilon\in C^1$ as assumed in Theorems \ref{thm:StrichartzEstimatesMaxwell2d} and \ref{thm:StrichartzEstimatesL1LinfCoefficients}.
Observe that $\tilde\varepsilon$ inherits the assumptions on $\varepsilon$ (up to constants) of these theorems.

\subsection{Diagonalizing the principal symbol}

Firstly, we carry out the diagonalization of the principal symbol. To obtain a better approximation, we consider the operators directly. We remark that diagonalizing the symbol combined with the arguments from \cite{Tataru2000} allows to prove Strichartz estimates for coefficients in $C^1$ or with derivative in $L^p L^{\infty}$. However, since the estimates obtained in \cite{Tataru2000} are not sharp in terms of derivative loss, the corresponding estimates for first-order systems proved this way are not sharp either. Nonetheless, this observation can be useful as it saves error estimates for compounds of pseudo-differential operators. The error analysis is carried out in Paragraph \ref{subsection:ErrorBoundsCompoundSymbols} in the present context. For more complicated first order systems this might not be easily possible. Here, we carry out the detailed computations in our special case
\begin{equation*}
\varepsilon(x) = 
\begin{pmatrix}
\varepsilon^{11}(x) & \varepsilon^{12}(x) \\
\varepsilon^{12}(x) & \varepsilon^{22}(x)
\end{pmatrix}
.
\end{equation*}

For $P$ as in \eqref{eq:RoughSymbolMaxwell2d} we find $P=Op(\tilde{p}(x,\xi))$ with
\begin{equation}
\label{eq:SymbolP}
\begin{split}
\tilde{p}(x,\xi) &= 
\begin{pmatrix}
i \xi_0 & 0 & -i \xi_2 \\
0 & i \xi_0 & i \xi_1 \\
-i\xi_2 \varepsilon_{11}(x) + i\xi_1 \varepsilon_{12}(x) & i \xi_1 \varepsilon_{22}(x) - i \xi_2 \varepsilon_{12}(x) & i \xi_0
\end{pmatrix}
\\
&\quad  +
\begin{pmatrix}
0 & 0 & 0 \\
0 & 0 & 0 \\
(-\partial_2 \varepsilon_{11} + \partial_1 \varepsilon_{12})(x) & (\partial_1 \varepsilon_{22} - \partial_2 \varepsilon_{12})(x) & 0
\end{pmatrix}
,
\end{split}
\end{equation}
where the first matrix is the principal symbol $p(x,\xi)$. As the operator associated with the second matrix is bounded in $L^2$ for $\varepsilon \in C^1$, it will be neglected. Let
\begin{equation*}
\tilde{\varepsilon}(x) = (\tilde{\varepsilon}^{ij}(x))=
\begin{pmatrix}
\varepsilon_{22}(x) & -\varepsilon_{12}(x) \\
-\varepsilon_{21}(x) & \varepsilon_{11}(x)
\end{pmatrix}
\end{equation*}
denote the adjugate matrix of $\varepsilon^{-1}$, i.e., $\varepsilon$ up to determinant.
We compute the eigenvalues of $p$ to be $i\xi_0$, $i(\xi_0 - \|\xi' \|_{\tilde{\varepsilon}})$, and $i(\xi_0 + \|\xi' \|_{\tilde{\varepsilon}})$. Denote
\begin{equation}
\label{eq:DiagonalSymbol}
d(x,\xi) = \text{diag}(i\xi_0 , i(\xi_0 - \|\xi^\prime\|_{\tilde{\varepsilon}}), i(\xi_0+\|\xi^\prime\|_{\tilde{\varepsilon}})),
\end{equation}
and set $\xi^*_j = \xi_j / \| \xi^\prime \|_{\tilde{\varepsilon}}$ for $j\in\{1,2\}$. The corresponding eigenvectors we align as
\begin{equation}
\label{eq:EigenvectorMatrices}
m(x,\xi) =
\begin{pmatrix}
- \xi^*_1 \varepsilon_{22}(x) + \xi_2^* \varepsilon_{12}(x) & \xi_2^* & -\xi_2^* \\
\xi_1^* \varepsilon_{12}(x) - \xi_2^*\varepsilon_{11}(x)  & -\xi_1^* & \xi_1^* \\
0 & 1 & 1
\end{pmatrix}
.
\end{equation}
The inverse matrix is computed to
\begin{equation}
\label{eq:InverseEigenvectorMatrix}
m^{-1}(x,\xi) =
\begin{pmatrix}
-\xi_1^* & -\xi_2^* & 0 \\
\frac{\xi_2^* \varepsilon_{11}(x) - \xi_1^* \varepsilon_{12}(x)}{2 } & \frac{-\xi_1^* \varepsilon_{22}(x)+\xi_2^*\varepsilon_{12}(x) }{2 } & \frac{1}{2} \\
\frac{- \xi^*_2 \varepsilon_{11}(x) + \xi_1^* \varepsilon_{12}(x)}{2} & \frac{\xi_1^* \varepsilon_{22}(x)-\xi_2^* \varepsilon_{12}(x) }{2} & \frac{1}{2}
\end{pmatrix},
\end{equation}
and hence 
\begin{equation*}
m(x,\xi) d(x,\xi) m^{-1}(x,\xi) = p(x,\xi).
\end{equation*}
By the arguments from \cite{Tataru2000} and Proposition \ref{prop:MultiplierProposition} this decomposition gives (non-sharp) Strichartz estimates for $\varepsilon \in C^1$. To prove the sharp result for $\varepsilon \in C^2$, we diagonalize $P$ with pseudo-differential operators. Until the end of this subsection, we suppose that $\varepsilon \in C^\infty$. After having reduced to dyadic estimates, we shall see that we can truncate frequencies of $\varepsilon$. This will allow to work with smooth symbols. Of course, we have to show bounds independent of the dyadic frequency range.

 We now turn to the corresponding operators, starting with
\begin{equation}
\label{eq:DiagonalMatrixOperator} 
\mathcal{D}(x,D) = \text{diag}(\partial_t, \partial_t - iD_{\tilde{\varepsilon}}, \partial_t + iD_{\tilde{\varepsilon}})
\end{equation}
induced by $d(x,\xi)$. To the eigenvectors in $m$ we associate the operator
\begin{equation}
\label{eq:EigenvectorOperators}
\mathcal{M}(x,D) = 
\begin{pmatrix}
\frac{i}{D_{\tilde{\varepsilon}}} (\partial_1 (\varepsilon_{22} \cdot) - \partial_2( \varepsilon_{12} \cdot)) & \frac{-i}{D_{\tilde{\varepsilon}}} \partial_2 & \frac{i}{D_{\tilde{\varepsilon}}} \partial_2 \\
\frac{i}{D_{\tilde{\varepsilon}}} (\partial_2 (\varepsilon_{11} \cdot) - \partial_1 (\varepsilon_{12} \cdot)) & \frac{i}{D_{\tilde{\varepsilon}}} \partial_1 & 
\frac{-i}{D_{\tilde{\varepsilon}}} \partial_1 \\
0 & 1 & 1
\end{pmatrix},
\end{equation}
and to the inverse matrix $m^{-1}$ we relate
\begin{equation}
\label{eq:InverseEigenvectorOperators}
\mathcal{N}(x,D) =
\begin{pmatrix}
i \partial_1 \frac{1}{D_{\tilde{\varepsilon}}} & i \partial_2 \frac{1}{D_{\tilde{\varepsilon}}} & 0 \\
\frac{i (\varepsilon_{12} \partial_1 -\varepsilon_{11} \partial_2 )}{2} \frac{1}{D_{\tilde{\varepsilon}}} &    \frac{i(\varepsilon_{22} \partial_1 - \varepsilon_{12} \partial_2 )}{2} \frac{1}{D_{\tilde{\varepsilon}}} & \frac{1}{2} \\
\frac{i (\varepsilon_{11} \partial_2 - \varepsilon_{12} \partial_1)}{2} \frac{1}{D_{\tilde{\varepsilon}}} & \frac{i (\varepsilon_{12} \partial_2 - \varepsilon_{22} \partial_1}{2} \frac{1}{D_{\tilde{\varepsilon}}} & \frac{1}{2} 
\end{pmatrix}
.
\end{equation}
Here $(\varepsilon_{ij} \cdot)$ denotes the  multiplication operator induced by $\varepsilon_{ij}$. Note that
$\mathcal{M}$ and $\mathcal{N}$ are bounded in $L^2$. We compute 
\begin{equation}
\label{eq:SecondOrderApproximation}
\begin{split}
(\mathcal{M} \mathcal{D} \mathcal{N})_{11} &= - \frac{1}{D_{\tilde{\varepsilon}}} [ \partial_1 (\varepsilon_{22} \cdot) \partial_t \partial_1 - \partial_2 (\varepsilon_{12} \cdot) \partial_t \partial_1  + \partial_2 \partial_t (\varepsilon_{11} \cdot) \partial_2 - \partial_2 \partial_t(\varepsilon_{12} \cdot) \partial_1] \frac{1}{D_{\tilde{\varepsilon}}}, \\ 
(\mathcal{M} \mathcal{D} \mathcal{N})_{12} &= -\frac{1}{D_{\tilde{\varepsilon}}} [ \partial_1(\varepsilon_{22} \cdot) \partial_t \partial_2 - \partial_2(\varepsilon_{12} \cdot) \partial_t \partial_2 + \partial_2 \partial_t (\varepsilon_{12} \cdot) \partial_2 - \partial_2 \partial_t (\varepsilon_{22} \cdot) \partial_1] \frac{1}{D_{\tilde{\varepsilon}}}, \\
(\mathcal{M} \mathcal{D} \mathcal{N})_{13} &= - \frac{1}{D_{\tilde{\varepsilon}}} \partial_2 D_{\tilde{\varepsilon}}, \\
(\mathcal{M} \mathcal{D} \mathcal{N})_{21} &= - \frac{1}{D_{\tilde{\varepsilon}}} [\partial_2 (\varepsilon_{11} \cdot) \partial_t \partial_1 - \partial_1(\varepsilon_{12} \cdot) \partial_t \partial_1 + \partial_1 \partial_t (\varepsilon_{12} \cdot) \partial_1 - \partial_1 \partial_t(\varepsilon_{11} \cdot) \partial_2) ] \frac{1}{D_{\tilde{\varepsilon}}}, \\
 (\mathcal{M} \mathcal{D} \mathcal{N})_{22} &= - \frac{1}{D_{\tilde{\varepsilon}}} [ \partial_2 (\varepsilon_{11} \cdot) \partial_t \partial_2 - \partial_1 (\varepsilon_{12} \cdot) \partial_t \partial_2 + \partial_1 \partial_t (\varepsilon_{22} \cdot) \partial_1 - \partial_1 \partial_t (\varepsilon_{12} \cdot) \partial_2 ] \frac{1}{D_{\tilde{\varepsilon}}},\\
(\mathcal{M} \mathcal{D} \mathcal{N})_{23} &= \frac{1}{D_{\tilde{\varepsilon}}} \partial_1 D_{\tilde{\varepsilon}}, \\
(\mathcal{M} \mathcal{D} \mathcal{N})_{31} &= D_{\tilde{\varepsilon}} ((\varepsilon_{12} \cdot) \partial_1 - (\varepsilon_{11} \cdot) \partial_2) \frac{1}{D_{\tilde{\varepsilon}}}, \\
 (\mathcal{M} \mathcal{D} \mathcal{N})_{32} &= D_{\tilde{\varepsilon}} ((\varepsilon_{22} \cdot) \partial_1 - (\varepsilon_{12} \cdot) \partial_2) \frac{1}{D_{\tilde{\varepsilon}}},\\
  (\mathcal{M} \mathcal{D} \mathcal{N})_{33} &= \partial_t.
\end{split} 
\end{equation}
(We use the same symbol for the operator $\mathcal{D}$ and the displacement field. This should not cause confusion
since they do not appear in the same context.)
In Subsection \ref{subsection:ErrorBoundsCompoundSymbols} we shall see that the difference
$\mathcal{M} \mathcal{D} \mathcal{N} -P$ is bounded in $L^2$ with suitable frequency localization.

\subsection{Reductions for $C^2$-coefficients} 
\label{subsection:ReductionC2}

Next, we carry out the reductions for $u$ and $\varepsilon_{ij}$ needed for the proof of Theorem \ref{thm:StrichartzEstimatesMaxwell2d}, that is
\begin{itemize}
\item localization to a cube of size $1$ and to high frequencies,
\item reduction to dyadic estimates,
\item truncating frequencies of the coefficients,
\item reduction to half-wave equations.
\end{itemize}
Before these steps, by scaling we can assume that $|\partial^2_x \varepsilon^{ij}| \leq 1$ and $\kappa =1$.

\subsubsection{Localization to a cube of size 1 and to high frequencies} \label{subsub:cube}
We first show that it is sufficient to estimate $\| \langle D \rangle^{-\rho} u \|_{L^p L^q}$ instead of $\| |D|^{-\rho} u \|_{L^p L^q}$.
Let $s(\xi)$ be a symbol supported in $\{ 1/2 \leq |\xi| \leq 2\}$ such that 
\begin{equation*}
\sum_j s(2^{-j} \xi) = 1, \quad \xi \in \R^3 \backslash \{ 0 \}.
\end{equation*}
For $\lambda \in 2^{\mathbb{N}_0}$, let $S_\lambda = S(D/\lambda)$ denote the Littlewood--Paley multiplier, which localizes to frequencies of size $\lambda$ and $S_0 = 1 - \sum_{j \geq 0 } S_{2^j}$. Write $u= S_0 u + (1-S_0) u$.

Sobolev's embedding and the Hardy-Littlewood-Sobolev inequality yield
\begin{equation*}
 \Vert |D|^{-\rho} S_0 u \Vert_{L^p L^q} \lesssim \Vert S_0 u \Vert_{L^2} \lesssim \Vert u \Vert_{L^2}.
\end{equation*}
For the contribution of $(1-S_0) u$ in \eqref{eq:StrichartzEstimatesMaxwell2d} we observe that
\begin{equation*}
 \partial_k(\varepsilon_{ij}(1-S_0)u) = \partial_k(\varepsilon_{ij} u) - \partial_k(\varepsilon_{ij} S_0 u)
\end{equation*}
and
\begin{equation*}
\begin{split}
 \Vert \partial_k (\varepsilon_{ij} S_0 u) \Vert_{L^2} &\lesssim \Vert \varepsilon_{ij} \partial_k S_0 u \Vert_{L^2} + \Vert (\partial_k \varepsilon_{ij}) S_0 u \Vert_{L^2} \\
 &\lesssim \Vert \varepsilon_{ij} \Vert_{L^\infty} \Vert S_0 u \Vert_{L^2} + \Vert \partial_k \varepsilon_{ij} \Vert_{L^\infty} \Vert S_0 u \Vert_{L^2} \\
 &\lesssim \Vert \varepsilon_{ij} \Vert_{C^1} \Vert S_0 u \Vert_{L^2}.
\end{split}
\end{equation*}
Hence, $\Vert P (1-S_0) u\Vert_{L^2} \lesssim \Vert P u \Vert_{L^2} + \Vert \varepsilon_{ij} \Vert_{C^1} \Vert u \Vert_{L^2}$.

This means that low frequencies can always be estimated by the Hardy-Little\-wood-\-So\-bolev inequality and Sobolev's embedding, so that we can  assume that $u$ has only large frequencies. It suffices to prove 
\begin{equation*}
 \Vert (1+|D|^2)^{\frac{-\rho}{2}} u \Vert_{L^pL^q} \lesssim \Vert u \Vert_{L^2} + \Vert P u \Vert_{L^2} + \| \langle D' \rangle^{-\frac{1}{2}} \rho_e \|_{L^2}.
\end{equation*}
Take a smooth partition of unity
\begin{equation*}
 1 = \sum_{j \in \Z^{n+1}} \chi_j(x), \quad \chi_j(x) = \chi(x-j), \quad \text{supp} (\chi) \subseteq B(0,2),
\end{equation*}
and let $\rho_j = \partial_1 (\chi_j u_1) + \partial_2 (\chi_j u_2)$. By considering commutators, we first note
\begin{equation*}
 \sum_j \Big(\Vert \chi_j u \Vert^2_{L^2} + \Vert P (\chi_j u) \Vert_{L^2}^2  \Big)\lesssim \Vert u \Vert_{L^2}^2 + \Vert Pu \Vert^2_{L^2}.
\end{equation*}
On the other hand, since $p,q \geq 2$ we have
\begin{equation}
\label{eq:LocLpLq}
 \Vert (1+|D|^2)^{\frac{-\rho}{2}} u \Vert^2_{L^pL^q} \lesssim \sum_j \Vert (1+|D|^2)^{\frac{-\rho}{2}} \chi_j u \Vert_{L^pL^q}^2.
\end{equation}
To prove the latter estimate, we write
\begin{equation*}
(1+|D|^2)^{-\frac{\rho}{2}} u = \sum_{j,k} \chi_j(x) (1+|D|^2)^{-\frac{\rho}{2}} \chi_k(x) u.
\end{equation*}
If $|j-k| \geq 100$,  for the distributional kernel of $\chi_j (1+|D|^2)^{-\frac{\rho}{2}} \chi_k $
we find the bound
\begin{align*}
(2 \pi)^3 |G(x,y)| &= |\int e^{i \langle x-y, \xi \rangle}\chi_j(x)\chi_k(y) (1+|\xi|^2)^{-\frac{\rho}{2}} d\xi| \\
 &\lesssim_N (1+|j-k|)^{-N} (1+|x-y|)^{-N} 
\end{align*}
since $|x-y| \gtrsim |j-k|\ge1$. Hence, we can estimate $\chi_j (1+|D|^2)^{-\frac{\rho}{2}} \chi_k u$ 
for $|j-k| \geq 10 n$ using Young's inequality and obtain
\begin{equation*}
\| (1+|D|^2)^{-\frac{\rho}{2}} u \|^2_{L^p L^q} \lesssim \| u \|^2_{L^2} + \| \sum_{j,k: |j-k| \leq 10n } \chi_k (1+|D|^2)^{-\rho/2} \chi_j u \|^2_{L^p L^q}.
\end{equation*}
Let $\tilde{\chi}_j = \sum_{|k-j| \leq 10n} \chi_k$. Due to the pointwise bound
\begin{equation*}
\sum_j {\tilde{\chi}_j} (1+|D|^2)^{-\rho/2} \chi_j u \lesssim \big( \sum_j | \tilde{\chi}_j (1+|D|^2)^{-\rho/2} \chi_j u|^2 \big)^{1/2},
\end{equation*}
and Minkowski's inequality for $p,q \geq 2$, we conclude the proof of \eqref{eq:LocLpLq}. It remains to show
\begin{equation}
\label{eq:ChargeLoc}
\sum_j \| \langle D \rangle^{-\frac{1}{2}} \rho_j \|_{L^2}^2 
  \lesssim \| \langle D \rangle^{-\frac{1}{2}} \rho_e \|_{L^2}^2 + \|u\|_{L^2}^2.
\end{equation}
This fact is a consequence of the inequality $\sum_j \|\chi_j v\|^2_{H^{-1/2}}\le \|v\|^2_{H^{-1/2}}$ which is dual to
\begin{equation*}
\| u \|^2_{H^{\frac{1}{2}}} \leq \sum_j \| \chi_j u \|^2_{H^{\frac{1}{2}}}.
\end{equation*}
The above estimate follows by interpolation from its elementary variants in $L^2$ and $H^1$. This concludes the reduction to compact support, and we suppose in the following that $u$ is supported in the unit cube. 

\subsubsection{Reduction to dyadic estimates}

Here we shall see that it is enough to prove
\begin{equation}
\label{eq:DyadicLocalization}
 \lambda^{-\rho} \Vert S_\lambda u \Vert_{L^p L^q} \lesssim \Vert S_\lambda u \Vert_{L^2} + \Vert P S_\lambda u \Vert_{L^2} + \lambda^{-\frac{1}{2}} \| S_\lambda \rho_e \|_{L^2}
\end{equation}
for $\lambda\ge1$. To reduce to the above display, we have to show the commutator estimate
\begin{equation*}
 \sum_{\lambda = 2^j \geq 1} \Vert [P,S_\lambda] u \Vert_{L^2}^2 \lesssim \Vert u \Vert_{L^2}^2. 
\end{equation*}
Set $\tilde{S}_\lambda = \sum_{|j| \leq 2} S_{2^j \lambda}$ as a mildly enlarged version of $S_\lambda$, and
write $[P,S_\lambda] = [P,S_\lambda] \tilde{S}_\lambda - S_\lambda P (1-\tilde{S}_\lambda)$. We need the inequalities
\begin{equation*}
 \begin{split}
  \Vert [P,S_\lambda] v \Vert_{L^2} &\lesssim \Vert v \Vert_{L^2}, \\
  \Vert S_\lambda P (1-\tilde{S}_\lambda) u \Vert_{L^2} &\lesssim \lambda^{-\delta}\,\Vert u \Vert_{L^2} 
 \end{split}
\end{equation*}
for some $\delta>0$. Since the commutator
\[[\partial_k (\varepsilon_{ij} \cdot),S_\lambda] v = \partial_k(\varepsilon_{ij} S_\lambda v) - S_\lambda \partial_k (\varepsilon_{ij} v) = \partial_k [\varepsilon_{ij},S_\lambda] v\]
has the kernel $K(x,y) = \partial_k(\varepsilon_{ij}(x)-\varepsilon_{ij}(y)) \hat{s}(\lambda(x-y)) \lambda^{n+1}$,
the $L^2$-boundedness follows from $\varepsilon\in C^2$.
For the second term note that only frequencies of $\varepsilon_{ij}$ of size $\lambda$ and higher matter as
\begin{equation*}
S_\lambda \partial_k ( \varepsilon_{ij}(1-\tilde{S}_\lambda) u) = S_\lambda \partial_k (\varepsilon_{ij}^{\gtrsim \lambda} (1-\tilde{S}_\lambda) u).
\end{equation*}
We thus find
\begin{equation*}
\Vert S_\lambda \partial_k ( \varepsilon_{ij}^{\gtrsim \lambda} (1-\tilde{S}_\lambda) u) \Vert_{L^2} \lesssim \lambda \Vert \varepsilon_{ij}^{\gtrsim \lambda} \Vert_{L^\infty} \Vert u \Vert_{L^2} \lesssim \lambda^{-1} \Vert \varepsilon_{ij} \Vert_{C^2} \Vert u \Vert_{L^2},
\end{equation*}
due to the standard estimate
\begin{equation*}
\sup_\lambda \lambda^2 \Vert \varepsilon_{ij}^{\geq \lambda} \Vert_{L^\infty} \lesssim \Vert \varepsilon_{ij} \Vert_{C^2}.
\end{equation*}
Observe that $S_\lambda u$ is no supported in the unit cube anymore. It is still rapidly decreasing, which suffices
for the arguments in Section~\ref{section:ProofHalfWaveEstimates}, cf.\ \eqref{eq:v-loc}.

\subsubsection{Truncating the coefficients of $P$ at frequency $\lambda^{\frac{1}{2}}$}
We check that it is enough to prove \eqref{eq:DyadicLocalization} when the coefficients have Fourier transform supported in $\{|\xi| \leq \lambda^{\frac{1}{2}} \}$. We stress that the uniform ellipticity for $\varepsilon$ after frequency truncation will still be important. For that purpose, we observe that $\| \varepsilon_{ij}^{\geq \lambda^{\frac{1}{2}}} \|_{L^\infty} \lesssim \lambda^{-1} \| \varepsilon_{ij}^{\geq \lambda^{\frac{1}{2}}} \|_{C^2}$, and thus
\begin{align*}
\varepsilon_{ij}^{\leq \lambda^{\frac{1}{2}}} \xi_i \xi_j &= \varepsilon_{ij} \xi_i \xi_j - \varepsilon_{ij}^{\geq \lambda^{\frac{1}{2}}} \xi_i \xi_j 
\geq (\Lambda_1 - C\lambda^{-1} )\| \xi' \|^2\ge \tfrac12\Lambda_1\| \xi' \|^2
\end{align*}
for sufficiently large $\lambda\ge1$. 

The error term coming from frequency truncation in \eqref{eq:DyadicLocalization} is estimated by
\begin{equation*}
 \begin{split}
  \Vert \partial_k (\varepsilon_{ij}^{\gtrsim \lambda^{\frac{1}{2}}} S_\lambda u ) \Vert_{L^2} &\lesssim \Vert \partial_k( \varepsilon_{ij}^{\gtrsim \lambda^{\frac{1}{2}}}) \Vert_{L^\infty} \Vert S_\lambda u \Vert_{L^2} + \Vert \varepsilon^{\gtrsim \lambda^{\frac{1}{2}}}_{ij} \partial_k S_\lambda u \Vert_{L^2} \\
  &\lesssim \Vert \varepsilon_{ij} \Vert_{C^1} \| S_\lambda u \|_{L^2} + \lambda \Vert \varepsilon_{ij}^{\gtrsim \lambda^{\frac{1}{2}}} \Vert_{L^\infty} \Vert S_\lambda u \Vert_{L^2},
 \end{split}
\end{equation*}
which is bounded by the first term on the right-hand side of \eqref{eq:DyadicLocalization}. In the following we write $\varepsilon^{\lambda^{\frac{1}{2}}}_{ij} = \varepsilon^{\leq \lambda^{\frac{1}{2}}}_{ij}$. Consequently, it is enough to prove
\begin{equation}
\label{eq:DyadicLocalizationTruncatedFrequencies}
 \lambda^{-\rho} \Vert S_\lambda u \Vert_{L^p L^q} \lesssim \Vert S_\lambda u \Vert_{L^2} + \Vert P^\lambda S_\lambda u \Vert_{L^2} + \lambda^{-\frac{1}{2}} \| S_\lambda \rho_e \|_{L^2},
\end{equation}
with the operator  
\begin{equation*}
 P^\lambda = 
 \begin{pmatrix}
  \partial_t & 0 & -\partial_2 \\
  0 & \partial_t & \partial_1 \\
  -\partial_2(\varepsilon_{11}^{\lambda^{\frac{1}{2}}} \cdot) + \partial_1(\varepsilon_{12}^{\lambda^{\frac{1}{2}}} \cdot) & \partial_1(\varepsilon_{22}^{\lambda^{\frac{1}{2}}} \cdot) - \partial_2(\varepsilon_{12}^{\lambda^{\frac{1}{2}}} \cdot) & \partial_t
 \end{pmatrix}
 ,
\end{equation*}
whose coefficients are truncated at frequency $\lambda^{\frac{1}{2}}$. If there is no possibility of confusion,
the frequency truncation of the coefficients will be implicit in the following, and we write $P$ instead of $P^\lambda$. 

\subsubsection{Reduction to half-wave equations}
\label{subsubsection:ReductionHalfWaveEquations}
We shall consider the two refined regions $\{|\xi_0| \gg |(\xi_1,\xi_2)| \}$ and 
$\{|\xi_0| \lesssim |(\xi_1,\xi_2)| \}$. As noticed above, the first region is away from the characteristic 
surface. To deal with it, we first apply Theorem \ref{thm:ApproximationTheorem} obtaining 
\begin{equation*}
\Vert T_\lambda \left( \frac{P(x,D)}{\lambda} S_\lambda u \right) - p(x,\xi) T_\lambda S_\lambda u \Vert_{L^2_{\Phi}} \lesssim \lambda^{-\frac{1}{2}} \Vert S_\lambda u \Vert_{L^2}.
\end{equation*}
Suppose that $u$ has frequencies in $\{|\xi_0| \gg |(\xi_1,\xi_2)| \}$. By  Theorem \ref{thm:ApproximationTheorem}, 
the same is true for $v_\lambda = T_\lambda S_\lambda u$ up to an acceptable error 
$\lambda^{-\frac{1}{2}} \Vert S_\lambda u \Vert_{L^2}$. For such $v_\lambda$ it is easy to see that 
\begin{equation*}
\Vert v_\lambda \Vert_{L^2_\Phi}\lesssim \Vert p(x,\xi) v_\lambda \Vert_{L^2_\Phi}
\end{equation*}
because
\begin{equation*}
p(x,\xi) = m(x,\xi) d(x,\xi) m^{-1}(x,\xi) \quad \text{ and } \quad |d_{ii}| \gtrsim |\xi_0| \gtrsim 1.
\end{equation*}
Using also the $L^2$-mapping properties of $T_\lambda$ and the triangle inequality, we deduce
\begin{equation*}
 \Vert S_\lambda u \Vert_{L^2} =   \Vert v_\lambda \Vert_{L^2_\Phi}\lesssim \lambda^{-\frac{1}{2}} \Vert S_\lambda u \Vert_{L^2} +\lambda^{-1} \Vert P(x,D) S_\lambda u \Vert_{L^2}.
\end{equation*}
Sobolev's embedding thus yields
\begin{equation*}
 \lambda^{-\rho} \Vert S_\lambda u \Vert_{L^p L^q} 
 \lesssim \lambda^{-\rho} \lambda^{\frac{1}{2}+\rho} \Vert S_\lambda u \Vert_{L^2} 
 \lesssim \Vert S_\lambda u \Vert_{L^2}+ \Vert P(x,D) S_\lambda u \Vert_{L^2}.
\end{equation*}
So \eqref{eq:DyadicLocalizationTruncatedFrequencies} is true if the frequencies of  $u$ are confined to
 $\{|\xi_0| \gg |(\xi_1,\xi_2)| \}$.

We turn to the main contribution coming from $\{|\xi_0| \lesssim |(\xi_1,\xi_2)| \}$ where the 
characteristic surfaces are contained. In the following we assume that the space-time Fourier transform of $u$ 
 is supported in this region. We first prove
\begin{align*}
\lambda^{-\rho} \Vert S_\lambda w \Vert_{L^p L^q} 
 &\lesssim \Vert S_\lambda w \Vert_{L^2} + \Vert \mathcal{D} S_\lambda w \Vert_{L^2} + \lambda^{-\frac{1}{2}} \| S_\lambda \rho_e \|_{L^2}\\
 &\lesssim \Vert S_\lambda u \Vert_{L^2} + \Vert \mathcal{D} S_\lambda w \Vert_{L^2} + \lambda^{-\frac{1}{2}} \| S_\lambda \rho_e \|_{L^2},
\end{align*}
where  $w = \tilde{S}_\lambda \mathcal{N} S_\lambda u$ and as above
\begin{equation*}
 \mathcal{D} = 
 \begin{pmatrix}
  i \partial_t & 0 & 0 \\
  0 & i \partial_t - D_{\tilde{\varepsilon}} & 0 \\
  0 & 0 & i \partial_t + D_{\tilde{\varepsilon}}
 \end{pmatrix}
 .
\end{equation*}
The estimates of the second and third component of $S_\lambda w$ are a consequence of Proposition \ref{prop:HalfWaveEstimates} to be established in Section~\ref{section:ProofHalfWaveEstimates}. 
Here the charge $\rho_e$ does not enter. For the first component, as in Subsection~\ref{subsection:ErrorBoundsCompoundSymbols} 
below one shows that the product $\partial_1 \frac1{D_{\tilde\varepsilon}} S_\lambda$ has the symbol $i\xi_j^* s_\lambda(\xi)$
up to an error bounded by $c\lambda^{-1/2}\Vert S_\lambda u \Vert_{L^2}$. Combined with Theorem \ref{thm:ApproximationTheorem}, we deduce 
\begin{equation*}
\Vert T_\lambda w_1 - [m^{-1}(x,\xi) T_\lambda S_\lambda u]_1 \Vert_{L^2_\Phi} \lesssim \lambda^{-\frac{1}{2}} \Vert S_\lambda u \Vert_{L^2}.
\end{equation*}

By $\partial_1 u_1 + \partial_2 u_2= \rho_e$ and Theorem \ref{thm:ApproximationTheorem}, we find
\begin{align*}
 \|&[m^{-1}(x,\xi) T_\lambda S_\lambda u]_1 \Vert_{L^2_\Phi} \\
&\le \Vert \frac{1}{D_{\tilde{\varepsilon}}} S_\lambda \rho_e \|_{L^2} + \Vert T_\lambda ( \frac{1}{D_{\tilde{\varepsilon}}} \partial_1 S_\lambda u_1 + \frac{1}{D_{\tilde{\varepsilon}}} \partial_2 S_\lambda u_2) - i \xi_1^* T_\lambda u_1 - i \xi_2^* T_\lambda u_2 \Vert_{L^2_{\Phi}}\\
&\lesssim \lambda^{-1} \|S_\lambda \rho_e \|_{L^2} + \lambda^{-\frac{1}{2}} \Vert S_\lambda u \Vert_{L^2}.
\end{align*}
Recall that the ultimate estimate for the first term is a consequence of Lemma \ref{lem:L2BoundednessRoughSymbols} and the previous frequency localization $\{ |\xi'| \sim |\xi| \}$.

Hence, using Sobolev's embedding and the $L^2$-mapping properties of the FBI-transform,
the contribution of $S_\lambda w_1$ can be estimated by 
\begin{equation*}
\lambda^{-\rho} \Vert S_\lambda w_1 \Vert_{L^p L^q} \lesssim \lambda^{\rho + \frac{1}{2} - \rho} \Vert S_\lambda w_1
\Vert_{L^2} = \lambda^{\frac{1}{2}}\Vert T_\lambda w_1 \Vert_{L^2_\Phi}\lesssim \Vert S_\lambda u \Vert_{L^2} + \lambda^{-\frac{1}{2}} \| S_\lambda \rho_e \|_{L^2},
\end{equation*}
which  gives the estimate for $w_1$. So far, we have proved that
\begin{equation*}
\begin{split}
\lambda^{-\rho} \Vert \tilde{S}_\lambda \mathcal{N} S_\lambda u \Vert_{L^p L^q} 
&\lesssim \Vert  S_\lambda u \Vert_{L^2} + \Vert \mathcal{D} \tilde{S}_\lambda \mathcal{N} S_\lambda u \Vert_{L^2} + \lambda^{-\frac{1}{2}} \| S_\lambda \rho_e \|_{L^2} \\
\end{split}
\end{equation*}

 For \eqref{eq:DyadicLocalization}, we yet have to show that
\begin{align}
\label{eq:L2BoundednessNInverse}
\lambda^{-\rho} \Vert S_\lambda u \Vert_{L^p L^q} &\lesssim \lambda^{-\rho} \Vert \tilde{S}_\lambda \mathcal{N} S_\lambda u \Vert_{L^p L^q} + \Vert S_\lambda u \Vert_{L^2},\\
\label{eq:L2BoundednessMInverse}
\Vert \mathcal{D} \tilde{S}_\lambda \mathcal{N} S_\lambda u \Vert_{L^2} &\lesssim \Vert \tilde{S}_\lambda \mathcal{M} \tilde{S}_\lambda \mathcal{D} \tilde{S}_\lambda \mathcal{N} S_\lambda u \Vert_{L^2} + \| S_\lambda u \|_{L^2},
\end{align}
and \eqref{eq:L2MatrixErrorEstimate} below.
We start with the proof of \eqref{eq:L2BoundednessMInverse}. Let $\tilde{w} = \mathcal{D} \tilde{S}_\lambda \mathcal{N} S_\lambda u$. We can as well consider $\tilde{w} = \tilde{S}_\lambda \mathcal{D} \tilde{S}_\lambda \mathcal{N} S_\lambda u $ as $\mathcal{D}$ and $\mathcal{N}$ respect frequency localization up to negligible errors. Theorem \ref{thm:ApproximationTheorem} and calculations as in the next subsection yield 
\begin{equation*}
\Vert T_\lambda^* m(x,\xi) T_\lambda \tilde{S}_\lambda \tilde{w} - \tilde{S}_\lambda \mathcal{M} \tilde{S}_\lambda \tilde{w} \Vert_{L^2} \lesssim \lambda^{-\frac{1}{2}} \Vert \tilde{S}_\lambda \tilde{w} \Vert_{L^2}.
\end{equation*}
By the triangle inequality we infer
\begin{equation*}
\begin{split}
\Vert \tilde{S}_\lambda \tilde{w} \Vert_{L^2} &= \Vert T_\lambda^* m^{-1}(x,\xi) T_\lambda T_\lambda^* m(x,\xi) T_\lambda \tilde{S}_\lambda \tilde{w} \Vert_{L^2} \\
&\lesssim \Vert T_\lambda^* m(x,\xi) T_\lambda \tilde{S}_\lambda \tilde{w} \Vert_{L^2} \\
&\lesssim \Vert \tilde{S}_\lambda \mathcal{M} \tilde{S}_\lambda \tilde{w} \Vert_{L^2} + \lambda^{-\frac{1}{2}} \Vert \tilde{S}_\lambda \tilde{w} \Vert_{L^2},
\end{split}
\end{equation*}
where $\lambda^{-\frac{1}{2}} \Vert \tilde{S}_\lambda w \Vert_{L^2}$ can be absorbed into the left hand-side for $\lambda$ large enough.

For the proof of \eqref{eq:L2BoundednessNInverse},  we write in a similar way
\begin{equation*}
\tilde{S}_\lambda \mathcal{N} S_\lambda u = T_\lambda^* m^{-1}(x,\xi) T_\lambda S_\lambda u + E S_\lambda u
\end{equation*}
with $\Vert E \Vert_{L^2 \to L^2} \lesssim \lambda^{-\frac{1}{2}}$. 
Proposition \ref{prop:MultiplierProposition} and Sobolev's embedding then imply
\begin{align*}
\lambda^{-\rho} \Vert S_\lambda u \Vert_{L^p L^q} 
   &\lesssim \lambda^{-\rho} \|T_\lambda^* m^{-1}(x,\xi) T_\lambda S_\lambda u \Vert_{L^pL^q} \\
  &\lesssim \lambda^{-\rho} \Vert \tilde{S}_\lambda \mathcal{N} S_\lambda u \Vert_{L^p L^q} 
        + \lambda^{\frac12} \|E S_\lambda u\|_{L^2}\\
&\lesssim \lambda^{-\rho} \Vert \tilde{S}_\lambda \mathcal{N} S_\lambda u \Vert_{L^p L^q} + \|S_\lambda u\|_{L^2} .
 \end{align*}

\subsection{Estimates of the error terms}
\label{subsection:ErrorBoundsCompoundSymbols}
The purpose of this paragraph is to prove
\begin{equation}
\label{eq:L2MatrixErrorEstimate}
\| \mathcal{M} \mathcal{D} \mathcal{N} S_{\lambda,\tau}^\prime - P S_{\lambda,\tau}^\prime \|_{L^2 \to L^2} \lesssim 1,
\end{equation}
where $S'_{\lambda,\tau}$ was introduced at the end of Section \ref{section:PDOFBITransform} and $\mathcal{M}$, $\mathcal{D}$, $\mathcal{N}$ in \eqref{eq:DiagonalMatrixOperator}--\eqref{eq:InverseEigenvectorOperators}. We show the estimate for $C^1$-coefficients $\tilde\varepsilon$ which are frequency truncated at $\lambda^{\frac{1}{2}}$. This will make the estimate applicable for our proofs of Theorems \ref{thm:StrichartzEstimatesMaxwell2d} and \ref{thm:StrichartzEstimatesL1LinfCoefficients}. For the sake of brevity, we write $S_\lambda'$ for $S'_{\lambda,\tau}$ in the following.
\begin{proposition}\label{prop:MDN-P}
Let $\varepsilon \in C^1$ and suppose that \eqref{eq:EllipticityPermittivity} is satisfied. With the notations from the previous sections, we find \eqref{eq:L2MatrixErrorEstimate} to hold.
\end{proposition}

\begin{proof}
We use Lemma \ref{lem:FrequencyPropagation} to include frequency projections between $\mathcal{M}$, $\mathcal{D}$, and $\mathcal{N}$.\footnote{Strictly speaking, we should always enlarge the frequency projection a little bit when applying Lemma \ref{lem:FrequencyPropagation}. This is not recorded to lighten the notation.} We prove \eqref{eq:L2MatrixErrorEstimate} componentwise. In detail we shall analyze $(\mathcal{M} \mathcal{D} \mathcal{N})_{11}$, $(\mathcal{M} \mathcal{D} \mathcal{N})_{13}$, and $(\mathcal{M} \mathcal{D} \mathcal{N})_{31}$ as the claim follows for the other components by the same means.

\smallskip

\emph{Estimate for }$(\mathcal{M} \mathcal{D} \mathcal{N})_{11}$: Firstly, suppose that $\tilde{\varepsilon}$
is time-independent. (Later we see that the argument extends to time-dependent $\tilde{\varepsilon}$.) We shall show
\begin{equation*}
\Vert \frac{1}{D_{\tilde{\varepsilon}}} \tilde{S}_\lambda^\prime (\partial_i(\tilde{\varepsilon}^{ij}_{\lambda^{\frac{1}{2}}}(\cdot) \partial_j) \tilde{S}_\lambda^\prime \frac{1}{D_{\tilde{\varepsilon}}} S_\lambda^\prime - S_\lambda^\prime \Vert_{L^2 \to L^2} \lesssim \lambda^{-1}.
\end{equation*}
The above display implies $\| (\mathcal{M} \mathcal{D} \mathcal{N})_{11}S_\lambda' - \partial_t S_\lambda' \|_{L^2 \to L^2} \lesssim 1$.
In the proof come up error terms such that
\begin{equation*}
\Vert \frac{1}{D_{\tilde{\varepsilon}}} \tilde{S}_\lambda^\prime ((\partial_1 \tilde{\varepsilon}_{\lambda^{\frac{1}{2}}}^{11}) \cdot ) \partial_1 \tilde{S}_\lambda^\prime \frac{1}{D_{\tilde{\varepsilon}}} \tilde{S}_\lambda^\prime \Vert_{L^2 \to L^2} \lesssim \frac{1}{\lambda},
\end{equation*}
which are straightforward. Thus, it is enough to analyze
\begin{equation*}
-\frac{1}{D_{\tilde{\varepsilon}}} \tilde{S}_\lambda^\prime (\tilde{\varepsilon}_{\lambda^{\frac{1}{2}}}^{ij} \partial_i \partial_j ) \tilde{S}_\lambda^\prime \frac{1}{D_{\tilde{\varepsilon}}} S_\lambda^\prime.
\end{equation*}
This we break into three symbols
\begin{equation*}
\frac{1}{\|\xi\|_{\tilde{\varepsilon}}} \tilde{a}_\lambda(\xi), \quad (\tilde{\varepsilon}^{ij}_{\lambda^{\frac{1}{2}}} \xi_i \xi_j \tilde{a}_\lambda(\xi)), \quad \frac{1}{\|\xi\|_{\tilde{\varepsilon}}} a_\lambda(\xi).
\end{equation*}
We compute the expansion of the first and second symbol 
\begin{equation}\label{eq:symb-MDN11}
\frac{1}{\|\xi\|_{\tilde{\varepsilon}}} \tilde{\varepsilon}^{ij}_{\lambda^{\frac{1}{2}}} \xi_i \xi_j \tilde{a}_\lambda(\xi) + \sum_{1 \leq |\alpha| \leq N^\prime} \frac{1}{\alpha !} \big( D_\xi^\alpha \frac{1}{\|\xi\|_{\tilde{\varepsilon}}} \big) \partial_x^\alpha \big( \tilde{\varepsilon}^{ij}_{\lambda^{\frac{1}{2}}} \xi_i \xi_j \tilde{a}_\lambda(\xi) \big) + R_{N^\prime},
\end{equation}
with $\| R_{N^\prime} \|_{L^2 \to L^2} \lesssim_{N^\prime} \lambda^{-N}$ for some $N\ge1$ depending on $N'$,
cf.\ \eqref{eq:IntegrationbyParts}.

Denote the operators in the expansion by $E^\alpha$. We find from collecting powers of $\lambda$, namely 
$\lambda^{\frac{|\alpha| -1}{2}}\lambda^2$ for factors with derivatives in $x$ and $\lambda^{-1-|\alpha|}$ 
for those with derivatives in $\xi$, that
\begin{equation*}
\Vert E^\alpha \Vert_{L^2 \to L^2} \lesssim_\alpha \lambda^{\frac{1}{2} - \frac{|\alpha|}{2}}.
\end{equation*}
Consequently, the error is bounded in $L^2$. Together with $\frac{1}{D_{\tilde{\varepsilon}}} S_\lambda^\prime$ 
it allows us to estimate the  $L^2$-operator norm by $\frac{1}{\lambda}$, which is acceptable.

The leading-order symbol in \eqref{eq:symb-MDN11} is given by
$\|\xi\|_{\tilde{\varepsilon}}^{-1}\tilde{\varepsilon}^{ij}_{\lambda^{\frac{1}{2}}} \xi_i \xi_j \tilde{a}_\lambda(\xi).$
It is homogeneous of degree $1$. We calculate the composite with $\frac{1}{D_{\tilde{\varepsilon}}} S^\prime_\lambda$.
Like in the previous computation, we can estimate the lower-order terms in $L^2$ by $\frac{1}{\lambda}$. The leading-order term is given by
\begin{equation*}
\frac{1}{\|\xi\|_{\tilde{\varepsilon}}^2} \tilde{\varepsilon}^{ij}_{\lambda^{\frac{1}{2}}} \xi_i \xi_j \tilde{a}_\lambda(\xi) = \tilde{a}_\lambda(\xi).
\end{equation*}
Consequently,
\begin{equation*}
-\frac{1}{D_{\tilde{\varepsilon}}} \tilde{S}_\lambda^\prime \Delta_{\tilde{\varepsilon}} \tilde{S}_\lambda^\prime \frac{1}{D_{\tilde{\varepsilon}}} S_\lambda^\prime = S_\lambda^\prime + E_{N^\prime} S_\lambda^\prime \quad \text{ with \ } \| E_{N^\prime} \|_{L^2 \to L^2} \lesssim \lambda^{-1}.
\end{equation*}
Finally, for time-dependent $\tilde{\varepsilon}$, we compute as above 
\begin{equation*}
\tilde{S}_\lambda^\prime \partial_t \frac{1}{D_{\tilde{\varepsilon}}} S_\lambda^\prime = \tilde{S}_\lambda^\prime \frac{1}{D_{\tilde{\varepsilon}}} \partial_t S_\lambda^\prime + O_{L^2}(\lambda^{-1}),
\end{equation*}
and in the same venue,
\begin{equation*}
\Vert \frac{1}{D_{\tilde{\varepsilon}}} \partial_i \tilde{S}_\lambda^\prime \partial_t (\tilde{\varepsilon}^{ij}_{\lambda^{1/2}}) \partial_j \tilde{S}_\lambda^\prime \frac{1}{D_{\tilde{\varepsilon}}} \tilde{S}_\lambda^\prime \Vert_{L^2 \to L^2} \lesssim 1.
\end{equation*}

\smallskip

\emph{Estimate for} $(\mathcal{M} \mathcal{D} \mathcal{N} \big)_{13}$: We have to prove 
\begin{equation*}
\| \big( \mathcal{M} \mathcal{D} \mathcal{N} \big)_{13} S_\lambda^\prime - P_{13} S_\lambda^\prime \|_{L^2 \to L^2} = \|{-}\frac{1}{D_{\tilde{\varepsilon}}} \partial_2 D_{\tilde{\varepsilon}} S_\lambda^\prime + \partial_2 S_\lambda^\prime \|_{L^2 \to L^2} \lesssim 1.
\end{equation*}
First note that
\begin{equation*}
\| {-}\frac{1}{D_{\tilde{\varepsilon}}} \partial_2 \tilde{S}_\mu^\prime D_{\tilde{\varepsilon}} \tilde{S}_\lambda^\prime \|_{L^2 \to L^2} \lesssim_N (\mu \vee \lambda)^{-N}
\end{equation*}
for $\mu \ll \lambda$ or $\mu \gg \lambda$  because of the estimates
\begin{equation*}
\| S^\prime_\mu D_{\tilde{\varepsilon}} S^\prime_\lambda \|_{L^2 \to L^2} \lesssim_N (\mu \vee \lambda)^{-N} \quad \text{and} \quad \| \frac{1}{D_{\tilde{\varepsilon}}} \partial_i \tilde{S}^\prime_\mu \|_{L^2 \to L^2} \lesssim \mu,
\end{equation*}
see Lemma \ref{lem:FrequencyPropagation}. Hence, it is enough to show
\begin{equation*}
\Vert  {-} \frac{1}{D_{\tilde{\varepsilon}}} \partial_2 \tilde{S}_\lambda^\prime D_{\tilde{\varepsilon}} S_\lambda^\prime + \partial_2 S_\lambda^\prime \Vert_{L^2 \to L^2} \lesssim 1.
\end{equation*}
The first operator is composed of the two operators with symbols
\begin{equation*}
p(x,\xi) = - \frac{1}{\|\xi\|_{\tilde{\varepsilon}}} (i \xi_2) \tilde{a}_\lambda(\xi) \quad \text{ and } \quad q(x,\xi) = \|\xi\|_{\tilde{\varepsilon}} a_\lambda(\xi).
\end{equation*}
We find
\begin{equation*}
(P \circ Q)(x,D) = Op(-i \xi_2 a_\lambda(\xi)) + \sum_{1 \leq |\alpha| \leq N^\prime} \frac{1}{\alpha !} Op( D_\xi^\alpha p(x,\xi) \partial_x^\alpha q(x,\xi) ) + R_{N'}
\end{equation*}
with $\| R_{N'} \|_{L^2 \to L^2} \lesssim_{N'} \lambda^{-N}$.
The operators $E^\alpha$ in the sum over $\alpha$ can be bounded by collecting powers of $\lambda$ and 
exploiting homogeneity, where one obtains $\lambda^{-|\alpha|}$ from  $D_\xi^\alpha p$ 
and $\lambda^{\frac{|\alpha|+1}{2}}$ from $D_x^\alpha q$. In total, we find the inequality 
$\| E^\alpha \|_{L^2 \to L^2} \lesssim \lambda^{\frac{1-|\alpha|}{2}}$, and conclude the $L^2$-boundedness of $\sum_{1 \leq |\alpha| \leq N^\prime} E^\alpha$. 
The proof is complete since
\begin{equation*}
Op(i \xi_2 a_\lambda(\xi)) = \partial_2 S_\lambda^\prime.
\end{equation*}

\smallskip

\emph{Estimate for }$ \big( \mathcal{M} \mathcal{D} \mathcal{N} \big)_{31}$:
Below we show
\begin{equation*}
\| \big[ \big( \mathcal{M} \mathcal{D} \mathcal{N} \big)_{31} + \partial_2 \big( \varepsilon_{11}^{ \lambda^{1/2}} \cdot \big) - \partial_1 (\varepsilon_{12}^{\lambda^{1/2}} \cdot) \big] S_\lambda^\prime \|_{L^2 \to L^2} \lesssim 1.
\end{equation*}
First, $((\partial_2 \varepsilon_{11}) - (\partial_1 \varepsilon_{12}))(\cdot )$ is bounded in $L^2$. As above the contribution of $\tilde{S}^\prime_\mu \frac{1}{D_{\tilde{\varepsilon}}} \tilde{S}^\prime_\lambda$ can be neglected for $\lambda \ll \mu$ or $\mu \ll \lambda$. It remains to verify
\begin{equation*}
\| {-} D_{\tilde{\varepsilon}} \tilde{S^\prime_\lambda} \varepsilon_{11}^{\lambda^{\frac{1}{2}}} \partial_2 \tilde{S^\prime_\lambda} \frac{1}{D_{\tilde{\varepsilon}}} S^\prime_\lambda + \varepsilon^{\lambda^{\frac{1}{2}}}_{11} \partial_2 S^\prime_\lambda \|_{L^2 \to L^2} \lesssim 1
\end{equation*}
since the terms containing $\varepsilon_{12}^{\lambda^{\frac{1}{2}}}$ can be estimated in the same venue.

The first operator we perceive as composition of the operators associated with the symbols
\begin{equation*}
-\|\xi\|_{\tilde{\varepsilon}} \tilde{a}_\lambda(\xi), \quad \varepsilon^{\lambda^{\frac{1}{2}}}_{11} (-i \xi_2) \tilde{a}_\lambda(\xi), \quad \frac{1}{\|\xi\|_{\tilde{\varepsilon}}} a_\lambda(\xi).
\end{equation*}
For the composite of the first and second symbol, we find
\begin{equation*}
Op(\varepsilon_{11}^{\lambda^{1/2}} (i \xi_2) \|\xi\|_{\tilde{\varepsilon}} \tilde{a}_\lambda ) + \sum_{1 \leq |\alpha| \leq N^\prime} \frac{1}{\alpha !} Op(D_\xi^\alpha (-\|\xi\|_{\tilde{\varepsilon}} \tilde{a}_\lambda) \partial_x^\alpha \varepsilon_{11}^{\lambda^{\frac{1}{2}}} (-i\xi_2) \tilde{a}_\lambda(\xi) \big) + R_{N'}
\end{equation*}
with $\| R_{N'} \|_{L^2 \to L^2} \lesssim \lambda^{-N}$. The operators $E^\alpha$ satisfy
\begin{equation*}
\| E^\alpha \|_{L^2 \to L^2} \lesssim_\alpha \lambda^{1-|\alpha|} \lambda^{\frac{|\alpha|-1}{2}} \lambda = \lambda^{\frac{3-|\alpha|}{2}}.
\end{equation*}
The error in $L^2 \to L^2$ is thus bounded  by $\lambda$. Since $\| \frac{1}{D_{\tilde{\varepsilon}}} \tilde{S}_\lambda^\prime \|_{L^2 \to L^2} \lesssim \frac{1}{\lambda}$, this gives an $L^2$-bounded contribution.

Next, we compute the composite $C$ of
\begin{equation*}
Op(\varepsilon^{\lambda^{1/2}}_{11} (i \xi_2) \|\xi\|_{\tilde{\varepsilon}} \tilde{a}_\lambda(\xi) ) \quad \text{ and } \quad \frac{1}{D_{\tilde{\varepsilon}}} S_\lambda^\prime.
\end{equation*}
We find
\begin{equation*}
C= \varepsilon^{\lambda^{1/2}}_{11}(x) \partial_2 S_\lambda^\prime + \sum_{1\le|\alpha|\le N'} \frac{1}{\alpha !} Op(D_\xi^\alpha (\varepsilon^{\leq \lambda^{1/2}}_{11} i \xi_2 |\xi|_{\tilde{\varepsilon}} \tilde{a}_\lambda(\xi)) \partial_x^\alpha \big( \frac{1}{\|\xi\|_{\tilde{\varepsilon}}} a_\lambda(\xi) \big) + R_{N'},
\end{equation*}
where $\| R_{N'} \|_{L^2 \to L^2} \lesssim_{N'} \lambda^{-N}$. The error estimate
\begin{equation*}
\sum_{1 \leq |\alpha| \leq N^\prime} \| E^\alpha \|_{L^2 \to L^2} \lesssim 1
\end{equation*}
is routine by now. The proof is complete.
\end{proof}

\subsection{Reductions for $\partial^2_x \varepsilon \in L^1 L^\infty$ }
As in Subsection~\ref{subsection:ReductionC2} we carry out the following steps to reduce Theorem \ref{thm:StrichartzEstimatesL1LinfCoefficients} 
to the dyadic estimates
\begin{equation}
\label{eq:DyadicEstimateL1LinfCoefficients}
\lambda^{-\rho} \| S_\lambda u \|_{L^p L^q} \lesssim \| S_\lambda u \|_{L^\infty L^2} + \| P S_\lambda u \|_{L^2} + \lambda^{-\frac{1}{2}} \| S_\lambda\rho_e \|_{L^2}
\end{equation}
for $\lambda\gtrsim1$, where the Fourier support of $\tilde{\varepsilon}$ contained in $\{ | \xi | \leq \lambda^{1/2} \}$ and $u$ essentially supported in the unit cube and its space-time Fourier transform is supported in 
$\{ |\xi_0| \lesssim |(\xi_1,\xi_2)|  \}$. These steps are
\begin{itemize}
\item reduction to the case $\kappa = 1$,
\item reduction to a cube of size $1$ and to large frequencies, 
\item estimate away from the characteristic surface,
\item reduction to dyadic estimates,
\item truncating the coefficients at frequency $\lambda^{\frac{1}{2}}$.
\end{itemize}
At this point, we can use the estimate from Section \ref{subsection:ErrorBoundsCompoundSymbols} to complete the reduction from Theorem \ref{thm:StrichartzEstimatesL1LinfCoefficients} to Proposition \ref{prop:HalfWaveEstimates}.

\subsubsection{Reduction to the case $\kappa = 1$.} For this we can follow the argument from \cite[Section~3]{Tataru2002} closely. We omit the details. 

\subsubsection{Reduction to a cube of size $1$ and to large frequencies.} 
\label{subsubsection:ReductionCubeLargeFrequencies}
If $\kappa = 1$, then we can choose $T =1$ by rescaling. It is enough to show
\begin{equation}
\label{eq:InhomogeneousEstimateMixedCoefficientsI}
\begin{split}
\| \langle D^\prime \rangle^{- \rho} u \|_{L^p(0,1;L^q)} &\lesssim \| u \|_{L^\infty L^2} + \| P(x,D) u \|_{L^1 L^2} \\
&\quad + \| \langle D' \rangle^{-\frac{1}{2}} \rho_e(0) \|_{L^2(\R^2)} + \| \langle D' \rangle^{-\frac{1}{2}} \partial_t \rho_e \|_{L^1 L^2}
\end{split}
\end{equation}
because  the Hardy-Littlewood-Sobolev and Bernstein's inequality yield
\begin{equation*}
\| | D^\prime |^{-\rho} S_0^\prime u \|_{L^p(0,1;L^q)} \lesssim \| S_0^\prime u \|_{L^\infty L^2}\lesssim \| u \|_{L^\infty L^2} 
\end{equation*}
for the low frequencies. For high frequencies, inequalites \eqref{eq:InhomogeneousEstimateMixedCoefficientsI}
and \eqref{eq:StrichartzEstimatesL1LinfCoefficients} with $\kappa=1$ are equivalent. 
We next reduce \eqref{eq:InhomogeneousEstimateMixedCoefficientsI} to the estimate
\begin{equation}
\label{eq:InhomogeneousEstimateReductionI}
 \| |D^\prime|^{-\rho} u \|_{L^p(0,2;L^q)} \lesssim  \| u \|_{L^2}+ \| P(x,D) u \|_{L^2} + \| \langle D' \rangle^{-\frac{1}{2}} \rho_e \|_{L^2} .
 \end{equation}
Indeed, if one applies this inequality to solutions $u$ of the homogenous problem with a cut-off
in time, one obtains 
\begin{align*}
\| |D^\prime|^{-\rho} u \|_{L^p(0,1;L^q)}
&\lesssim \|u\|_{L^2} + \| \langle D' \rangle^{-\frac{1}{2}} \rho_e \|_{L^2} 
 \lesssim\|u(0)\|_{L^2(\R^2)} \!+ \| \langle D' \rangle^{-\frac{1}{2}} \rho_e(0)\|_{L^2(\R^2)}
 \end{align*}
using also the basic energy estimate. Combined with Duhamel's formula and Minkowski's inequality, 
this estimate implies \eqref{eq:InhomogeneousEstimateMixedCoefficientsI}. Regarding the role of the charge, we observe that for free solutions $(u_1,u_2,u_3)$ we have $\partial_t \rho_e = 0$. However, a free solution emanating from $Pu(s)$ gives charges $\partial_t \rho_e$ by \eqref{eq:rho}. Since we use Duhamel's formula in the proof, the additional term $\| |D'|^{-\frac{1}{2}} \partial_t \rho_e \|_{L^1 L^2}$ appears.
 
 Inequality \eqref{eq:InhomogeneousEstimateReductionI} respects the finite speed of propagation and, as in 
 Paragraph~\ref{subsub:cube},  we can decompose $u$ in components supported in cubes of sidelength $1$. These estimates sum up to \eqref{eq:InhomogeneousEstimateReductionI}.

\subsubsection{Estimate away from the characteristic surface}
Next, we argue that it is enough to prove the stronger estimate
\begin{equation}
\label{eq:EstimateAwayFromSurface}
\| |D|^{-\rho} u \|_{L^p L^q} \lesssim \| u \|_{L^2} + \| P(x,D) u \|_{L^2} 
   + \|| D|^{-\frac{1}{2}} \rho_e \|_{L^2} .
\end{equation}
Apparently, for $u$ having space-time Fourier transform in the region $\{ |\tau| \lesssim |\xi^\prime| \}$, \eqref{eq:EstimateAwayFromSurface} implies \eqref{eq:InhomogeneousEstimateReductionI}. But for $\{|\xi^\prime| \ll |\tau | \}$, $P$ is an elliptic operator with Lipschitz coefficients of order $1$, which gains one derivative (cf.\ Paragraph \ref{subsubsection:ReductionHalfWaveEquations}), and \eqref{eq:InhomogeneousEstimateReductionI} follows from Sobolev's embedding.

\subsubsection{Reduction to a dyadic estimate}
We replace \eqref{eq:EstimateAwayFromSurface} by the stronger estimate
\begin{equation}
\label{eq:LrL2Estimate}
\| |D|^{-\rho} u \|_{L^p L^q} \lesssim \| u \|_{L^2} + \| P(x,D) u \|_{L^r L^2} + \| |D|^{-\frac{1}{2}} \rho_e \|_{L^2}
\end{equation}
with $1<r<2$, to use  a result from \cite{Tataru2002}. We now show that \eqref{eq:LrL2Estimate} follows from
\begin{equation}
\label{eq:DyadicLrL2Estimate}
\lambda^{-\rho} \| S_\lambda u \|_{L^p L^q} \lesssim \| S_\lambda u \|_{L^2} + \| P(x,D) S_\lambda u \|_{L^r L^2} +  \lambda^{-\frac{1}{2}}\|  S_\lambda\rho_e \|_{L^2}.
\end{equation}
By Littlewood--Paley theory, here one only has to prove the commutator bound
\begin{equation*}
\sum_{\lambda \in 2^{\mathbb{N}_0}} \| [P,S_\lambda] u \|^2_{L^r L^2} \lesssim \| u \|^2_{L^2}.
\end{equation*}
We rewrite $P$ in non-divergence form, where the error terms are easily estimated in $L^2$. It thus suffices to prove
\begin{equation*}
\sum_{ \lambda \geq 1} \| [\varepsilon_{ij},S_\lambda] \partial_j u \|^2_{L^r L^2} \lesssim \| u \|^2_{L^2},
\end{equation*}
which is \cite[Equ.~(3.8)]{Tataru2002}. At last, similar as in Paragraph \ref{subsubsection:ReductionCubeLargeFrequencies} the inequality \eqref{eq:DyadicLrL2Estimate} is replaced by
\begin{equation}
\label{eq:DyadicL1L2Estimate}
\lambda^{-\rho } \| S_\lambda u \|_{L^p L^q} \lesssim \| S_\lambda u \|_{L^\infty L^2} + \| P S_\lambda u \|_{L^1 L^2} +  \lambda^{-\frac{1}{2}}\|  S_\lambda\rho_e \|_{L^2}.
\end{equation}

\subsubsection{Truncating the coefficients at frequencies $\lambda^{\frac{1}{2}}$} 
Finally, the Fourier coefficients of $\tilde{\varepsilon}$ are truncated to the region 
$\{ | \xi | \lesssim \lambda^{1/2} \}$. For $\lambda > 0$, let
\begin{equation}\label{eq:U}
U_\lambda = \sum_{j \geq 0} S_{2^j \lambda}.
\end{equation}
We estimate the contribution of $U_{c\lambda^{1/2}} \varepsilon_{ij} =: \varepsilon_{ij}^{\geq c\lambda^{\frac{1}{2}} }$ in $P$,
with $c\le 1$. To pass to the summand $\| P S_\lambda u \|_{L^1 L^2}$, we compute
\begin{align*}
\| \partial_k (\varepsilon_{ij}^{\ge c\lambda^{\frac{1}{2}}} S_\lambda u) \|_{L^1 L^2} &\le \| (\partial_k \varepsilon_{ij}^{\geq c \lambda^{\frac{1}{2}}} ) S_\lambda u \|_{L^1 L^2} + \| \varepsilon^{\geq c \lambda^{\frac{1}{2}}}_{ij} \partial_k S_\lambda u \|_{L^1 L^2} \\
&\lesssim \| \partial_x \varepsilon_{ij}^{\geq c \lambda^{\frac{1}{2}}} \|_{L^1 L^\infty} \| S_\lambda u \|_{L^\infty L^2} + \lambda \| \varepsilon_{ij}^{\geq c \lambda^{\frac{1}{2}}} \|_{L^1 L^\infty} \| S_\lambda u \|_{L^\infty L^2}\\
&\lesssim \| S_\lambda u \|_{L^\infty L^2}
\end{align*}
by means of the assumptions on $\varepsilon$. Hence, \eqref{eq:DyadicL1L2Estimate} is a consequence of 
\begin{equation}
\label{eq:DyadicL1L2EstimateTruncatedCoefficients}
\lambda^{-\rho} \| S_\lambda u \|_{L^p L^q} \lesssim \| S_\lambda u \|_{L^\infty L^2} + \| P^\lambda S_\lambda u \|_{L^1 L^2} +  \lambda^{-\frac{1}{2}}\| S_\lambda\rho_e \|_{L^2},
\end{equation}
where $P^\lambda$ denotes $P$ with Fourier-truncated $\varepsilon$. We shall usually drop the superscript to lighten the notation.

By $L^2$-wellposedness, the energy inequality and the estimate away from the characteristic surface, 
similar as above we see  that it is enough to prove
\begin{equation*}
\lambda^{-\rho} \| S_\lambda u \|_{L^p L^q} \lesssim \| S_\lambda u \|_{L^\infty L^2} + \| P^\lambda S_\lambda u \|_{L^2} +  \lambda^{-\frac{1}{2}}\|  S_\lambda \rho_e \|_{L^2},
\end{equation*}
which is \eqref{eq:DyadicEstimateL1LinfCoefficients}. As in the second part of Paragraph~\ref{subsubsection:ReductionHalfWaveEquations}
and using  Proposition~\ref{prop:MDN-P}, we can now reduce Theorem \ref{thm:StrichartzEstimatesL1LinfCoefficients} to \eqref{eq:StrichartzEstimatesHalfWaveEquationL1LinfCoefficients} as stated in Proposition \ref{prop:HalfWaveEstimates}.

\section{Proof of the Half-wave estimate}
\label{section:ProofHalfWaveEstimates}
This section is devoted to the proof of Proposition \ref{prop:HalfWaveEstimates}.
We follow the strategy of \cite{Tataru2001,Tataru2002} to establish the estimates
\begin{align}
\label{eq:HalfWaveI}
\lambda^{-\rho} \| S_\lambda u \|_{L^p L^q} &\lesssim \| S_\lambda u \|_{L^2} + \| Q(x,D) S_\lambda u \|_{L^2}, \\
\label{eq:HalfWaveII}
\lambda^{-\rho} \| S_\lambda u \|_{L^p L^q} &\lesssim \| S_\lambda u \|_{L^\infty L^2} + \| Q(x,D) S_\lambda u \|_{L^2},
\end{align}
where
\begin{equation*}
Q(x,D) = Op(q(x,\xi)), \quad q(x,\xi) = \xi_0 - \big(\tilde{\varepsilon}^{ij}_{\lambda^{\frac{1}{2}}} \xi_i \xi_j \big)^{1/2}.
\end{equation*}
Furthermore, as pointed out in the previous section, we can suppose that $u$ has space-time Fourier transform in $\{ |\xi_0| \lesssim |(\xi_1,\xi_2)| \}$ and is essentially supported in the unit cube. For estimate \eqref{eq:HalfWaveI}, we suppose that $\| \partial^2_x \tilde{\varepsilon} \|_{L^\infty} \lesssim 1$ and for \eqref{eq:HalfWaveII} we suppose that $\| \partial^2_x \tilde{\varepsilon} \|_{L^1 L^\infty} \lesssim 1$. We start with the proof of \eqref{eq:HalfWaveI}.

\subsection{Proof for $C^2$-coefficients}

\subsubsection{Reduction to a neighborhood of the characteristic surface}
Let $v_\lambda = T_\lambda S_\lambda u$. The map $v_\lambda$ is concentrated in the region
\begin{equation*}
U = \{ |x| \leq 2, \quad \frac{1}{4} \leq |\xi| \leq 4 \}.
\end{equation*}
More precisely, we find 
\begin{equation}\label{eq:v-loc}
\| v_\lambda \|_{L^2_{\Phi}(U^c)} \lesssim e^{-c \lambda} \| S_\lambda u \|_{L^2},
\end{equation}
see \cite[p.~397]{Tataru2001}. Hence, it suffices to obtain estimates for $v_\lambda$ in $U$.

For such $v_\lambda$, Theorem \ref{thm:ApproximationTheorem} with $s=2$  yields
\begin{align}
\label{eq:ApproximationI}
v_\lambda \in L^2_\Phi, &\qquad (\lambda q + 2 (\bar{\partial} q)(\partial - i \lambda \xi))v_\lambda \in L^2_\Phi \\
\lambda^{-1/2} (\partial_\xi - \lambda \xi) v_\lambda \in L^2_\Phi, &\qquad \lambda^{-1/2} (\partial_\xi - \lambda \xi) (\lambda q + 2(\bar{\partial} q)(\partial - i \lambda \xi))v_\lambda \in L^2_\Phi, \label{eq:ApproximationII}
\end{align} 
 cf.\ \cite[Eq.~(23), (24)]{Tataru2001} and \eqref{eq:a^2}.
In the above display $f \in L^2_\Phi$ means that $f$ is uniformly bounded in $\lambda$ in terms of the right-hand side of \eqref{eq:HalfWaveI}.

We take the second estimate from \eqref{eq:ApproximationI} and the first from \eqref{eq:ApproximationII} to infer 
\begin{equation*}
\lambda^{\frac{1}{2}} q v_\lambda \in L^2_\Phi,
\end{equation*}
using that $i(\partial_\xi- \lambda\xi)=\partial-i\lambda \xi$ by holomorphy.
Consequently, we can suppose that $v_\lambda$ is supported in a small neighbourhood of $K= \{ q=0\}$. Away from the characteristic set, the gain of $\lambda^{\frac{1}{2}}$ is enough to conclude the reduced estimate by Sobolev's embedding.

As observed in \cite{Tataru2001}, we can replace $q$ by a $C^1$-multiple of it using the bound 
\begin{equation*}
q (\partial - i \lambda \xi) v_\lambda \in L^2_\Phi,
\end{equation*}
see \cite[Eq.~(26)]{Tataru2001}. It is precisely this computation by which in \cite{Tataru2001}
a wave symbol like $p(x,\xi) = \xi^2_0 - \tilde\varepsilon^{ij}(x) \xi_i \xi_j$ can be replaced by one of the form
\begin{equation*}
q(x,\xi) = \xi_0 - t(x,\xi^\prime),
\end{equation*}
where $t$ is $1$-homogeneous in $\xi^\prime$. This achieved by factorizing $p$ into the symbols of two 
half-wave equations and considering separated neighborhoods of $K \cap U$. In the present context we can 
thus adopt the arguments and use several results from   \cite{Tataru2001}.

For the sake of completeness, we sketch the reasoning in the following. 
Since $v_\lambda$ is holomorphic, estimate \eqref{eq:ApproximationI} yields that
\begin{align*}
[i(q_x \partial_\xi - q_\xi \partial_x) + \lambda(q-i \xi \cdot q_x - \xi \cdot q_\xi)] v_\lambda &\in L^2_\Phi, \\
[(q_x \partial_x + q_\xi \partial_\xi) + \lambda ( q - \xi \cdot q_\xi - i \xi \cdot q_x)] v_\lambda &\in L^2_\Phi;
\end{align*}
see \cite[p.~398]{Tataru2001}. For $w = \Phi^{1/2} v_\lambda$ we deduce
\begin{align}
[(q_x \partial_\xi - q_\xi \partial_x) - i \lambda (q- \xi \cdot q_\xi )] w &\in L^2, \label{eq:ODEHamiltonFlow}\\
[(q_x \partial_x + q_\xi \partial_\xi) + \lambda ( q - i\xi \cdot q_x ] w &\in L^2. \label{eq:ODEGradientFlow}
\end{align}
The first equation is an ODE along the Hamiltonian flow, which is used to derive estimates on the cone; the second equation is an ODE along the gradient curves, which is needed to obtain estimates away from the cone.

\subsubsection{Estimates on the cone}\label{subsec:est-come}
The first bounds in  \eqref{eq:ApproximationI} and \eqref{eq:ApproximationII}  translate to $w$ as
\begin{equation*}
w \in L^2, \quad \lambda^{-\frac{1}{2}} \partial_\xi w \in L^2.
\end{equation*}
The trace theorem then implies 
\begin{equation*}
\lambda^{-\frac{1}{4}} w \in L^2(K \cap U).
\end{equation*}
 Starting from \eqref{eq:ApproximationII}, we also  find
\begin{equation*}
\lambda^{-\frac{1}{2}} \partial_\xi [(q_x \partial_\xi - q_\xi \partial_x) - i \lambda (q-q_\xi \cdot \xi)] w \in L^2,
\end{equation*}
Note that the scalar function $q-q_\xi \cdot \xi$ vanishes for $1$-homogeneous $q$. Again by the trace theorem 
we obtain
\begin{equation*}
\lambda^{-\frac{1}{4}} H_q w:=\lambda^{-\frac{1}{4}} (q_x \partial_\xi - q_\xi \partial_x) w   \in L^2(K\cap U)
\end{equation*}
Only these $L^2$ estimates for $w$ and $H_q w$ are  used later on, cf.\ \cite[Eq.~(30), (31)]{Tataru2001}.

\subsubsection{Estimates away from the cone}
To simplify the analysis, we replace $q$ by the distance function $r$ to $K$ which solves the eikonal equation
\begin{equation*}
| \nabla_{x, \xi} r | = 1, \quad r = 0 \text{ \ in } K.
\end{equation*}
$r$ is a $C^1$-function, which yields a $C^1$-diffeomorphism $K \cap U \times (-\varepsilon,\varepsilon) \to \tilde{U}$, where $\tilde{U}$ is a neighborhood of $K \cap U$. 
It is the inverse of
\begin{equation*}
(x_0,\xi_0,r) \to (x,\xi ) = (x_0,\xi_0) + r \frac{\nabla_{x,\xi} q (x_0,\xi_0)}{|\nabla_{x,\xi} q (x_0,\xi_0)|},
\end{equation*}
which is a local diffeomorphism from $K \cap U \times [-\varepsilon,\varepsilon]$ onto a neighborhood of $K$,
see \cite[Eq.~(32)]{Tataru2001}. 

As seen on p.~400 of \cite{Tataru2001}, the quotient $r/q$ is  Lipschitz so that we can replace $q$ by $r$ in \eqref{eq:ODEGradientFlow}.
We introduce $C^2$-coordinates on $K \cap U$, denoted by $\zeta$. Hence $(\zeta,r)$ are new coordinates in $U$ near $K$.
In these coordinates, equation \eqref{eq:ODEGradientFlow} becomes
\begin{equation*}
Jw:=[\partial_r + \lambda(r-i r_x(\zeta)(\xi(\zeta)+rr_\xi(\zeta)))] w = f \in L^2.
\end{equation*}
Here we consider $\xi$, $x$, $r_x$ and $r_\xi$ as functions of  $\zeta\in K\cap U$, which is partly suppressed below. 
We split this into the homogeneous and inhomogeneous problems
\begin{align*}
Jw_1&:=f \in L^2, \qquad w_1=0 \quad \text{on \ } K,\\
Jw_2&:=f \in L^2, \qquad w_2=w \quad \text{on \ } K,
\end{align*}
cf.\ \cite[Eq.~(34), (35)]{Tataru2001}. In \cite[Eq.~(36)]{Tataru2001} it was shown that
\begin{equation*}
\lambda^{\frac{1}{2}} \| w_1 \|_{L^2} \lesssim \| f \|_{L^2},
\end{equation*}
and hence $w_2\in L^2$. 

Below  we use the transformation $dx d\xi = h(r,\zeta) dr d\zeta$, where $h$ is strictly positive,  $h(0,\zeta) = 1$ 
and  $d \zeta$ denotes the Lebesgue measure on $K\cap U$. The function $\tilde{w}_2 = (1-\frac1{h})w_2$ solves
$J\tilde{w}_2= (\partial_r\frac1{h})w_2\in L^2$ and  $\tilde{w}_2=0$ on $K$. Hence, $\tilde{w}_2$ can be estimated as $w_1$. 
Corresponding to  $w=(w_1 + \tilde{w}_2) + \frac1h w_2$, we split $S_\lambda u = u_1 + u_2$; i.e., $u_2=T_\lambda^*(\Phi^\frac12 \frac1h w_2)$.
The above estimate yields $\| u_1 \|_{L^2} = O(\lambda^{-1/2})$, and thus the claim for $u_1$ follows from Sobolev's embedding. Passing to $\frac{1}{h} w_2$ normalizes $h$ in the following computations.

\subsubsection{Reduction to oscillatory integral estimates}\label{subsec:halfwave-C2}
It remains to analyze $u_2$.
Using the ODE for $w_2$, we can write
\begin{align*}
u_2(y) \!&= c_n\lambda^{\frac{3(n+1)}4}\!\! \int\!\! e^{-\frac{\lambda}{2}(y-x-rr_x-i(\xi+rr_\xi))^2}\! e^{-\frac{\lambda}{2}(\xi + r r_\xi)^2} \!e^{-\frac{\lambda}{2} r^2}\! e^{i \lambda(rr_x \xi + \frac{1}{2} r^2 r_x r_\xi)} w(\zeta) dr d\zeta \\
&= c_n\lambda^{-\frac12} \lambda^{\frac{3(n+1)}{4}} \int_K e^{i \lambda \xi (y-x)} e^{- \frac{\lambda}{2} \omega_\zeta(y-x)^2} a(\zeta) w(\zeta) d\zeta,
\end{align*}
see  \cite[Eq.~(38)]{Tataru2001}, where
\begin{equation*}
\omega_\zeta(y-x) = (y-x)^2 - \frac{[(r_x+i r_\xi) \cdot (y-x)]^2}{r_\xi^2 + 2 r_x^2 + i r_x \cdot r_\xi}, \quad a(\zeta) = c (r_\xi^2 + 2r_x^2 + i r_x r_\xi)^{-1/2}.
\end{equation*}
As noted on p.~402 of \cite{Tataru2001}, the coefficients of the quadratic form $\omega$  are continuous in $x$
and smooth in $\xi$, and we have $\Re \omega > 0$. 

Because of Paragraph~\ref{subsec:est-come}, it suffices to show
\begin{equation}
\label{eq:OscillatoryIntegralEstimateI}
\| V_\lambda w \|_{L^p L^q} \lesssim \lambda^{\rho + 1/4}(\| w \|_{L^2(K)} + \| H_q w \|_{L^2(K)} )
\end{equation}
for all $w$ supported in $K \cap U$, where
\begin{equation*}
V_\lambda w = \lambda^{\frac{3(n+1)}{4}} \int_K e^{i \lambda \xi(y-x)} e^{- \frac{\lambda}{2} \omega_\zeta(y-x)^2} w(\zeta) d\zeta.
\end{equation*}
The oscillatory integral estimate \eqref{eq:OscillatoryIntegralEstimateI} is proved in Theorem 6 of \cite{Tataru2001} 
for symbols $q$ which  are $1$-homogeneous in $\xi'$ and  have the form
\begin{equation}
\label{eq:HalfWaveSymbol}
q(x,\xi) = \xi_0 - t(x,\xi^\prime)
\end{equation}
using $C^1$-equivalence, which is precisely the present concern (see \eqref{eq:ExplicitHWS}). We state this theorem. 
It involves the  Hamilton flow $(x_t,\xi_t)$ for $q$ starting at $(x,\xi)$, i.e.,
\begin{equation}
\label{eq:HamiltonFlow}
\left\{
\begin{array}{cl}
\partial_t x_t &= q_\xi(x_t,\xi_t), \quad x(0) = x,\\
\partial_t \xi_t &= -q_{x}(x_t,\xi_t), \quad \xi(0) = \xi.
\end{array} \right.
\end{equation}
\begin{proposition}
\label{prop:OscillatoryIntegralEstimate}
Let $V_\lambda$ as above and $b(x,\xi)$ be a smooth compactly supported function, which vanishes near the origin and is $1$ in $\{ 1/4 \leq |\xi| \leq 4\}$. Let $L$ denote the pseudo-differential operator along the Hamilton flow given by
\begin{equation*}
L w(x,\xi) = \int_0^\infty e^{-t} w(x_t,\xi_t) dt.
\end{equation*}
(One has $L=(H_q +1)^{-1}$.) Then,
\begin{equation*}
\| V_\lambda b(x,\xi) L \|_{L^2(K) \to L^p L^q} \lesssim \lambda^{\rho + 1/4}.
\end{equation*}
\end{proposition}
We will not revisit in detail the technical complex interpolation argument from \cite[p.~402--408]{Tataru2001}, which is in fact carried out for phase functions as in \eqref{eq:HalfWaveSymbol}. The final ingredient to finish the proof is to estimate the kernel $H^t$ of the operator
\begin{equation*}
Z^t = V_\lambda b F^t \delta_{q(x,\xi) = 0} a V_\lambda^*,
\end{equation*}
where $F^t$ denotes the translation by $t$ along the Hamilton flow. Such kernel bounds are provided by Theorem~7 of \cite{Tataru2001} which  we recall here.
\begin{proposition}\label{prop:H}
The kernels $H^t$ satisfy
\begin{equation*}
|H^t(y,\tilde{y})| \lesssim \lambda^{n+1} e^{-c \lambda(\tilde{y}_0 - y_0 -t)^2} (1+\lambda |y-\tilde{y}|)^{- \frac{n-1}{2}}.
\end{equation*}
\end{proposition}
To derive this result, the Hamilton flow has to be analyzed.
\subsubsection{The regularity of the Hamilton flow}
We revisit the proof of the regularity of the Hamilton flow. We note that the Fourier truncation yields
\begin{equation*}
| \partial_x^\alpha \tilde\varepsilon | \leq c_\alpha \lambda^{\frac{|\alpha|-2}{2}}, \quad |\alpha| \geq 2.
\end{equation*}
A word of clarification regarding the following estimates: Seemingly, the half wave symbol
\begin{equation}
\label{eq:ExplicitHWS}
q(x,\xi) = \xi_0 - (\tilde{\varepsilon}^{ij}_{\lambda^{\frac{1}{2}}}(x) \xi_i \xi_j)^{1/2}
\end{equation}
is less regular than the wave symbol
\begin{equation*}
p(x,\xi) = \xi_0^2 - \tilde{\varepsilon}^{ij}_{\lambda^{\frac{1}{2}}}(x) \xi_i \xi_j.
\end{equation*}
For the following estimates, however, it suffices to estimate the Hamilton flow in a small neighborhood $N$ 
of $K \cap U$ for times $t \in [0,1]$. In this region we have $|\xi^\prime| \sim 1$ and $|\xi_0| \lesssim 1$. It follows that in this range of $(\xi_0,\xi^\prime)$ the Hamilton flow for $H_q$ satisfies the same estimates as for $H_{p}$. This is again reflected through the $C^1$-equivalence of $p$ and $q$. In fact,
\begin{equation*}
\| H_p w \|_{L^2(K\cap U)} \approx \| H_{q} w \|_{L^2(K\cap U)}.
\end{equation*}
We record Lemmas~9 and 10 of \cite{Tataru2001} that estimate the flow and give an expansion.
\begin{lemma}
\label{lem:HamiltonFlowI}
In $N$, the Hamilton flow generated by $q$ from \eqref{eq:ExplicitHWS} satisfies
\begin{align*}
|\partial_\xi^\alpha x_t | &\leq c_\alpha t (1+t \sqrt{\lambda})^{|\alpha| - 1}, \quad |\alpha| \geq 1, \\
|\partial_\xi^\alpha \xi_t| &\leq c_\alpha (1 + t \sqrt{\lambda})^{|\alpha| - 1}, \quad |\alpha| \geq 1.
\end{align*}
\end{lemma}
\begin{lemma}
\label{lem:HamiltonFlowII}
For the Hamilton flow in $N$ we have the representation
\begin{align*}
x_t &= x + t q_\xi + t^2 g(t,x,\xi), \\
\xi_t &= \xi + th(t,x,\xi),
\end{align*}
where $g$ and $h$ are bounded by
\begin{equation*}
|\partial_\xi^\alpha h (t,x,\xi)|, \; |\partial_\xi^\alpha g(t,x,\xi)| \leq c_\alpha (1+ t \sqrt{\lambda})^{|\alpha| - 1},\quad |\alpha| \geq 1.
\end{equation*}
\end{lemma}
By the above, the kernel estimate in Proposition~\ref{prop:H} can be shown by (non)-stationary phase arguments. For the details we refer to \cite[p.~412-415]{Tataru2001}. This finishes the proof of \eqref{eq:HalfWaveI}  and thus of  Theorem \ref{thm:StrichartzEstimatesMaxwell2d}.

\subsection{Proof for $L^1 L^\infty$ - coefficients}
Next, we show that the estimates remain true if the assumption $\varepsilon \in C^2$ is replaced by 
$\partial_x^2 \varepsilon \in L^1 L^\infty$. We start with error estimates for conjugating with the FBI transform
taken from Theorem~2.3 of \cite{Tataru2002}.
\begin{theorem}\label{thm:commut2}
Let $\lambda \gtrsim 1$ and $\partial_x^2 \varepsilon\in L^1 L^\infty$. Set $a(x,\xi ) = \big( \tilde{\varepsilon}^{ij}_{\lambda^{1/2}}(x) \xi_i \xi_j \big)^{1/2}s_\lambda(\xi)$. Then, we find the following estimates to hold:
\begin{align*}
\| \Phi^{1/2} R_{\lambda,a} \|_{L^\infty L^2 \to L^2} &\lesssim \lambda^{-3/4}, \\
\| \Phi^{1/2} R_{\lambda,a} \|_{L^\infty L^2 \to L^1_{x_0} L^2_{x^\prime \xi^\prime}(K)} &\lesssim \lambda^{-3/4}.
\end{align*}
\end{theorem}
(On $K$ one uses its surface measure $d\zeta$.)
We turn to the proof of \eqref{eq:HalfWaveII} with the notation given there.
We again use the FBI transform and write
\begin{equation*}
w = \Phi^{1/2} T_\lambda S_\lambda u , \qquad S_\lambda u = T_\lambda^* \Phi^{1/2} w.
\end{equation*}
As in the previous subsection, $w$ is essentially supported in $U$, see \eqref{eq:v-loc}.
Recalling \eqref{eq:a^2} and \eqref {eq:Remainder}, we have 
\begin{equation*}
(\lambda q + 2(\bar{\partial} q) (\partial - i \lambda \xi)) \Phi^{-1/2} w = \Phi^{-1/2} g
\end{equation*}
with
\begin{equation*}
g = \Phi^{1/2} ( \lambda R_{\lambda,q} S_\lambda u + T_\lambda Q(x,D) S_\lambda u).
\end{equation*}
The error estimates from Theorem~\ref{thm:commut2} allow us to bound $g$ in terms of the right-hand side 
of \eqref{eq:HalfWaveII}. This is stated in the next result which is Lemma~3.1 of \cite{Tataru2002}.
\begin{lemma} With the above notation we have 
\begin{align*}
\| g \|_{L^2} &\lesssim \lambda^{1/4} \| S_\lambda u \|_{L^\infty L^2} + \| Q(x,D) S_\lambda u \|_{L^2}, \\
\| g \vert_K \|_{L^1_{x_0} L^2_{x^\prime \xi^\prime}} &\lesssim \lambda^{1/4} \| S_\lambda u \|_{L^\infty L^2} + \| Q(x,D) S_\lambda u \|_{L^2}. 
\end{align*}
\end{lemma}

\subsubsection{Gradient and Hamilton flow equations}
As in \eqref{eq:ODEHamiltonFlow}  and \eqref{eq:ODEGradientFlow}, we find the equations
\begin{align}
\label{eq:HamiltonFlowII}
[(q_x \partial_\xi - q_\xi \partial_x) - i \lambda ( q - \xi \cdot q_\xi)] w &= -ig, \\
\label{eq:GradientFlowII}
[(q_x \partial_x + q_\xi \partial_\xi) + \lambda ( q-i \xi \cdot q_x)] w &= g,
\end{align}
see \cite[Eq.~(3.15), (3.16)]{Tataru2002}. Again, the first equation is an ODE along the Hamilton flow of $q$, while the second is an ODE along the gradient curves of $q$. Moreover, \eqref{eq:HamiltonFlowII} is used to derive estimates on the cone $K$, and \eqref{eq:GradientFlowII} off the cone. The $L^1 L^\infty$-bound on $\partial^2_x\varepsilon$ is needed for the analysis of the Hamilton flow. We first sketch the estimate away from $K$, which works for Lipschitz coefficients. 

Taking the inner product with $qw$ of \eqref{eq:GradientFlowII} and integrating by parts, we obtain
\begin{equation*}
\lambda \| q w \|^2_{2} = \frac{1}{2} \langle (( |\nabla q |^2 + q \Delta q) w, w \rangle + \Re \langle qw, g \rangle.
\end{equation*}
The boundedness of $U$ and $\nabla q$ and the inequality $|\Delta q | \lesssim \lambda^{1/2}$ yield
\begin{equation*}
\lambda^{1/2} \| q w \|_2 \lesssim \| w \|_2 + \lambda^{-1/2} \| g \|_2.
\end{equation*}
Away from the cone, i.e., $|q| \geq c$,  this inequality gains half of a derivative allowing us to prove Strichartz estimates by Sobolev's embedding as before. In the following we thus suppose that $w$ and $g$ are supported in a neighborhood $N$ of $K\cap U$.

\subsubsection{Gradient flow decomposition}
We use the gradient flow equation to decompose $w$ into two parts,
\begin{equation*}
w = w_1 + w_2,
\end{equation*}
where $w_j$ solves the inhomogeneous and the homogenous equations
\begin{align}
\label{eq:Definitionw1}
[(q_x \partial_x + q_\xi \partial_\xi) + \lambda (q - i q_x \cdot \xi)] w_1 &= -ig, \quad w_1 \vert_K = 0, \\
\label{eq:Definitionw2}
[(q_x \partial_x + q_\xi \partial_\xi) + \lambda (q - i q_x \cdot \xi)] w_2 &= 0, \quad w_2 \vert_K = w.
\end{align}
We split $S_\lambda u$ correspondingly, i.e.,
\begin{equation*}
u_i = T_\lambda^* \Phi^{1/2} b(x,\xi) w_i,
\end{equation*}
where $b$ is a smooth cut-off for $N$.  The estimate for $w_1$ follows as on p.~431 of \cite{Tataru2002}.

To estimate $w_2$, we have to analyze the regularity of the gradient flow. Let $(x,\xi)$ be initial data  on $K\cap U=\{ q=0 \} \cap U$. The flow $(x_t,\xi_t)$ is given by
\begin{equation*}
\left\{ \begin{array}{cl}
\partial x_t &= q_x(x_t,\xi_t), \quad x(0) = x, \\
\partial \xi_t &= q_\xi(x_t,\xi_t), \quad \xi(0) = \xi.
\end{array} \right.
\end{equation*}

Due to the support of $w$, it is enough to analyze the regularity of the gradient flow in $U$. Here we have the same estimates as in Theorem~3.2 in \cite{Tataru2002}.
\begin{theorem}
For $q(x,\xi)$ given by \eqref{eq:HalfWaveSymbol},  the following estimates hold on $U$:
\begin{align*}
|\partial_x^\alpha \partial_\xi^\beta x_t| &\leq c_{\alpha, \beta} \lambda^{\frac{|\alpha|-1}{2}} e^{c_{\alpha,\beta} \sqrt{\lambda} |t|}, \qquad |\alpha| + |\beta| > 0,\\
|\partial_x^\alpha \partial_\xi^{\beta} (\xi_t-\xi)| &\leq c_{\alpha,\beta} \lambda^{\frac{|\alpha|-1}{2}} e^{c_{\alpha,\beta} \sqrt{\lambda} |t|}.
\end{align*}
\end{theorem}
The proof from \cite{Tataru2002} for the symbol $p$ works also for $q$ since in $U$ the derivatives of $q$ and $p$ satisfy the same bounds.

\subsubsection{Reduction to oscillatory integral estimate}
 In the next result $u_2$ is expressed in terms of trace of $w$ on the cone.
\begin{proposition}
Assume $q$ is of the form \eqref{eq:HalfWaveSymbol}. Then, we have
\begin{equation*}
u_2 = \lambda^{-1/2} V_\lambda w \vert_K,
\end{equation*}
where $V_\lambda$ is the integral operator
\begin{equation*}
V_\lambda w = \lambda^{\frac{3(n+1)}{4}} \int_K e^{i \lambda \xi (x-y)} G(x,y,\xi) w dx d\xi
\end{equation*}
with $G$ satisfying
\begin{equation*}
|\partial_x^\alpha \partial_\xi^\beta G(x,y,\xi)| \leq c_{\alpha,\beta} \lambda^{\frac{|\alpha|}{2}} e^{-c \lambda(x-y)^2}.
\end{equation*}
\end{proposition}
The proof of the corresponding Theorem~3.3 in \cite{Tataru2002} is based on the estimates from the previous lemma. Hence, the argument 
can be transfered to the half-wave symbol, as long as the estimates for the gradient flow apply. This is guaranteed by the support condition of $w_2$. The estimate for $u_2$ is thus reduced to the oscillatory integral bound
\begin{equation*}
\| V_\lambda w \|_{L^p L^q} \lesssim \lambda^{\rho + 1/4} \| H_q w \|_{L^1_{x_0} L^2_{x^\prime \xi^\prime}(K)}
\end{equation*}
for $w$ supported in $K \cap U$.

Such an estimate is proved in Theorem~3.4 of \cite{Tataru2002}. The analysis from \cite{Tataru2002} applies due to the Lipschitz equivalence of $p$ and $q$ on $U$. We record this result.
\begin{theorem}
Let $a(x,\xi)$ be a smooth compactly supported function, which is $0$ near $\xi = 0$ and $1$ in $\{1/4 \leq |\xi^\prime| \leq 4, \; |\xi_0| \leq 4 \}$. Then,
\begin{equation*}
\| V_\lambda a(x,\xi) L \|_{L^2(K \cap \{x_0 = 0 \}) \to L^p L^q} \lesssim \lambda^{\rho + 1/4},
\end{equation*}
where $L$ is the transport operator along the Hamilton flow given by
\begin{equation*}
(Lw)(x,\xi) = 
\begin{cases}
0, \quad \text{ if } x_0 < 0,\\
w(x_t,\xi_t), \quad \text{ if } x_{t0} = 0, x_0 \geq 0.
\end{cases}
\end{equation*}
\end{theorem}
The above estimate is a consequence of the regularity of the Hamilton flow. As proved in \cite[p.~434--436]{Tataru2002}, the estimates for $C^2$ coefficients stated in Lemma \ref{lem:HamiltonFlowI} and \ref{lem:HamiltonFlowII} remain valid. As pointed out above, these estimates are a consequence of estimates for the derivatives, which are equivalent for the wave and half-wave symbol in a suitable neighbourhood of the cone away from the origin. This finishes the proof of \eqref{eq:HalfWaveII} and thus of  Theorem \ref{thm:StrichartzEstimatesL1LinfCoefficients}.

\section{Estimates for less regular coefficients}
\label{section:Consequences}
In this section we first prove the weaker estimates stated in Theorems \ref{thm:StrichartzEstimatesMaxwell2dCs} and \ref{thm:MaxwellStrichartzEstimatesBesovRegularity} for $\varepsilon$ without a second derivative. 
Since the results are clear for $s=0$ by Sobolov's embedding, we assume that $s\in(0,2)$.

\begin{proof}[Proof of Theorem \ref{thm:StrichartzEstimatesMaxwell2dCs}]

We follow the argument from \cite{Tataru2001}. By rescaling we can assume that $\| \varepsilon^{ij}\|_{\dot{C}^s}\leq 1$ 
and $\kappa = 1$. Note this transfers to $\tilde\varepsilon$, up to a constant.  As in Paragraph~\ref{subsub:cube}, the low frequencies are estimated by Sobolev's embedding.
Here we need that $\|\partial_k(\varepsilon_{ij} S_0u)\|_{H^{-\sigma}}\lesssim\|u\|_{L^2} $, which is true since $1-\sigma<s$. 
Hence we will suppose that the Fourier transform of $u$ vanishes near the origin.  We can then localize $u$ to a 
cube of sidelength $1$. Indeed, take a smooth partition of unity $(\chi_j)_{j \in \Z^3}$,
localizing to cubes with sidelength $1$. Compared to Paragraph~\ref{subsub:cube}, the inequality
\begin{equation*}
\begin{split}
\sum_j& \big( \| \chi_j u \|_{L^2}^2 + \| P(x,D) (\chi_j u) \|_{H^{-\sigma}}^2 + \| \langle D' \rangle^{-\frac{1}{2}-\frac{\sigma}{2}} \rho_j \|_{L^2}^2 \big) \\
 &\lesssim \| u \|_{L^2}^2 + \| P(x,D) u \|^2_{H^{- \sigma}} + \| \langle D \rangle^{-\frac{1}{2}-\frac{\sigma}{2} } \rho_e \|_{L^2}^2
 \end{split}
\end{equation*}
needs a bit more care. The estimate for the first term is clear. The summands in the second look like
\begin{equation*}
\partial_i (\varepsilon_{kl} \chi_j u) = \chi_j \partial_i (\varepsilon_{kl} u) + (\partial_i \chi_j) \varepsilon_{kl} u.
\end{equation*}
Here the second term on the right-hand side can be estimated in $L^2$. For the first one we note that
\[ \sum_j \| \chi_j v\|^2_{H^{-\sigma}}\lesssim  \|v\|_{H^{-\sigma}}^2\]
follows by duality and interpolation from $\|\sum_j\chi_j u_j\|_{H^k}^2\lesssim\sum_j\|u_j\|_{H^k}^2$ for $k\in\{0,1\}$. The third term can be estimated as previously.

We next claim that \eqref{eq:StrichartzEstimatesMaxwell2dCs} is a consequence of the dyadic estimate
\begin{equation}
\label{eq:DyadicEstimateCsCoefficients}
\lambda^{- \rho } \lambda^{- \frac{\sigma}{2}} \| S_\lambda u \|_{L^p L^q} \lesssim \| S_\lambda u \|_{L^2} + \| P^\lambda S_\lambda u \|_{H^{-\sigma}} + \lambda^{-\frac{1}{2}-\frac{\sigma}{2}} \| S_\lambda \rho_e \|_{L^2},
\end{equation}
where $P^\lambda$ is the operator $P$ after Fourier-truncation of $\varepsilon_{ij}$ at $\nu = \lambda^{\frac{2}{2+s}}$.
To recover \eqref{eq:StrichartzEstimatesMaxwell2dCs} from the above display, it suffices to establish
\begin{equation*}
\sum_{\lambda \in 2^{\mathbb{N}_0}} \| (P^\lambda S_\lambda - S_\lambda P) u \|^2_{H^{-\sigma}} \lesssim \| u \|^2_{L^2}.
\end{equation*}
Observe that also the frequencies of $P^\lambda S_\lambda u$ are localized to $\lambda$, since the coefficients are 
truncated at $\nu \ll \lambda$. Taking out $\partial_k$, we thus have to show
\begin{equation*}
\sum_{\lambda \in 2^{\mathbb{N}_0}} \lambda^{2-2\sigma} \| (\varepsilon_{ij}^{\lesssim \nu} S_\lambda - S_\lambda \varepsilon_{ij}) u\|^2_{L^2} \lesssim \| u \|^2_{L^2}.
\end{equation*}
which follows from
\begin{align*}
\| [ \varepsilon_{ij}^{\lesssim \nu}, S_\lambda] v \|_{L^2} &\lesssim \lambda^{\sigma - 1-\delta} \| v \|_{L^2}, \\
\sum_{\lambda \in 2^{\mathbb{N}_0}} \lambda^{2-2\sigma} \| S_\lambda ( \varepsilon_{ij}^{\gtrsim \nu} v) \|_{L^2}^2 &\lesssim \| v \|_{L^2}^2.
\end{align*}
for some $\delta>0$. These estimates can be proved as on p.~417 of \cite{Tataru2001}  using  $1-\sigma<s$.

Finally, since the frequencies of $P^\lambda S_\lambda u$
are located at $\lambda$, inequality \eqref{eq:DyadicEstimateCsCoefficients} becomes
\begin{equation*}
\lambda^{- \rho} \| S_\lambda u \|_{L^p L^q} \lesssim \lambda^{\frac{\sigma}{2}} \| S_\lambda u \|_{L^2} + \lambda^{-\frac{\sigma}{2}} \| P^\lambda S_\lambda u \|_{L^2} + \lambda^{-\frac{1}{2}} \| S_\lambda \rho_e \|_{L^2}.
\end{equation*}
This is a special case of \eqref{eq:StrichartzEstimatesMaxwell2d} with $\kappa:= \lambda^{\sigma/2}$. 
Because the Fourier transform of the coefficients $\varepsilon_{ij}^{\lesssim \nu}$ of $P^\lambda$ are supported in $\{ | \xi| \leq \nu=\lambda^{\frac{2}{2+s}} \}$, we find
\begin{equation*}
\| \partial^2_x \hat\varepsilon^{ij}_{\lesssim \nu} \|_{L^\infty} 
\lesssim \| \partial^2_x \varepsilon_{ij}^{\lesssim \nu} \|_{L^\infty} 
\lesssim \nu^{2-s} \| \varepsilon_{ij} \|_{\dot{C}^s}
\lesssim \lambda^{2 \sigma} \| \varepsilon^{ij} \|_{\dot{C}^s} \lesssim \lambda^{2 \sigma}.
\end{equation*}
(Note that  $ (\hat\varepsilon^{ij}_{\lesssim \nu}):= (\varepsilon_{ij}^{\lesssim \nu})^{-1}$
enters the statement of  Theorem \ref{thm:StrichartzEstimatesMaxwell2d} for $P^\lambda$.)
Therefore, Theorem \ref{thm:StrichartzEstimatesMaxwell2d} with $\kappa= \lambda^{\sigma/2}$ yields \eqref{eq:DyadicEstimateCsCoefficients}, and Theorem \ref{thm:StrichartzEstimatesMaxwell2dCs} is proven.
\end{proof}

Next, we turn to the proof of Theorem \ref{thm:MaxwellStrichartzEstimatesBesovRegularity}.
\begin{proof}[Proof of Theorem \ref{thm:MaxwellStrichartzEstimatesBesovRegularity}]
To show Theorem \ref{thm:MaxwellStrichartzEstimatesBesovRegularity}, we note that the proof of Theorem \ref{thm:StrichartzEstimatesL1LinfCoefficients} yields moreover the estimate
\begin{equation*}
\begin{split}
\| |D|^{-\rho} u \|_{L^p(0,T;L^q)} &\lesssim\kappa^{\frac{1}{p}} \| u \|_{L^\infty L^2} +\kappa^{-\frac{1}{p'}} \| P(x,D) u \|_{L^1 L^2} \\
&\quad + T^{\frac{1}{2}} \big( \sum_{\lambda \geq 1} \lambda^{-1} \| S_\lambda \rho_e \|_{L^\infty L^2}^2 + \sum_{\lambda \geq 1} \lambda^{-1} \| S_\lambda \partial_t \rho_e \|_{L^1 L^2}^2 \big)^{\frac{1}{2}}.
\end{split}
\end{equation*}

By rescaling we can suppose that $T=1$ and $\| \varepsilon \|^2_{\mathcal{X}^s} \leq\kappa^{2+s}$. 
We have to transfer the condition on $\varepsilon$ to $\tilde\varepsilon$. For this we 
characterize $\mathcal{X}^s$ by differences as for usual Besov spaces, cf.\ Theorems~2.36 and 2.37 of  
\cite{BahouriCheminDanchin2011}. Then one can proceed as for $\varepsilon\in C^s$.
After reducing the claim to dyadic estimate and truncating the coefficients at frequency $\kappa^{\frac{1}{2}} \lambda^{\frac{2}{2+s}}$, we will infer the estimates from Theorem \ref{thm:StrichartzEstimatesL1LinfCoefficients}.

We start with a first frequency localization. The low frequencies $\lambda\le \max\{1,\kappa^{-\frac1\sigma}\!\}$ are estimated by Sobolev's embedding as above. For $\lambda \geq \max\{1,\kappa^{-\frac1\sigma}\}$, we truncate $\varepsilon_{ij}$ at frequencies $\lambda/16$ and let $P^\lambda$ denote $P$ with these coefficients $\varepsilon_{ij}^\lambda$. We claim that it suffices to prove the dyadic estimates
\begin{equation}
\label{eq:DyadicEstimateXs}
\begin{split}
\lambda^{-\rho - \frac{\sigma}{p}} \| S_\lambda u \|_{L^p L^q} &\lesssim\kappa^{\frac{1}{p}} \| S_\lambda u \|_{L^\infty L^2} + \lambda^{-\sigma}\kappa^{-\frac{1}{p^\prime}} \| P^\lambda S_\lambda u \|_{L^1 L^2} \\
&\quad + \lambda^{-\frac{1}{2}-\frac{\sigma}{p}} \| S_\lambda \rho_e \|_{L^\infty L^2} + \lambda^{-\frac{1}{2}-\frac{\sigma}{p}} \| S_\lambda \partial_t \rho_e \|_{L^1 L^2}.
\end{split}
\end{equation}
Indeed, this inequality implies \eqref{eq:GeneralBesovRegularityEstimates} provided that
\begin{equation*}
\| |D|^{-\sigma} (S_\lambda P - P^\lambda S_\lambda) u \|^2_{L^1 L^2} \lesssim\kappa^2 \| \tilde S_\lambda u \|^2_{L^\infty L^2}.
\end{equation*}
holds. Factoring out the derivatives, we thus have to show 
\begin{equation}\label{eq:DyadicEstimateXs-1}
\| (S_\lambda \varepsilon_{ij} - \varepsilon_{ij}^\lambda S_\lambda) u \|^2_{L^1 L^2} \lesssim\kappa^2 \lambda^{2(\sigma - 1)} \| \tilde S_\lambda u \|^2_{L^\infty L^2}.
\end{equation}
Because $u$ is compactly supported in time, it suffices to take  compactly supported $\varepsilon_{ij}$. Then, we know that
\begin{equation*}
\| \varepsilon_{ij} \|_{L^1 L^\infty} \lesssim 1, \quad \| \varepsilon_{ij} \|_{\mathcal{X}^s} \lesssim\kappa^{\frac{2+s}{2}}
\end{equation*}
(using also the boundedness of $\varepsilon_{ij}$),  
which gives
\begin{equation*}
\| S_\lambda \varepsilon_{ij} \|_{L^1 L^\infty} \lesssim \min(\kappa^{\frac{2+s}{2}} \lambda^{-s}, 1).
\end{equation*}
As $1-\sigma= 2s/(2+s)$, we infer
\begin{align*}
\| \varepsilon_{ij} \|_{\dot{B}^{1 \infty 1}_{1-\sigma}} &= \sum_{\lambda} \lambda^{1-\sigma} \| S_\lambda \varepsilon_{ij} \|_{L^1 L^\infty} \\
&\lesssim \sum_{\lambda \lesssim\kappa^{\frac{2+s}{2s}}} \lambda^{1-\sigma} + \sum_{\lambda \gtrsim\kappa^{\frac{2+s}{2s}}} \lambda^{1-\sigma}\kappa^{\frac{2+s}{2}} \lambda^{-s} \lesssim\kappa.
\end{align*}
At this point, \eqref{eq:DyadicEstimateXs-1} follows from inequality (4.4) of \cite{Tataru2002}.

Next, we further restrict the Fourier support of $\varepsilon^\lambda_{ij}$  to $\kappa^{\frac12} \lambda^{\frac{2}{2+s}}$.
This is possible since we can control the error in the second term in \eqref{eq:DyadicEstimateXs} by means of
\begin{align*}
\| \partial_x (U_{\kappa^{\frac{1}{2}} \lambda^{\frac{2}{2+s}}} \varepsilon^\lambda_{ij} ) S_\lambda u \|_{L^1 L^2} &\lesssim \lambda \| U_{\kappa^{\frac{1}{2}} \lambda^{\frac{2}{2+s}}} \varepsilon_{ij} \|_{L^1 L^\infty} \| S_\lambda u \|_{L^\infty L^2} \\
&\lesssim \lambda \big( \kappa^{\frac{1}{2}} \lambda^{\frac{2}{2+s}} \big)^{-s} \| \varepsilon_{ij} \|_{\mathcal{X}^s} \| S_\lambda u \|_{L^\infty L^2} \\
&\lesssim \kappa \lambda^{\sigma} \| S_\lambda u \|_{L^\infty L^2},
\end{align*}
where the ultimate inequality follows from the $\mathcal{X}^s$-bound for $\varepsilon$ and $U_\nu$ was defined in \eqref{eq:U}. To apply Theorem \ref{thm:StrichartzEstimatesL1LinfCoefficients}, we also note that
\begin{equation*}
\| \partial^2_x (S_{\lesssim\kappa^{\frac{1}{2}} \lambda^{\frac{2}{2+s}}} \varepsilon^\lambda_{ij})  \|_{L^1 L^\infty} \lesssim \kappa^{\frac{2-s}{2}} \lambda^{\frac{2(2-s)}{2+s}} \| \varepsilon_{ij} \|_{\mathcal{X}^s}
 \lesssim (\kappa \lambda^\sigma)^2,
\end{equation*}
where $\kappa\lambda^\sigma\ge1$.
Now, \eqref{eq:DyadicEstimateXs} is a consequence of Theorem \ref{thm:StrichartzEstimatesL1LinfCoefficients}. 
\end{proof}

Finally, we discuss how  the inhomogeneous estimates given in Theorem \ref{thm:MaxwellInhomogeneousStrichartz} 
can be deduced from Theorem \ref{thm:StrichartzEstimatesMaxwell2d} and the analysis in  \cite{Tataru2001}.
\begin{proof}[Proof of Theorem \ref{thm:MaxwellInhomogeneousStrichartz}]
By rescaling it is enough to consider the case
\begin{equation*}
\kappa = 1, \qquad \| \partial^2_x \varepsilon \|_\infty \leq 1.
\end{equation*}
Then we can modify the arguments in the proof of Theorem \ref{thm:StrichartzEstimatesMaxwell2d} to reduce 
the desired inequality to $u$ with support in a cube of size $1$,  to localize it 
to a dyadic estimate at frequency $\lambda\gtrsim 1$ and for the region $\{|\xi_0|\lesssim|\xi'|\}$, and
to truncate the coefficients at frequency $\lambda^{\frac{1}{2}}$. Since the diagonalization $P = \mathcal{M} \mathcal{D} \mathcal{N}$ essentially respects frequency localization and $L^p$-properties, 
we finally  reduce to the inequality
\begin{equation}
\label{eq:DyadicEstimateInhomogeneous}
\lambda^{-\rho} \| S_\lambda u \|_{L^p L^q} \lesssim \| S_\lambda u \|_{L^2} + \| \tilde{f}^\lambda_1 \|_{L^2} + \lambda^\rho \| \tilde{f}^{\lambda}_2 \|_{L^{p^\prime} L^{q^\prime}} + \lambda^{-\frac{1}{2}} \| S_\lambda \rho_e \|_{L^2},
\end{equation}
where $\mathcal{D}(x,D) S_\lambda u = \tilde{f}^\lambda_1 + \tilde{f}^\lambda_2$. The degenerate component of $\mathcal{D}$ 
is handled as in Paragraph~\ref{subsubsection:ReductionHalfWaveEquations}. We sketch the reduction to the above display. 
Using Proposition~\ref{prop:MDN-P}, we can write up to error terms of order $O_{L^2}(\lambda^{-N})$:\footnote{We suppress frequency localization and slightly enlarged variants between the pseudo-differential operators to lighten the notation.}
\begin{equation*}
P S_\lambda = \mathcal{M} \mathcal{D} \mathcal{N} S_\lambda + E_\lambda
\end{equation*}
with $\| E_\lambda \|_{L^2 \to L^2} \lesssim 1$. Let $v$ denote the original function and $S_\lambda \mathcal{N} S_\lambda v = \tilde{S}_\lambda u$. We let $\mathcal{M} S_\lambda \mathcal{D} S_\lambda \mathcal{N} S_\lambda v = f_1^\lambda + f_2^\lambda + g^\lambda$ with $\| g^\lambda \|_{L^2} \lesssim \| S_\lambda v \|_{L^2}$. This requires 
\begin{equation*}
\mathcal{M} \tilde{f}_1^\lambda = f_1^\lambda + g^\lambda, \qquad \mathcal{M} \tilde{f}_2^\lambda = f_2^\lambda.
\end{equation*}
Multiplication with $\mathcal{N}$ gives (cf. Proposition \ref{prop:KohnNirenberg})
\begin{equation*}
\tilde{f}_1^\lambda + E_1^\lambda \tilde{f}_1^\lambda = \mathcal{N} f_1^\lambda + \mathcal{N}g^\lambda, \quad \tilde{f}_2^\lambda + E_2^\lambda \tilde{f}_2^\lambda = \mathcal{N} f_2^\lambda,
\end{equation*}
where we have
\begin{equation*}
\| E_1^\lambda \|_{L^2 \to L^2} \lesssim \lambda^{-1} \quad \text{ and } \quad \| E_2^\lambda \|_{L^{p'} L^{q'} \to L^{p'} L^{q'}} \lesssim \lambda^{-1}
\end{equation*}
 by Lemma \ref{lem:L2BoundednessRoughSymbols}.
Hence, the equations for $\tilde{f}_1^\lambda$ and $\tilde{f}_2^\lambda$ can be solved in $L^2$ and $L^{p'} L^{q'}$, 
respectively, by the Neumann series, yielding the estimates 
\begin{equation*}
\| \tilde{f}_1^\lambda \|_{L^2} \lesssim \| f_1^\lambda \|_{L^2} + \| S_\lambda v \|_{L^2}, \quad \| \tilde{f}_2^\lambda \|_{L^{p'} L^{q'}} \lesssim \| f_2^\lambda \|_{L^{p'} L^{q'}}.
\end{equation*}
We abbreviate $\tilde{f}_i = \tilde{f}_i^\lambda$ in the following.

 As before, let
\begin{equation*}
v_\lambda = T_\lambda S_\lambda u_i, \ \ i =2,3, \quad \text{ and } \quad q(x,\xi) = \xi_0 - \big(\tilde \varepsilon^{ij}_{\lambda^{\frac{1}{2}}} (x) \xi_i \xi_j \big)^{1/2}.
\end{equation*}
An application of Theorem \ref{thm:ApproximationTheorem} shows that (cf.\ \cite[p.~418]{Tataru2001})
\begin{equation*}
v_{\lambda} \in L^2_{\Phi}, \qquad [\lambda q + 2 (\bar{\partial} q)(\partial - i \lambda \xi)] v_{\lambda} - T_{\lambda} \tilde{f}_2 \in L^2_{\Phi}.
\end{equation*}
Now set
\begin{equation*}
w = \Phi^{1/2} v_\lambda, \qquad g = \Phi^{1/2} T_\lambda \tilde{f}_2.
\end{equation*}
The properties of the ODE along the Hamilton flow yield
\begin{equation*}
\lambda^{-1/4} w \in L^2(K \cap U), \qquad \lambda^{-1/4} (H_q w - g) \in L^2(K \cap U).
\end{equation*}
As above, we use the ODE along the gradient flow of $q$ to decompose $w = w_1 + w_2$, where
\begin{align*}
[(q_x \partial_\xi + q_\xi \partial_x) + \lambda (q- \xi \cdot q_\xi )] w_1 &=f, \qquad w_1 \big|_{K} = 0,\\
[(q_x \partial_x + q_\xi \partial_\xi) + \lambda ( q - i\xi \cdot q_x ] w_2 &=0,  \qquad w_2 \big|_{K} = w.
\end{align*}
The function $w_2$ can further be decomposed into ``good" and ``bad" parts,
\begin{equation*}
w_2 = w_2^g + w_2^b,
\end{equation*}
whose traces on the cone $K \cap U$ satisfy
\begin{equation*}
\lambda^{-1/4} H_q w_2^g \in L^2, \qquad H_q w_2^b = g.
\end{equation*}
The map $w_2^g$ can be treated as $w_2$ in Paragraph~\ref{subsec:halfwave-C2}. In a similar way $w_2^b$ can be controlled
by $g$, and thus $\tilde{f}_2$, as indicated on p.~419 of \cite{Tataru2001}. Analogously, we decompose $w_1$ as
\begin{equation*}
w_1 = w_1^g + w_1^b,
\end{equation*}
and solve
\begin{equation*}
(q_x \partial_x + q_\xi \partial_\xi + \lambda(q-iq_x \cdot \xi)) w_1^g \in L^2,  \qquad w_1^g \big|_{K} = 0,
\end{equation*}
respectively,
\begin{equation*}
(q_x \partial_x + q_\xi \partial_\xi + \lambda(q-i q_x \cdot \xi)) w_1^b = g,  \qquad  w_1^b \big|_{K}=0.
\end{equation*}
Again, $w_1^g$ can be treated as in  Paragraph~\ref{subsec:halfwave-C2}. For the estimate of $w_1^b$ we refer to \cite[pp.~419--422]{Tataru2001}. The argument applies due to the Lipschitz equivalence of wave and half-wave symbol in the phase space region of interest, which was already discussed above.
\end{proof}

We next point out how homogeneous Strichartz estimates yield inhomogeneous Strichartz estimates. For this we invoke the 
following consequence of the Christ--Kiselev lemma (cf. \cite{ChristKiselev2001}).
\begin{lemma}[{\cite[Lemma~8.1]{HassellTaoWunsch2006}}]
\label{lem:ChristKiselev}
Let $X$ and $Y$ be Banach spaces and for all $s,t \in \R$ let $K(s,t): X \to Y$ be an operator-valued kernel from $X$ to $Y$. Suppose we have the estimate
\begin{equation*}
\| \int_{\R} K(s,t) f(s) ds \|_{L^q(\R,Y)} \leq A \| f \|_{L^p(\R,X)}
\end{equation*}
for some $A > 0$ and $1 \leq p < q \leq \infty$, and $f \in L^p(\R;X)$. Then, we have

\begin{equation*}
\| \int_{s<t} K(s,t) f(s) ds \|_{L^q(\R,Y)} \leq C_{p,q} A \| f \|_{L^p(\R,X)}.
\end{equation*}
\end{lemma}
We are ready for the proof of Corollary \ref{cor:InhomogeneousStrichartzC1Coefficients}.
\begin{proof}[Proof of Corollary \ref{cor:InhomogeneousStrichartzC1Coefficients}]
By well-posedness in $L^2$ for $C^1$-coefficients, let $(U(t,s))_{t,s \in \R}$ denote the propagator of $P$ in $L^2$. So $U(t,s) u_0$ is the solution at time $t$ to
\begin{equation*}
\left\{ \begin{aligned}
Pu &= 0, \\
u(s) &= u_0 \in L^2.
\end{aligned} \right.
\end{equation*}
We let $Pu(t) = (\partial_t 1_{3 \times 3} - A(t)) u(t)$ so that $A(t)$ denotes the time-dependent generator.
The full solution is given by Duhamel's formula
\begin{equation*}
u(t) = U(t,0) u_0 + \int_0^t U(t,s) f(s) ds.
\end{equation*}

Let $\mathcal{T}:L^2 \to L^p([0,T],L^q)$, $\mathcal{T} f = \langle D' \rangle^{-\rho - \frac{\sigma}{2}} U(t,0) f$. The estimates from Theorem \ref{thm:StrichartzEstimatesMaxwell2dCs} applied to homogeneous solutions for $\varepsilon \in C^s$, $1 \leq s \leq 2$, yield
\begin{equation*}
\| \mathcal{T} u_0 \|_{L^p(0,T;L^q)} \lesssim_{\delta} \| U(t,0) u_0 \|_{L^2} \lesssim_{T,\| \varepsilon \|_{C^1}} \| u(0) \|_{L^2}.
\end{equation*}
Indeed, the full derivatives can be replaced by the purely spatial derivatives as a consequence of microlocal estimates 
since the coefficients of $P$ belong to $C^1$, and the energy estimate holds true because $\| \partial_x \varepsilon \|_{L^1 L^\infty} \lesssim_T \| \varepsilon \|_{C^1}$.
We aim to estimate the Duhamel term
\begin{equation*}
\langle D' \rangle^{-\rho-\frac{\sigma}{2}} \int_0^t U(t,s) f(s) ds = \langle D' \rangle^{-\rho- \frac{\sigma}{2}} U(t,0) \int_0^t U(0,s) f(s) ds.
\end{equation*}
By invoking Lemma \ref{lem:ChristKiselev}, it suffices to estimate
\begin{equation}
\label{eq:ModifiedDuhamelTerm}
\| \langle D' \rangle^{-\rho-\frac{\sigma}{2}} U(t,0) \int_0^T U(0,s) f(s) ds \|_{L^p L^q} \lesssim \| \langle D' \rangle^{\tilde{\rho} + \frac{\sigma}{2}} f \|_{L^{\tilde{p}'} L^{\tilde{q}'}}.
\end{equation}

To prove boundedness of \eqref{eq:ModifiedDuhamelTerm}, we use the $\mathcal{T} \mathcal{T}^*$-argument with duality with respect to $L^2$. Note that
\begin{equation*}
\mathcal{T}^* F = \int_0^T U(s,0)^* \big( \langle D' \rangle^{-\rho-\frac{\sigma}{2}} f(s) \big) ds.
\end{equation*}
Hence, to derive estimates for $\int_0^T U(0,s) f(s) ds$, it suffices to show 
\begin{equation*}
\| \langle D' \rangle^{-\tilde{\rho} - \frac{\sigma}{2}} U(0,t)^* g \|_{L^{\tilde{p}}(0,T;L^{\tilde{q}})} \lesssim \| g \|_{L^2(\R^2)}.
\end{equation*}
Since $L^2$-duality respects divergence-free vector fields, $(u_1,u_2)$ is divergence-free, and so is $(f_1,f_2)$, we can suppose that $(g_1,g_2)$ is divergence-free.

We compute the time-dependent generator for $U(0,t)^*$  by
\begin{equation*}
\partial_t \langle U(0,t)^* u, v \rangle = \partial_t \langle u, U(0,t) v \rangle = \langle u , - U(0,t) A(t) v \rangle = - \langle A(t)^* U(0,t)^* u, v \rangle
\end{equation*}
in the general case $\varepsilon = (\varepsilon^{ij})_{i,j=1,2}$. We thus find for the full operator
\begin{equation*}
P^* = 
\begin{pmatrix}
- \partial_t & 0 & \varepsilon_{11} \partial_2 - \varepsilon_{12} \partial_1 \\
0 & - \partial_t & - \varepsilon_{22} \partial_1 + \varepsilon_{12} \partial_2 \\
\partial_2 & - \partial_1 & - \partial_t 
\end{pmatrix}
= - P^t + E
\end{equation*}
with $\| E \|_{L^2 \to L^2} \lesssim 1$ for Lipschitz coefficients. Hence, it suffices to prove estimates for $P^t$. These are a consequence of the proof of Theorem \ref{thm:StrichartzEstimatesMaxwell2dCs}. The only difference happens in the diagonalization: Transposing \eqref{eq:DiagonalSymbol}-\eqref{eq:InverseEigenvectorMatrix} yields
\begin{equation*}
(m^{-1})^t(x,\xi) d(x,\xi) m^t(x,\xi) = p^t(x,\xi).
\end{equation*}
Note
\begin{equation*}
m^t(x,\xi) = 
\begin{pmatrix}
- \xi_1^* \varepsilon_{22}(x) + \xi_2^* \varepsilon_{12}(x) & \xi_1^* \varepsilon_{12}(x) - \xi_2^* \varepsilon_{11}(x) & 0 \\
\xi_2^* &  -\xi_1^* & 1 \\
-\xi_2^* & \xi_1^* & 1
\end{pmatrix}
,
\end{equation*}
and analogously for the operators in \eqref{eq:DiagonalMatrixOperator}--\eqref{eq:InverseEigenvectorOperators}.
The estimate for $P^t$ corresponding of Theorem \ref{thm:StrichartzEstimatesMaxwell2dCs}  is
\begin{equation*}
\| \langle D' \rangle^{-\tilde\rho- \frac{\sigma}{2}} u \|_{L^{\tilde p} L^{\tilde q}} \lesssim_{T, \kappa} \| u \|_{L^2} + \| P^t u \|_{L^2} + \| |D'|^{-\frac{1}{2}-\frac{\sigma}{2}} \rho_e^* \|_{L^2}
\end{equation*}
with $\rho_e^* = \nabla \cdot(\tilde{ \varepsilon}^{-1} \tilde{u})$, $\tilde{u} = (u_1,u_2)$, and the adjugate
$\tilde{\varepsilon}^{-1} = \varepsilon \cdot \det (\varepsilon^{-1})$  of $\varepsilon^{-1}$.

Consequently, for $U(0,t)^* u_0$ we find
\begin{equation*}
\begin{split}
\| \langle D' \rangle^{-\tilde \rho-\frac{\sigma}{2}} U(0,t)^* u_0 \|_{L^{\tilde p} L^{\tilde q}} &\lesssim \| U(0,t)^* u_0 \|_{L^2} + \| (P^* - E) U(0,t)^* u_0 \|_{L^2} \\
&\qquad + \| |D'|^{-\frac{1}{2}-\frac{\sigma}{2}} \nabla \cdot(\tilde{\varepsilon}^{-1} U(0,t)^* u_0) \|_{L^2} \\
&\lesssim \| u_0 \|_{L^2} + \| |D'|^{-\frac{1}{2}-\frac{\sigma}{2}} \rho_e^* \|_{L^2},
\end{split}
\end{equation*}
which follows from the energy estimate in $L^2$. For low frequencies, the energy estimate also allow to dominate the second 
term in the last line  by the first. For high frequencies, we can take advantage that $\tilde{u}_0$ is 
divergence free, that $\rho_e^*(t)= \nabla\cdot (\tilde{ \varepsilon}^{-1}\tilde{u}_0)$, and that
$\varepsilon$ is isotropic. It follows
\begin{equation*}
\| S_{\gtrsim 1}' ( |D'|^{-\frac{1}{2}-\frac{\sigma}{2}} \nabla \cdot (\tilde{\varepsilon}^{-1} \tilde{u})) \|_{L^2} \lesssim_{\| \varepsilon \|_{C^1}} \| u(0) \|_{L^2}.
\end{equation*}
We have proved for  divergence-free $(f_1,f_2)$ that
\begin{equation*}
\| \langle D' \rangle^{-\rho - \frac{\sigma}{2}} \int_0^t U(t,s) f(s) ds \|_{L^p(0,T;L^q)} \lesssim \| \langle D' \rangle^{\tilde{\rho} + \frac{\sigma}{2}} f \|_{L^{\tilde{p}'}(0,T;L^{\tilde{q}'})},
\end{equation*}
where $(\rho,p,q,2)$, $(\tilde{\rho},\tilde{p},\tilde{q},2)$ are Strichartz pairs. Note that the case $p = \tilde{p} =2$ cannot be covered by Lemma \ref{lem:ChristKiselev}, but $p=2$ is not wave-admissible in two dimensions anyway.
We finish the proof by Duhamel's formula and the triangle inequality.
\end{proof}

In the following we derive Strichartz estimates only with spatial derivatives for $\| \partial_x \varepsilon \|_{L^2 L^\infty} \lesssim 1$. 
This will become useful in the quasilinear case where we shall control $\| \nabla_{x'} u \|_{L^4 L^\infty}$ and have
$\| \partial_x (\varepsilon(u)) \|_{L^2 L^\infty} \lesssim_{\| u \|_{L^\infty}} \| \nabla_{x'} u \|_{L^2 L^\infty} \lesssim_T \| \nabla_{x'} u \|_{L^4 L^\infty}$. Also note  that
\begin{equation*}
\| \partial_x \varepsilon \|_{L^2 L^\infty} \lesssim \| \varepsilon \|_{\mathcal{X}^s}
\end{equation*}
by Sobolev's embedding for $s > \frac{3}{2}$. Actually, this condition is only needed in our proof of the microlocal estimate, for which we use the following commutator estimate for the FBI transform, also due to Tataru \cite{Tataru2000}.
Let $X = L^2 L^\infty$ and $X^1 = \{ u \in X : \partial_x u \in X \}$.
\begin{theorem}[{\cite[Theorem~2]{Tataru2000}}]
\label{thm:FBIL2Linfty}
Assume that $a \in X^1 C^\infty_c$. Then,
\begin{equation*}
\| R_{\lambda,a}  \|_{L^\infty L^2 \to L^2_\Phi} = \| T_\lambda A_\lambda - a T_\lambda \|_{L^\infty L^2 \to L^2_\Phi} \lesssim_{\| a \|_{X^1 C^\infty_c}} \lambda^{-\frac{1}{2}}.
\end{equation*}
\end{theorem}
We can now show Corollary \ref{cor:StrichartzEstimatesL2LinfCoefficients}.
\begin{proof}[Proof of Corollary \ref{cor:StrichartzEstimatesL2LinfCoefficients}]
Let $S_\lambda^{\ll \tau}$ denote the smooth frequency localization to regions $\{ |\xi| \sim \lambda \wedge |\xi'| \ll \lambda \}$. In the following we derive a favorable estimate of $\| \langle D' \rangle^{-\rho -\frac{\sigma}{p}} S_\lambda^{\ll \tau}u \|_{L^p L^q}$ for $\lambda \gg 1$, $(\rho,p,q,2)$ a Strichartz pair, and $p \neq \infty$. As symbol at unit frequencies we consider 
\begin{equation*}
p(x,\xi) = i \begin{pmatrix}
\xi_0 & 0 & -\xi_2 \\
0 & \xi_0 & \xi_1 \\
-\varepsilon_{11} (x) \xi_2 + \varepsilon_{21}(x) \xi_1 & \varepsilon_{22}(x) \xi_1 - \varepsilon_{12}(x) \xi_2 & \xi_0
\end{pmatrix}
s_0(\xi)
\end{equation*}
with $s_0 \in C^\infty_c(B(0,2))$, $s_0 = 1$ for $1 \sim |\xi| \sim |\xi_0| \gg |\xi'|$ such that $S_\lambda^{\ll \tau} = s_0(\xi/\lambda)$.

 By H\"older's inequality and Sobolev embedding, we find
\begin{equation*}
\begin{split}
\| \langle D' \rangle^{-\rho - \frac{\sigma}{p}} S_\lambda^{\ll \tau} u \|_{L^p L^q} &\lesssim \lambda^{\frac{1}{2}-\frac{\sigma}{p}} \| S_\lambda^{\ll \tau} u \|_{L^2} \\
&= \lambda^{\frac{1}{2}-\frac{\sigma}{p}} \| T^*_\lambda p^{-1}(x,\xi) T_\lambda \, T_\lambda^* p(x,\xi) T_\lambda \, S_\lambda^{\ll \tau} u \|_{L^2} \\
&\lesssim \lambda^{\frac{1}{2}-\frac{\sigma}{p}} \| p(x,\xi) T_\lambda S_\lambda^{\ll \tau} u \|_{L^2} \\
&\lesssim \lambda^{\frac{1}{2}-\frac{\sigma}{p}} \| R_{\lambda,p} \|_{L^\infty L^2 \to L^2} \| S_\lambda^{\ll \tau} u \|_{L^\infty L^2} \\
&\quad + \lambda^{-\frac{1}{2}-\frac{\sigma}{p}}\big( \| P(x,D) S_\lambda^{\ll \tau} u \|_{L^2} 
   +  \| \partial_x \varepsilon \|_{L^2 L^\infty} \| S_\lambda^{\ll \tau} u \|_{L^\infty L^2}\big).
\end{split}
\end{equation*}
Since $\| R_{\lambda,p} \|_{L^\infty L^2 \to L^2} \lesssim_{\| \partial_x \varepsilon \|_{L^2 L^\infty}} \lambda^{-\frac{1}{2}}$ by Theorem \ref{thm:FBIL2Linfty}, the first term is acceptable. The third term comes from recovering divergence form and is clearly admissible. Also note that we omitted an error term from localizing $T_\lambda S_\lambda^{\ll \tau} u$ to $\{|\xi_0| \sim 1, \; |\xi'| \ll 1\}$ as this gains arbitrarily many derivatives. We turn to the second term. For this purpose let $P^\lambda(x,D)$ denote $P$ with coefficients  frequency truncated at $\lambda$. 
The coefficients of $P^\lambda$ have Lipschitz norm $\lesssim \lambda^{\frac{1}{2}}$. 
Note that
\begin{equation*}
\begin{split}
\| P(x,D) S_\lambda^{\ll \tau} u \|_{L^2} &\leq \| P^\lambda(x,D) S_\lambda^{\ll \tau} u \|_{L^2} + \| P^{\gtrsim \lambda}(x,D) S_\lambda^{\ll \tau} u \|_{L^2} \\
&\lesssim \| P^\lambda(x,D) S_\lambda^{\ll \tau} u \|_{L^2} + \| \partial_x \varepsilon \|_{L^2 L^\infty} \| u \|_{L^\infty L^2}.
\end{split}
\end{equation*}
Moreover, a kernel estimate yields
\begin{equation*}
\begin{split}
\| P^{\lambda} S_\lambda^{\ll \tau} u \|_{L^2} &\leq \| S_{\lambda}^{\ll \tau} P^{ \lambda} u \|_{L^2} + \| [P^{\lambda}, S_\lambda^{\ll \tau}] u \|_{L^2} \\
&\lesssim \| S_\lambda^{\ll \tau} P^\lambda u \|_{L^2} + \lambda^{\frac{1}{2}} \| u \|_{L^2} \\
&\lesssim \| S_\lambda^{\ll \tau} P u \|_{L^2} + \| S_\lambda^{\ll \tau} P^{\gtrsim \lambda} u \|_{L^2} + \lambda^{\frac{1}{2}} \| u \|_{L^2}.
\end{split}
\end{equation*}
For the first term we use Bernstein's inequality in time
\begin{equation*}
\| S_\lambda^{\ll \tau} P u \|_{L^2} \lesssim \lambda^{\frac{1}{2}} \| S_\lambda^{\ll \tau} P u \|_{L^1 L^2}
\end{equation*}
and note that it is still summable in $\lambda$ due to the additional factor $\lambda^{-\frac{\sigma}{p}}$
and since we only want to estimate $\| \langle D' \rangle^{-\alpha} u \|_{L^p L^q}$ for $\alpha > \rho + \frac{\sigma}{p}$
in Corollary \ref{cor:StrichartzEstimatesL2LinfCoefficients}.
 Likewise, the third term can be summed. For the second term we note
\begin{equation*}
\| S_\lambda^{\ll \tau} P^{\gtrsim \lambda} u \|_{L^2} \lesssim \| \partial_x \varepsilon \|_{L^2 L^\infty} \| u \|_{L^\infty L^2}.
\end{equation*}
This handles the part with spatial frequencies much smaller than temporal ones.

For the dyadic frequency blocks $\{|\tau| \sim |\xi'| \sim \lambda \}$ we can use Theorem \ref{thm:MaxwellStrichartzEstimatesBesovRegularity} with $P^\lambda$ to recover the dyadic estimate \eqref{eq:DyadicEstimateXs}:
\begin{equation*}
\begin{split}
\lambda^{-\rho - \frac{\sigma}{p}} \| S_\lambda S_\lambda' u \|_{L^p L^q} &\lesssim \| S_\lambda u \|_{L^\infty L^2} + \lambda^{-\sigma} \| P^\lambda S_\lambda S'_\lambda u \|_{L^1 L^2} \\
&\quad + \lambda^{-\frac{1}{2}-\frac{\sigma}{p}} \| S'_\lambda \rho_e(0) \|_{L^2} + \lambda^{-\frac{1}{2} - \frac{\sigma}{p}} \| S'_\lambda \partial_t \rho_e \|_{L^1 L^2} \\
&\lesssim \| u \|_{L^\infty L^2} + \lambda^{-\sigma} \| S_\lambda P u \|_{L^1 L^2} \\
&\quad + \lambda^{-\frac{1}{2}-\frac{\sigma}{p}} \| S'_\lambda \rho_e(0) \|_{L^2} + \lambda^{-\frac{1}{2}-\frac{\sigma}{p}} \| S'_\lambda \partial_t \rho_e \|_{L^1 L^2}.
\end{split}
\end{equation*}
We used the same commutator considerations as in the proof of Theorem \ref{thm:MaxwellStrichartzEstimatesBesovRegularity} and the fundamental theorem of calculus together with Minkowski's inequality. By the energy estimate, we conclude for $\alpha>\rho + \frac{\sigma}{p}$
\begin{equation*}
\begin{split}
\| \langle D' \rangle^{-\alpha} u \|_{L^p(0,T;L^q)} &\lesssim \| u_0 \|_{L^2} + \| Pu \|_{L^1 L^2} \\
&\quad + \| \langle D' \rangle^{-\frac{1}{2}-\frac{\sigma}{p}} \rho_e(0) \|_{L^2} + \| \langle D' \rangle^{-\frac{1}{2}-\frac{\sigma}{p}} \partial_t \rho_e \|_{L^1 L^2},
\end{split}
\end{equation*}
and the proof is complete.
\end{proof}

\section{Local well-posedness for a 2d quasilinear Maxwell system}
\label{section:QuasilinearMaxwell}

This section is devoted to the proof of Theorem \ref{thm:LocalWellposednessQuasilinearMaxwell}. Recall that the system under consideration is given by
\begin{equation}
\label{eq:QuasilinearMaxwellEquationsSection}
\left\{ \begin{array}{cl}
\partial_t u_1 &= \partial_2 u_3, \qquad u(0) = u_0 \in H^s(\R^2;\R)^3, \\
\partial_t u_2 &= - \partial_1 u_3, \qquad \partial_1 u_1 + \partial_2 u_2 = 0, \\
\partial_t u_3 &= \partial_2 (\varepsilon^{-1}(u) u_1) - \partial_1 (\varepsilon^{-1}(u) u_2),
\end{array} \right.
\end{equation}
where $\varepsilon^{-1}(u) = \psi(|u_1|^2+|u_2|^2)$ and $\psi:\R_{\geq 0} \to \R_{\geq 0}$ is smooth, monotone increasing and satisfies $\psi(0)=1$. We denote $\tilde{u} = (u_1,u_2)$.

Without using  dispersive properties, energy methods yield local well-posedness in $H^s(\R^2)$ for $s>2$. To improve on this by Strichartz estimates, we follow the arguments from Ifrim--Tataru \cite{IfrimTataru2020}. Let $A= \sup_{0 \leq t' \leq t} \| u(t') \|_{L^\infty_{x^\prime}}$ and $B(t) = \| \nabla_{x'} u(t) \|_{L^\infty_{x^\prime}}$. In the following we take local existence of smooth solutions for granted, which follows by classical arguments, e.g., parabolic regularization. We focus on proving estimates in rough norms.
The argument consists of three steps:
\begin{enumerate}
\item[1)] We prove energy estimates 
\begin{equation}
\label{eq:EnergyEstimatesSolutions}
E^s(u(t)) \lesssim e^{c(A) \int_0^t B(t') dt'} E^s(u(0)),
\end{equation}
for smooth solutions $u$, where $E^s(u) \approx_A \| u \|_{H^s}$ and $s \geq 0$.
\item[2)]  We show $L^2$-Lipschitz bounds for differences of solutions $v= u^1-u^2$
\begin{equation}
\label{eq:L2LipschitzBoundsDifferencesSolutions}
\| v(t) \|^2_{L^2} \lesssim e^{c(A) \int_0^t B(t') dt'} \| v(0) \|^2_{L^2},
\end{equation}
 where $u^1$ and $u^2$ solve \eqref{eq:QuasilinearMaxwellEquationsSection}  and 
\begin{align*}
A &= \sup_{0 \leq t' \leq t}  \| u^1(t') \|_{L^\infty_{x^\prime}} + \sup_{0 \leq t' \leq t} \| u^2(t') \|_{L^\infty_{x^\prime}}, \\
B(t) &= \| \nabla_{x'} u^1(t) \|_{L^\infty_{x^\prime}} + \| \nabla_{x'} u^2(t) \|_{L^\infty_{x^\prime}}.
\end{align*}
\item[3)]  We conclude the proof of Theorem \ref{thm:LocalWellposednessQuasilinearMaxwell}, using frequency envelopes. These were introduced by Tao in the context of wave maps \cite{Tao2001} and turned out as very useful to treat quasilinear evolution equations.
\end{enumerate}

We start with energy estimates.
\begin{proposition}
\label{prop:EnergyEstimatesSolutions}
Let $s \geq 0$. Then, we find \eqref{eq:EnergyEstimatesSolutions} to hold. Let  $s>11/6$. For $u_0 \in H^s$,
there is a time $T=T(\|u_0\|_{H^s})$ such that $T$ is lower semicontinuous and 
\begin{equation*}
\sup_{t \in [0,T]} \| u(t) \|_{H^s} \lesssim \| u_0 \|_{H^s}.
\end{equation*}
\end{proposition}
\begin{proof}
We consider energy norms
\begin{equation*}
\| u \|^2_{E^s} = \langle \langle D^\prime \rangle^s u, C(u) \langle D^\prime \rangle^s u \rangle \approx_A \| u \|^2_{H^s},
\end{equation*}
where we define the smooth map $C$ in \eqref{eq:ScalarProductMatrix} below. We rewrite \eqref{eq:QuasilinearMaxwellEquationsSection} as 
\begin{equation}
\label{eq:MaxwellEquationsTimeDerivative}
\partial_t u = \mathcal{A}^j(u) \partial_j u,
\end{equation}
 with the coefficient matrices
\begin{align*}
\mathcal{A}^1(u) &= 
\begin{pmatrix}
0 & 0 & 0 \\
0 & 0 & -1 \\
-2 \psi^\prime(|\tilde{u}|^2) u_1 u_2 & - 2 \psi^\prime(|\tilde{u}|^2) u_2^2 - \psi(|\tilde{u}|^2) & 0
\end{pmatrix}
, \\
\mathcal{A}^2(u) &= 
\begin{pmatrix}
0 & 0 & 1 \\
0 & 0 & 0\\
2 \psi^\prime(|\tilde{u}|^2) u_1^2 + \psi(|\tilde{u}|^2) & 2 \psi^\prime(|\tilde{u}|^2) u_1 u_2 & 0
\end{pmatrix}
.
\end{align*}
For the time derivative we find
\begin{align}\label{eq:energy}
\frac{d}{dt} \| u \|^2_{E^s} &= \langle \langle D^\prime \rangle^s ( \mathcal{A}^j(u) \partial_j u), C(u) \langle D^\prime \rangle^s u \rangle + \langle \langle D^\prime \rangle^s u, C(u) \langle D^\prime \rangle^s ( \mathcal{A}^j(u) \partial_j u) \rangle \\
&\quad + \langle \langle D' \rangle^s u,C'(u) (\partial_t u) \langle D' \rangle^s u \rangle,\notag
\end{align}
at time $t$, which is suppressed. Because of \eqref{eq:MaxwellEquationsTimeDerivative}, the last term is  estimated by
\begin{equation*}
| \langle D' \rangle^s u, \tilde{C}(u) (\partial_t u) \langle D' \rangle^s u \rangle | \lesssim c(A) B \|u\|^2_{H^s}.
\end{equation*}
The first term in \eqref{eq:energy} can be expressed as
\begin{align*}
 \langle \langle D^\prime \rangle^s& (\mathcal{A}^j(u) \partial_j u), C(u) \langle D^\prime \rangle^s u \rangle \\
&= \langle \mathcal{A}^j(u) ( \langle D^\prime \rangle^s \partial_j u) + (\langle D^\prime \rangle^s  \mathcal{A}^j(u) - \mathcal{A}^j(u) \langle D^\prime \rangle^s) \partial_j u, C(u) \langle D^\prime \rangle^s u \rangle \\
&= \langle \mathcal{A}^j(u) (\langle D^\prime \rangle^s \partial_j u), C(u) \langle D^\prime \rangle^s u \rangle + II.
\end{align*}
Paraproduct/Moser estimates yield
\begin{equation*}
II \leq \| (\langle D^\prime \rangle^s \mathcal{A}^j(u)  - \mathcal{A}^j(u) \langle D^\prime \rangle^s) \partial_j u \|_{L^2} \| C(u) \langle D^\prime \rangle^s u \|_{L^2} \lesssim_A B \| u \|^2_{H^s}.
\end{equation*}
Integrating by parts, the other summand becomes
\begin{equation*}
\begin{split}
 \langle \mathcal{A}^j(u)& \langle D^\prime \rangle^s \partial_j u, C(u) \langle D^\prime \rangle^s u \rangle_{L^2}\\
&= - \langle \langle D^\prime \rangle^s u,\partial_j (\mathcal{A}_j(u)^* C(u) \langle D^\prime \rangle^s u ) \rangle_{L^2} \\
&= - \langle \langle D^\prime \rangle^s u, \mathcal{A}^j(u)^* C(u) \langle D^\prime \rangle^s \partial_j u \rangle_{L^2} + \mathcal{O}_A(B \| u \|^2_{H^s}).
\end{split}
\end{equation*}
The second summand in \eqref{eq:energy} can be treated analogously. We seek to cancel the highest-order terms.
This gives the condition
\begin{equation*}
\mathcal{A}^j(u)^* C(u) = C(u)^* \mathcal{A}^j(u).
\end{equation*}
Solving the system of equations and setting $C_{33} = 1$, we find
\begin{equation}
\label{eq:ScalarProductMatrix}
C(u) =
\begin{pmatrix}
\psi + 2 \psi^\prime \cdot u_1^2 & 2 \psi^\prime \cdot u_1 u_2 & 0 \\
2 \psi^\prime \cdot u_1 u_2 & \psi + 2 \psi^\prime \cdot u_2^2 & 0 \\
0 & 0 & 1
\end{pmatrix}
.
\end{equation}
We thus conclude
\begin{equation*}
\frac{d}{dt} \| u(t) \|^2_{E^s} \lesssim_A B(t) \| u(t) \|^2_{E^s},
\end{equation*}
and  Gr\o nwall's lemma implies \eqref{eq:EnergyEstimatesSolutions}. To check that
\begin{equation*}
\| u \|^2_{H^s} \approx_A \| u \|^2_{E^s},
\end{equation*}
it is enough to show that $C$ is uniformly elliptic. This follows from  Young's inequality
in the computation 
\begin{align*}
\Big\langle \begin{pmatrix}
\xi_1 \\ \xi_2
\end{pmatrix}&
,
\begin{pmatrix}
\psi+ 2 \psi^\prime \cdot u_1^2 &  2 \psi^\prime \cdot u_1 u_2 \\
2 \psi^\prime \cdot u_1 u_2 & \psi + 2 \psi^\prime \cdot u_2^2
\end{pmatrix}
\begin{pmatrix}
\xi_1 \\ \xi_2
\end{pmatrix}
\Big\rangle \\
&\geq (\psi + 2 \psi^\prime \cdot u_1^2) \xi_1^2 - 4 \psi^\prime \cdot (\frac{u_1^2 \xi_1^2}{2} + \frac{u_2^2 \xi_2^2}{2} ) + (\psi + 2 \psi^\prime \cdot u_2^2) \xi_2^2 \\
&= \psi \cdot |\xi|^2.
\end{align*} 

The second claim of the proposition is a consequence of Sobolev's embedding for $s>2$ . To improve on this, we use 
Strichartz estimates and show
\begin{equation}\label{eq:L4Linfty}
\| \nabla_{x'} u \|_{L^4(0,T; L^\infty)} \lesssim \| u_0 \|_{H^s}
\end{equation}
for $s>11/6$. To bootstrap, we require that $\|\nabla_{x'} u\|_{L^4(0,T_0; L^\infty)}\le K$ for a fixed number $K>0$ and 
a maximally defined time $T_0>0$. Take $T\in(0,T_0)$ with $\|\partial_x \varepsilon\|_{L^2(0,T;L^\infty)}\lesssim_A T^{1/4}K\le 1$ and $\| \partial_x \varepsilon \|_{L^1(0,T;L^\infty)} \lesssim_A T^{\frac{3}{4}} K \leq 1$. 
We thus have uniform constants in the energy inequality \eqref{eq:EnergyEstimatesSolutions} and in the Strichartz estimate 
\begin{equation}
\label{eq:StrichartzEstimate}
 \| \langle D' \rangle^{-\alpha} w \|_{L^p(0,T; L^q)} \lesssim \|w_0 \|_{L^2} + \|P(x,D) w \|_{L^1 L^2},
\end{equation}
for $\alpha > \rho + \frac{1}{3p}$ from Corollary \ref{cor:StrichartzEstimatesL2LinfCoefficients} with $\tilde{s} = 1$, if $\partial_1 w_1+\partial_2 w_2=0$.
For low frequencies, Bernstein's  inequality and \eqref{eq:EnergyEstimatesSolutions} yield
\begin{equation*}
\| S^\prime_{\lesssim 1} \nabla_{x'} u \|_{L^4 L^\infty} \lesssim_A T^{\frac{1}{4}} \| u_0 \|_{L^2}.
\end{equation*}
 For high frequencies, we define the auxiliary function $v = \langle D^\prime \rangle^s u$ satisfying $\partial_1 v_1+\partial_2 v_2=0$.
Similar as above, a fixed time estimate  gives
\begin{equation}
\label{eq:FixedTimeEstimate}
\| P(x,u,D) v(t) \|_{L^2} = \| [P(x,u,D), \langle D^\prime \rangle^s] u(t)  \|_{L^2} \lesssim_A \| \nabla u(t) \|_{L_{x^\prime}^\infty} \| u(t) \|_{H^s},
\end{equation}
where $P(x,u,D)$ has coefficients $\varepsilon(u)^{-1}$.
For the Strichartz pair $(3/4,4,\infty,2)$, inequality \eqref{eq:StrichartzEstimate} implies
\begin{align*}
\| S_{\ge 1}^\prime \nabla_{x'} u \|_{L^4 L^\infty} &\lesssim \| \langle D' \rangle^{1-s} v \|_{L^4 L^\infty} 
 \lesssim \| v_0 \|_{L^2} + \| P(x,u,D) v \|_{L^1 L^2},
\end{align*}
since $s>1+\rho+\frac{1}{3p} = \frac{11}6$ by our assumption.
By virtue of \eqref{eq:FixedTimeEstimate} and \eqref{eq:EnergyEstimatesSolutions}, we conclude
\begin{align*}
\| S_{\ge 1}^\prime \nabla_{x'} u \|_{L^4 L^\infty} &\lesssim_A \| u_0 \|_{H^s} + \| \nabla_{x'} u \|_{L^1 L^\infty_{x^\prime}} \| u \|_{L^\infty H^s} \\
&\lesssim \|u_0 \|_{H^s} + T^{\frac{4}{3}} \| \nabla_{x'} u \|_{L^4 L_{x^\prime}^\infty} \| u \|_{L^\infty H^s}\\
&\lesssim \|u_0 \|_{H^s} + T^{\frac{4}{3}} \| \nabla_{x'} u \|_{L^4 L_{x^\prime}^\infty} \| u_0 \|_{H^s}
\end{align*}
on $[0,T ]$ also using the equivalence $\| u(t) \|_{E^s} \approx_A \| u(t) \|_{H^s}$.
Starting with a sufficiently large $K$, we can now fix a small $ T_1=T_1(\|u_0 \|_{H^s})\in (0,T_0)$ such that
$\|\nabla_{x'} u \|_{L^4(0,T_1; L^\infty)}\lesssim  \|u_0 \|_{H^s} < K$. 
\end{proof}

We turn to the $L^2$-bound for differences.
\begin{proposition}
\label{prop:L2BoundDifferencesSolutions}
Let $u^1$ and $u^2$ be two solutions to \eqref{eq:QuasilinearMaxwellEquationsSection} with finite $A$ and $B$,  
and set $v=u^1-u^2$. Then, we find \eqref{eq:L2LipschitzBoundsDifferencesSolutions} to hold. Moreover, if $s>11/6$, 
there is a time $T=T(\|u^i(0)\|_{H^s})$ such that $T$ is lower semicontinuous and 
\begin{equation}
\label{eq:L2LipschitzBoundsDifferencesSolutionsII}
\sup_{t \in [0,T]} \| v(t) \|_{L^2} \lesssim_{\| u^i(0) \|_{H^s}} \| v(0) \|_{L^2}.
\end{equation}
\end{proposition}
\begin{proof}
We observe that $v$ solves the equation
\begin{align*}
\partial_t v &= \mathcal{A}^j(u^1) \partial_j v + [ \mathcal{A}^1(u^1) - \mathcal{A}^1(u^2)] \partial_1 u^2 + [ \mathcal{A}^2(u^1) - \mathcal{A}^2(u^2)] \partial_2 u^2 \\
&= \mathcal{A}^j(u^1) \partial_j v + \mathcal{B}^j(u^1, u^2)(v, \partial_j u^2).
\end{align*}
To prove the claim, we work with the equivalent norm
\begin{equation*}
\| v \|_{2,u^1}^2 = \langle v, C(u^1) v \rangle_2
\end{equation*}
with $C$ from the previous proof. For fixed $t$, we calculate
\begin{align*}
\frac{1}{2} \frac{d}{dt} \| v \|^2_{2,u^1} &= \langle \mathcal{A}^j(u^1) \partial_j v, C(u^1) v \rangle + \langle v, C(u^1) \mathcal{A}^j(u^1) \partial_j v \rangle \\
&\quad + \langle \mathcal{B}^j(u^1,u^2)(v,\partial_j u^2), C(u^1) v \rangle + \langle v, C(u^1) \mathcal{B}^j(u^1,u^2)(v,\partial_j u^2) \rangle \\
&\quad + \mathcal{O}_A(B \| v \|_2^2),
\end{align*}
where the error accounts for the contribution of the time derivative of $C(u^1)$.
Furthermore,
\begin{align*}
 | \langle \mathcal{B}^j(u^1,u^2)(v,\partial_j u^2), C(u^1) v \rangle | 
&\leq \| \mathcal{B}^j(u^1,u^2)(v, \partial_j u^2) \|_2 \,\| C(u^1) v \|_2 
\lesssim_A B \| v \|_2^2.
\end{align*}
The key estimate
\begin{equation*}
| \langle \mathcal{A}^j(u^1) \partial_j v, C(u^1) v \rangle + \langle v, C(u^1) \mathcal{A}^j(u^1) \partial_j v \rangle | \lesssim_A B \| v \|_2^2
\end{equation*}
is carried out as in the proof of Proposition \ref{prop:EnergyEstimatesSolutions}, employing our choice 
of $C$. Estimate \eqref{eq:L4Linfty} now implies \eqref{eq:L2LipschitzBoundsDifferencesSolutionsII}.
\end{proof}

In the third step, we show continuous dependence by means of the frequency envelope argument detailed in Ifrim--Tataru \cite{IfrimTataru2020}. The envelopes represent the dyadically localized Sobolev energy. We use the following tailored version, where $P_k$ is the standard Littlewood--Paley projector at frequency $2^k$. 
\begin{definition}
$(c_k)_{k \geq 0} \in \ell^2$ is called a \emph{frequency envelope} for a function $u$ in $H^s$ if it has the following properties:
\begin{itemize}
\item[a)] Energy bound:
\begin{equation*}
\| P_k u\|_{H^s} \leq c_k.
\end{equation*}
\item[b)] Slowly varying: There is $\delta > 0$ such that for all $j,k \in \mathbb{N}$
\begin{equation*}
\frac{c_k}{c_j} \lesssim 2^{-\delta |j-k|}.
\end{equation*}
\end{itemize}
The envelopes are called \emph{sharp}
if they also  $\| u \|^2_{H^s} \approx \sum_k c_k^2$ for a family of functions.
\end{definition}

As noted in \cite{IfrimTataru2020} such envelopes always exist.
The idea is to show that for a solution not only the $H^s$-norm is propagated, but also the frequency envelope for a time depending on the control parameters. This allows to infer continuous dependence. We give the details.

\emph{Regularization.} Let $u_0 \in H^s(\mathbb{R}^2)$ and $(c_k)_{k \geq 0}$ be a sharp frequency envelope for $u_0$ in $H^s$. The regularized initial data  $u_0^n = P_{\leq n} u_0$ for $n \in \mathbb{N}$ have the following properties.
\begin{itemize}
\item[i)] Uniform bounds:
\begin{equation*}
\| P_k u_0^n \|_{H^s} \lesssim c_k,
\end{equation*}
\item[ii)] High frequency bounds:
\begin{equation*}
\| u_0^n \|_{H^{s+j}} \lesssim 2^{jn} c_n,
\end{equation*}
\item[iii)] Difference bounds:
\begin{equation*}
\| u_0^{n+1} - u_0^n \|_{L^2} \lesssim 2^{-sn} c_n,
\end{equation*}
\item[iv)] Limit as $n \to \infty$:
\begin{equation*}
u_0 = \lim_{n \to \infty} u_0^n \text{ in } H^s.
\end{equation*}
\end{itemize}
The regularized initial data give rise to a family of smooth solutions.

\textit{Uniform bounds:} Propositions \ref{prop:EnergyEstimatesSolutions} and \ref{prop:L2BoundDifferencesSolutions}
now yield a time interval of length $T=T(\| u_0 \|_{H^s})$ on which the solutions exist, and also $L^2$-bounds their differences:
\begin{itemize}
\item[i)] High frequency bounds:
\begin{equation*}
\| u^n \|_{C([0,T],H^{s+j})} \lesssim 2^{nj} c_n,
\end{equation*}
\item[ii)] Difference bounds:
\begin{equation*}
\| u^{n+1} - u^n \|_{C([0,T],L^2)} \lesssim 2^{-sn} c_n.
\end{equation*}
\end{itemize}
Interpolation implies
\begin{equation*}
\| u^{n+1} - u^n \|_{C([0,T],H^m)} \lesssim c_n 2^{-(s-m)n}
\end{equation*}
for $m\ge0$. By the $L^2$-bound, we have
\begin{equation*}
\| u - u^n \|_{C([0,T],L^2)} \lesssim 2^{-sn}.
\end{equation*}
This gives convergence in $L^2$ since
\begin{equation*}
u - u^n = \sum_{j=n}^\infty (u^{j+1} - u^j).
\end{equation*}
The frequency localization of the summands at $2^{j+1}$ and the error bound show that
\begin{equation*}
\| u - u^n \|_{C([0,T],H^s)} \lesssim c_{\geq n} = \big( \sum_{j \geq n} c_j^2 \big)^{1/2}, 
\end{equation*}
which gives convergence in $C([0,T],H^s)$. We can now  show Theorem \ref{thm:LocalWellposednessQuasilinearMaxwell}.
\begin{proof}[Proof of Theorem \ref{thm:LocalWellposednessQuasilinearMaxwell}]
 Consider a sequence of initial data
\begin{equation*}
u_{0k} \rightarrow u_0 \text{ \ in } H^s(\mathbb{R}^2), \quad s > \frac{11}{6}.
\end{equation*}
Propositions \ref{prop:EnergyEstimatesSolutions} and \ref{prop:L2BoundDifferencesSolutions} show that the
solutions $u_k$ with inital values $u_{0,k}$  exist on a common time interval $[0,T]$ with $T=T(\| u_0 \|_{H^s})$
and that $u_k$ coverges to $u$ in $C([0,T],L^2)$. For the solutions $u^n_k$ and $u^n$ with regularized initial data 
$u^n_{0k}$ and $u^n$, we have $u^n_{0,k} \to u^n_{0}$ in $H^s$ and hence
\begin{align*}
u^n_k \to u^n\text{ \ in } C([0,T],H^s)
\end{align*}
as $k\to\infty$ by the above reasoning. We then derive
\begin{align*}
\| u_k - u \|_{C([0,T],H^s)} &\lesssim \| u^n_k - u^n \|_{C([0,T],H^s)} + \| u^n - u \|_{C([0,T],H^s)} \\
&\quad + \| u^n_k - u_k \|_{C([0,T],H^s)} \\
&\lesssim \| u_k^n - u^n \|_{C([0,T],H^s)} + c_{\geq n} + c^k_{\geq n}.
\end{align*}
The convergence $u_{0k} \to u_0$ in $H^s$ allows us to choose a sequence of frequency envelopes $c^k \to c$ in $\ell^2$.
We can thus take $n$ uniformly in $k$ such that the second and third term become arbitrarily small, and then let $k\to\infty$.
\end{proof}

\section{Sharpness of derivative loss}
\label{section:Sharpness}
Next, we connect the Maxwell system with a wave equation to infer the sharpness of derivative loss in Theorem \ref{thm:MaxwellStrichartzEstimatesBesovRegularity} for permittivity coefficients $\varepsilon^{ij} \in C^s$ with $1 \leq s \le 2$. For this purpose, we recall the counterexamples in Smith--Tataru \cite{SmithTataru2002}. Smith and Tataru constructed time-independent $C^s$-metrics $g$ for $0 \leq s \le 2$ and solutions $u$ to the corresponding wave equations, which exhaust a derivative loss in Strichartz estimates analogous to that proved in Theorem \ref {thm:MaxwellStrichartzEstimatesBesovRegularity}. They treated the second-order hyperbolic operator 
\begin{equation*}
Q(t,x,\partial_t,\partial_x) = \partial_t^2 - \partial_i g^{ij}(t,x) \partial_j
\end{equation*}
on $[0,1] \times \R^n$ with $g^{ij} \in C^s$ for $n \geq 2$ and $0\le s \le 2$, and the Strichartz estimates
\begin{equation*}
\begin{split}
\| u \|_{L^p (0,1; L^q)} &\lesssim\| u \|_{L^\infty(0,1; H^{\rho^\prime})} + \| \partial_t u \|_{L^\infty(0,1;H^{\rho^\prime - 1})} + \| Q u \|_{L_t^1(0,1;H^{\rho^\prime - 1})},
\end{split}
\end{equation*}
where $(\rho,p,q,n)$ is a Strichartz pair and
\begin{equation*}
\rho^\prime = \rho + \frac{\sigma}{p}, \quad \sigma = \frac{2-s}{2+s}, \quad \delta=\frac2{2+s}=\frac{1+\sigma}2.
\end{equation*}
In \cite{SmithTataru2002} it is proved that $\rho^\prime$ cannot be lowered. For $2/3\le s\le 2$, the authors 
at first look at the  equation
\begin{equation}
\label{eq:SharpMetric}
Q_\lambda(y,\partial_t,\partial_x,\partial_y) = \partial_t^2 - g_\lambda(y) \partial_x^2 - \Delta_y, 
\end{equation}
with $(x,y) \in \R \times\R^{n-1}$ and $g_\lambda(y) = 1+\lambda^{2 \sigma} |y|^2$. Later $g_\lambda$ is replaced by 
$\tilde{g}_\lambda(y)=1+\lambda^{2 \sigma-2\delta} a(\lambda^\delta |y|)$ for a smooth $a\ge0$ 
supported in $[0,2)$ with $a(r)=r^2$ for $r\in[0,1]$. The $C^s$ norm of $\tilde{g}_\lambda$ is uniformly bounded in 
$\lambda$.

We consider smooth solutions to Maxwell equations in two spatial dimensions 
\begin{equation}
\label{eq:MaxwellCounterexample}
\left\{ \begin{array}{cl}
\partial_t \cD &= \nabla_\perp \cH, \qquad \partial_1 \cD_1 + \partial_2 \cD_2 = \rho_e,\\
\partial_t \cH &= - \nabla \times \cE = \partial_2 \cE_1 - \partial_1 \cE_2,
\end{array} \right.
\end{equation}
with the rough permittivity $\varepsilon_\lambda$ given by
\begin{equation*}
\varepsilon_\lambda^{-1}(x,y) = 
\begin{pmatrix}
1 & 0 \\
0 & 1+\lambda^{2 \sigma} y^2
\end{pmatrix}
\end{equation*}
with $\lambda\ge2$. Corresponding to the Strichartz estimate above, we look at
\begin{equation}\label{eq:strich-gamma}
\| (\cD,\cH) \|_{L^p L^q} \lesssim \| (\cD,\cH) \|_{L^\infty H^\gamma} + \| P (\cD,\cH) \|_{L^1 H^{\gamma-\sigma}} + \| \rho_e \|_{L^\infty H^{\gamma-\frac{1}{2}-\frac{\sigma}{p}}}
\end{equation}
with $P$ as in \eqref{eq:RoughSymbolMaxwell2d}, omitting the time interval $(0,1)$. We want to show that this estimate  
can only hold for all $u$ if $\gamma\ge\rho+\frac{\sigma}p$, cf.\ Theorem~\ref{thm:MaxwellStrichartzEstimatesBesovRegularity}.

The field $\cH$ from \eqref{eq:MaxwellCounterexample} solves the wave equation
\begin{align*}
\partial_t^2 \cH &=  (1+\lambda^{2 \sigma} y^2) \partial_1^2 \cH + \partial_2^2 \cH =: \Delta_{\varepsilon_\lambda^{-1}} \cH.
\end{align*}
As in \cite{SmithTataru2002} we choose 
\begin{equation}
\label{eq:HFieldExample}
\cH(t,x,y)= \cH^\lambda(t,x,y) = \frac{1}{(\log \lambda)^2} \int \beta((\log \lambda)^{-2} r) u^\lambda_r(t,x,y) dr
\end{equation}
for the function
\begin{equation}
\label{eq:Ulambda}
u^\lambda_r(t,x,y) = e^{ir \lambda (t-x) - \frac{i}{2}\lambda^\sigma t - \frac{1}{2} r \lambda^{2 \sigma} y^2}
\end{equation}
with  $0\neq\beta \in C^\infty_c(\R)$ satisfying $\operatorname{supp} \beta \subseteq [1,2]$ and $ \beta \geq 0$. As noted in \cite{SmithTataru2002}, the function $\cH^\lambda$ is essentially supported in $K_\lambda^t=\{|t-x| \leq \lambda^{-1} (\log \lambda)^{-2}, |y| \leq \lambda^{-\delta} (\log \lambda)^{-1} \}$. 

To solve \eqref{eq:MaxwellCounterexample}, we set
\begin{equation}
\label{eq:Dlambda}
\cD^\lambda(t,x,y) =  \int_0^t \nabla_\perp \cH^\lambda(s,x,y) ds + C^\lambda(x,y) + \varepsilon_\lambda \nabla \psi^\lambda(x,y)
\end{equation}
with $C^\lambda$ and $\psi^\lambda$ to be determined.

Below we choose $C^\lambda$ with support in
\begin{equation*}
K_\lambda^0=\{|x| \leq \lambda^{-1} (\log \lambda)^{-2}, |y| \leq \lambda^{-\delta} (\log \lambda)^{-1} \}.
\end{equation*}
In view of the second line of \eqref{eq:MaxwellCounterexample}, we compute
\begin{equation*}
\partial_t \cH^\lambda + \nabla \times (\varepsilon^{-1}_\lambda  \cD^\lambda) = \partial_t \cH^\lambda(t,x,y) - \int_0^t \Delta_{\varepsilon_\lambda^{-1}} \cH^\lambda(s,x,y) ds + \nabla \times (\varepsilon_\lambda^{-1} C^\lambda).
\end{equation*}
Inserting the equation
\begin{equation*}
\partial_t^2 \cH^\lambda - \Delta_{\varepsilon_\lambda^{-1}} \cH^\lambda = -  \tfrac{1}{4} \lambda^{2 \sigma} \cH^\lambda
\end{equation*}
from \cite{SmithTataru2002}, we obtain
\begin{equation*}
\partial_t \cH^\lambda(t,x,y) - \int_0^t \Delta_{\varepsilon_\lambda^{-1}} \cH^\lambda(s,x,y) ds
    = \partial_t \cH^\lambda(0,x,y) -  \frac{\lambda^{2 \sigma}}{4} \int_0^t \cH^\lambda(s,x,y) ds.
\end{equation*}
Because of \eqref{eq:HFieldExample} and \eqref{eq:Ulambda}, we can calculate the time integral. Compared to \eqref{eq:Hgamma} it gives
 a lower-order term on the right-hand side \eqref{eq:strich-gamma} as $\lambda\to\infty$. (Here we use that $2\sigma\le 1$ since $s\ge 2/3$.)
Hence, we seek $\nabla \times (\varepsilon_\lambda^{-1} C^\lambda) \approx \partial_t \cH^\lambda(0,x,y)$. We set
\begin{equation*}
C_1^\lambda =0, \quad C_2^\lambda(x,y) = \varphi(x,y)g_\lambda(y)^{-1} \int_{-\infty}^x \partial_t \cH^\lambda(0,x^\prime,y) dx^\prime,
\end{equation*}
with $\varphi$ localized to $B(0,1/2)$, which contains the essential support of $\partial_t \cH^\lambda(0,x,y)$ for $\lambda \gg 1$.  
Due to \eqref{eq:HFieldExample} and \eqref{eq:Ulambda}, one can check that also the term
\begin{equation*}
r(x,y)=\nabla \times(\varepsilon_\lambda^{-1} C^\lambda) - \partial_t \cH^\lambda(0)
\end{equation*}
gives a lower-order term. Using (5) and (6) in \cite{SmithTataru2002}, we have
\begin{align}
\| \cH^\lambda \|_{L^p L^q} &\gtrsim \lambda^{-(1+\delta)/q} (\log \lambda)^{-3/q}, \notag\\
\| \cH^\lambda(t) \|_{H^\gamma(\R^2)} &\lesssim \lambda^{\gamma - (1+\delta)/2} (\log \lambda)^{2 \gamma -3/2}. \label{eq:Hgamma}
\end{align}
If we can also prove that
\begin{equation}\label{eq:DH}
\| \rho^{\lambda}_{e} \|_{H^{\gamma-\frac12-\frac{\sigma}p}}\lesssim \|\cH^\lambda \|_{H^\gamma} \quad \text{ and } \quad
\| \cD^\lambda \|_{H^\gamma} \lesssim \| \cH^\lambda \|_{H^\gamma},
\end{equation}
letting $\lambda\to\infty$
we see that the validity of \eqref{eq:strich-gamma} implies $\gamma\ge\rho+\frac{\sigma}p$. This remains
true if we  modify $1+\lambda^{2 \sigma} |y|^2$
to $\tilde{g}_\lambda(y)$ in $\varepsilon_\lambda^{-1}$, since $\cH^\lambda$ decays fast enough outside $K_\lambda$. 
(See p.~202 in \cite{SmithTataru2002}.)

To make the charge as small as possible, in view of \eqref{eq:Dlambda} we solve approximately
\begin{equation}
\label{eq:gauge}
\partial_1 C_1^\lambda + \partial_2 C_2^\lambda 
 = -(\partial_1^2 + \partial_2(1+ \lambda^{2 \sigma} y^2)^{-1} \partial_2) \psi^\lambda 
 = - \Delta_{\tilde{\varepsilon}_\lambda^{-1}} \psi^\lambda.
\end{equation}
For this purpose, we solve \eqref{eq:gauge} in $B(0,1)$ with Dirichlet boundary conditions. Let $\psi^\lambda_0$ denote the solution in $B(0,1)$. Let $\chi: [0,\infty) \rightarrow [0,1]$ be a smooth, monotone decreasing function with $\chi(r) = 1$ for $r \in [0,\frac{1}{2}]$ and $\chi(r) = 0$ for $r \geq \frac{3}{4}$. We set $\psi^\lambda(x,y) = \chi(|(x,y)|) \psi^\lambda_0(x,y)$.

We can now show \eqref{eq:DH}. We treat $\cD^\lambda$ by estimating the contributions of the terms 
in \eqref{eq:Dlambda} separately: From \eqref{eq:Ulambda} follows that the contribution of 
$\int_0^t \nabla_{\perp} \cH^\lambda ds$ is acceptable because the anti-derivative in $t$ cancels the derivative loss. The same argument bounds the contribution of $C^\lambda$, so that  $\| C^\lambda \|_{H^\gamma} \lesssim \| \cH^\lambda \|_{H^\gamma}$. By elliptic regularity, we thus obtain the bound $\| \psi^\lambda \|_{H^{\gamma+1}} \lesssim \| C^\lambda \|_{H^\gamma}\lesssim \| \cH^\lambda \|_{H^\gamma}$ since the coefficients in $\Delta_{\tilde{\varepsilon}_\lambda^{-1}}$  are uniformly bounded in $C^s$. 
Here we need the additional assumption $1 \leq s \leq 2$, which is not necessary in \cite{SmithTataru2002}.

Lastly, we compute for the charge 
\begin{equation*}
\rho_e^\lambda= \nabla \cdot \cD^\lambda= \nabla \cdot C^\lambda + \Delta_{\tilde{\varepsilon}_\lambda^{-1}} \psi^\lambda.
\end{equation*}
Hence, we find $\rho_e^\lambda = 0$ in $B(0,1/2)$ and $B(0,3/4)^c$. In $B(0,3/4) \backslash B(0,1/2)$, we find
\begin{equation*}
\rho^{\lambda}_{e} = 2 (\varepsilon_\lambda \nabla \chi) (\nabla \psi^\lambda_0) + (\partial_2 g_\lambda^{-1}) (\partial_2 \psi^\lambda_0) \chi + \text{l.o.t.}.
\end{equation*}
Since $\| \psi^\lambda \|_{H^{\gamma+1}} \lesssim \| \cH^\lambda \|_{H^\gamma}$ and $g_\lambda^{-1}\in C^s$ with $s\ge1$, 
only the second summand in the display is problematic. First we handle $\gamma - \frac{1}{2} - \frac{\sigma}{p} \le0$, in which case the estimate
\begin{equation*}
\| \rho_e^\lambda \|_{H^{\gamma - \frac{1}{2} - \frac{\sigma}{p}}} \lesssim \| \cH^\lambda \|_{H^\gamma}
\end{equation*}
is straight-forward.
For $\gamma - \frac{1}{2} - \frac{\sigma}{p} > 0$, observe that $\|g_\lambda^{-1}\|_{C^\theta(K^0_\lambda)}\lesssim \lambda^{2\sigma-(2-\theta)\delta}$. We refer to \cite{SmithTataru2002} for the true
coefficient $\tilde{g}_\lambda$ with $\| \tilde{g}^{-1}_\lambda \|_{C^\theta(B(0,1))} \lesssim \| g_\lambda^{-1} \|_{C^\theta(K^0_\lambda)}$. Suppose that $0<\gamma < \rho + \frac{\sigma}{p}$. Hence, 
$\|\partial_2 g_\lambda^{-1}\|_{C^{\gamma- \frac{1}{2} - \frac{\sigma}{p}}}\lesssim \lambda^\alpha$ with 
$\alpha= 2\sigma-(\frac32 -\gamma +\frac{\sigma}p) \delta$. Furthermore,
\begin{equation*}
\begin{split}
\| (\partial_2 g_\lambda^{-1} ) (\partial_2 \psi^\lambda_0) \chi \|_{H^{\gamma - \frac{1}{2} - \frac{\sigma}{p}}} &\lesssim \| \partial_2 g_\lambda^{-1} \|_{C^{\gamma - \frac{1}{2} - \frac{\sigma}{p}}} \| \partial_2 \psi^\lambda_0 \|_{L^2} + \| \partial_2 \psi_0^\lambda \|_{H^{\gamma - \frac{1}{2} - \frac{\sigma}{p}}} \\
&\lesssim \lambda^\alpha \lambda^{- \frac{1+\delta}{2}} (\log \lambda)^{- \frac{3}{2}} + \| \cH^\lambda \|_{H^\gamma}.
\end{split}
\end{equation*}
It remains to check
$\alpha <\gamma$, which follows from $(1-\delta)\gamma> \frac54 \sigma - \frac34$. The latter inequality holds as $\gamma > 0$,
$\delta \le \frac23$, and $\sigma \le \frac13$ because of $s\ge1$.  \hfill $\Box$

\subsection*{Acknowledgement}

Funded by the Deutsche Forschungsgemeinschaft (DFG, German Research Foundation) -- Project-ID 258734477 -- SFB 1173.

\end{document}